\documentclass{amsart}
\usepackage{amsthm,amsmath,amssymb, color, graphics,enumerate,enumitem}
\usepackage[dvipsnames]{xcolor}
\usepackage{tikz-cd}
\usepackage{mathrsfs,mathtools}
\usepackage{adjustbox}

\usepackage[pdftex]{hyperref}
\hypersetup{
    colorlinks=true, 
    linktoc=all,     
    linkcolor=blue,  
    allcolors=blue,
}

\def\colim{\qopname\relax m{colim}}
\let\oldqedbox\qedsymbol
\newcommand{\twoqedbox}{\oldqedbox\oldqedbox}

\newtheorem*{thma}{Theorem~A}
\newtheorem*{thmb}{Theorem~B}
\newtheorem*{thmc}{Theorem~C}
\newtheorem*{thmd}{Theorem~D}
\newtheorem{theorem}{Theorem}[section]
\newtheorem{lemma}[theorem]{Lemma}
\newtheorem{proposition}[theorem]{Proposition}
\newtheorem{corollary}[theorem]{Corollary}
\newtheorem{claim}[theorem]{Claim}

\newtheorem{question}[theorem]{Question}
\newtheorem{example}[theorem]{Example}

\theoremstyle{definition}
\newtheorem{definition}[theorem]{Definition}
\newtheorem{remark}[theorem]{Remark}
\newtheorem{notation}[theorem]{Notation}

\newcommand{\bb}{\mathbb}
\newcommand{\mb}{\mathbf}
\newcommand{\ra}{\rightarrow}
\newcommand{\otp}{\mathrm{otp}}
\newcommand{\sh}{\mathrm{sh}}
\newcommand{\supp}{\mathrm{supp}}
\newcommand{\cf}{\mathrm{cf}}
\newcommand{\ssup}{\mathrm{ssup}}
\newcommand{\CondAb}{\mathsf{Cond(Ab)}}

\author[Bergfalk]{Jeffrey Bergfalk}
\address{Departament de Matem\`{a}tiques i Inform\`{a}tica \\
Universitat de Barcelona \\
Gran Via de les Corts Catalanes 585 \\ 08007 Barcelona, Catalonia}
\email{bergfalk@ub.edu}
\urladdr{https://www.jeffreybergfalk.com/}

\author[Lambie-Hanson]{Chris Lambie-Hanson}
\address{Institute of Mathematics of the Czech Academy of Sciences \\ 
	\v{Z}itn\'{a} 25, 110 00 Praha 1, Czechia}
\email{lambiehanson@math.cas.cz}
\urladdr{https://users.math.cas.cz/~lambiehanson/}

\title[Infinitary combinatorics in condensed math and strong homology]{Infinitary combinatorics \\ in condensed mathematics and strong homology}


\thanks{The first author was supported by Marie Sk\l odowska Curie (project 101110452) and Ram\'{o}n y Cajal fellowships; the second author was supported by GA\v{C}R project 23-04683S and the Academy of Sciences of the Czech Republic (RVO 67985840).}

\begin{document}

\begin{abstract}
Recent advances in our understanding of higher derived limits carry multiple implications in the fields of condensed and pyknotic mathematics, as well as for the study of strong homology. These implications are thematically diverse, pertaining, for example, to the sheaf theory of extremally disconnected spaces, to Banach--Smith duality, to the productivity of compact projective condensed anima, and to the structure of the derived category of condensed abelian groups. Underlying each of these implications are the combinatorics of multidimensionally coherent families of functions of small infinite cardinal height, and it is for this reason that we convene accounts of them together herein.
\end{abstract}

\maketitle

\section{Introduction}
The aim of this article is to describe the role of a rich but rather specific area of infinitary combinatorics, namely that of higher-dimensional coherence phenomena, 
in the answers to an array of superficially unrelated questions.
A main background to this work was Dustin Clausen and Peter Scholze's 2019 recognition that generalizations of the derived limit computations of \cite{Bergfalk_simultaneously_21}
would carry structural implications within their emergent framework of \emph{condensed mathematics}.
Not long thereafter, the present paper's authors realized that these generalizations afforded a more cohesive approach than hitherto to the well-studied questions of the additivity both of strong homology and of the derived limit functors themselves.
These generalizations center on a family of groups $\mathrm{lim}^n\,\mathbf{A}_{\kappa,\lambda}$ parametrizing nontrivial $n$-dimensional coherence phenomena on index systems of cardinal width $\kappa$ and height $\lambda$, and the interest of these phenomena derives in no small part from their relations to the cardinals $\aleph_n$.
Scholze in 2021 converted a further question in condensed mathematics to one, in essence, about such phenomena, and the answer to this question, as we will see, leverages precisely this relation.
More recently, Clausen and Scholze observed that a strong variant of the main result of \cite{Bergfalk_simultaneously_21} would carry pleasing implications for the Banach--Smith duality so prominent within condensed functional analysis, implications we describe herein as well.

More precisely, this paper records proofs of the theorems A through D listed below, showing in the 
process that they all, combinatorially speaking, drink from the same well.
We should underscore without delay that we neither presume nor require expertise in any of the several mathematical subfields which these theorems involve; that these results amount ultimately to little more than ordinal combinatorics is, after all, much of our point.
How these combinatorics arise is, on the other hand, much of the story as well, and we have accordingly included such basic accounts and definitions as should render these conversions legible.
That said, fuller comprehension of the often inescapably technical machinery of these conversions \emph{will} depend to some degree on readers' backgrounds; to better accommodate their range, we have shaped our text to permit multiple reading paths through its material, as we describe in greater detail below.

Our first theorem addresses the question, communicated by Clausen and 
Scholze, of whether the natural functor from the pro-category of derived abelian groups 
to the derived category of condensed abelian groups is consistently fully faithful.
For a fuller discussion and motivation of this question, see Section \ref{subsect:pro-to-D}; in Section \ref{subsect:systemsAkl} we show the following.

\begin{thma}
  The natural functor sending $\mathsf{Pro}(\mathsf{D}\mathsf{(Ab)})$ to 
  $\mathsf{D}\mathsf{(Cond(Ab))}$ is not full. In particular, the left hand side of the equation
  \[
    \mathrm{RHom}_{\mathsf{Pro}(\mathsf{D}^{\geq 0}(\mathsf{Ab}))}\left(\text{``}\prod_{\omega}\text{''}\,\bigoplus_{\omega_1} \bb{Z},\bb{Z}\right)\not\cong\mathrm{RHom}_{\mathsf{D}^{\geq 0}\mathsf{(Cond(Ab))}}\left(\prod_{\omega} \bigoplus_{\omega_1} \bb{Z}, \bb{Z}\right)
  \]
is concentrated in degree zero, but the right hand side is not.
\end{thma}

If, however, one restricts to \emph{countable} abelian groups and operations upon them, then the main pathologies (so to speak) of Theorem A are avoidable.
We show the following in Section \ref{subsect:alternative}; the $\mb{A}[H]$ appearing in its premise is the natural generalization of the inverse system $\mb{A}=\mb{A}[\bb{Z}]$ tracing back, via a substantial set theoretic literature, to \cite{Mardesic_additive_88}.

\begin{thmb}
  The following statements are all consistent; in fact each holds under the assumption 
  that $\lim^n \mb{A}[H] = 0$ for all $n > 0$ and all abelian groups $H$.
  \begin{enumerate}
    \item $\mathrm{RHom}_{\mathsf{D}^{\geq 0}\mathsf{(Cond(Ab))}}\left(\prod_{\omega} \bigoplus_{\omega} \bb{Z}, \bigoplus_\mu \bb{Z}\right)$ is concentrated in degree zero 
    for all cardinals $\mu$.
    \item Whenever $H$ is an abelian group and $\mb{M} = (M_i, \pi_{i,j}, \omega)$ is an inverse sequence of countable abelian groups 
    whose transition maps $\pi_{i,j} : M_{j} \to M_i$ are all surjective, then
$$\mathrm{Ext}_{\mathsf{Cond(Ab)}}^n\left(\lim_{i < \omega}\,M_i,H\right)=\colim_{i < \omega}\,\mathrm{Ext}_{\mathsf{Cond(Ab)}}^n\left(M_i,H\right)$$
for all $n \geq 0$.
    \item Whenever $p$ is a prime and $X$ is a separable solid $\bb{Q}_p$-Banach space, 
    \[
      \underline{\mathrm{Ext}}^i_{\mathsf{Solid}_{\bb{Q}_p}}(X, \bb{Q}_p) = 0
    \] 
    for all $i > 0$.
  \end{enumerate}
\end{thmb}
The first and the second of the theorem's conclusions are of additional interest (as we explain just after its proof) for each implying that the continuum is of cardinality at least $\aleph_{\omega+1}$.
Its clause (3), on the other hand, implies that the classical stereotype duality between 
Banach spaces and Smith spaces consistently extends to the derived category of 
solid $\bb{Q}_p$-vector spaces when restricted to separable 
solid $\bb{Q}_p$-Banach spaces.

Let us underscore before proceeding any further that the forms of Theorems A, B, and D, as well as their conversions to questions about derived limits, are all essentially due to Clausen and Scholze; in each case, our own contributions to the theorems consist principally in the derived limit analyses (exemplified by our Theorems \ref{thm:limsofAkappalambda} and \ref{thm:nonzero_cohomology_on_opens}) which complete their proofs.


The background to Theorem C, in contrast, is our own joint work on the additivity problem for the strong homology functor $\overline{\mathrm{H}}_\bullet$; see Section \ref{sec:additivity} for a brief discussion of this problem's history.
The main immediate point is that the computations underlying Theorem A readily furnish the $\mathsf{ZFC}$ counterexample to the additivity of $\overline{\mathrm{H}}_\bullet$ cited below, along with closely related $\mathsf{ZFC}$ counterexamples to the additivity of the functors $\lim^n : \mathsf{Pro(Ab)}\to\mathsf{Ab}$ for $n = 1$ and $n=2$.
None of these is the first such counterexample to have been recorded; that distinction belongs to those appearing in \cite{Prasolov_non_05}.
In both their unity and comparative simplicity, however, those recorded herein carry substantial advantages over those of \cite{Prasolov_non_05} (the generalized Hawaiian earring $Y^n$ below, for example, is compact, in contrast to its counterpart in \cite{Prasolov_non_05}), and together render the additivity problem's combinatorial content significantly plainer.


\begin{thmc}
  Fix a natural number $n > 0$, let $B^n$ denote the $n$-dimensional open ball, and 
  let $Y^n$ denote the one-point compactification of $\coprod_{\omega_1} B^n$. Then
  \[
    \overline{\mathrm{H}}_{n-1}\left(\coprod_{\omega} Y^n \right) \not\cong 
    \bigoplus_{\omega} \overline{\mathrm{H}}_{n-1}\left(Y^n\right).
  \]
\end{thmc}


For this paper's final results, we return to the setting of condensed mathematics.
From both conceptual and computational points of view, one of condensed categories' most fundamental virtues  is their possession of generating classes of
compact projective objects.
Both annoying, accordingly, and insufficiently well understood are the failures of these classes to be closed under binary products.
For example, the tensor product $X \otimes Y$ of two compact projective objects in the category $\mathsf{Cond}(\mathsf{Ab})$ of condensed abelian groups need not be projective; see the introduction to Section \ref{sect:products} for further discussion. 
More subtle is the still-open question of whether or when such $X \otimes Y$ have finite projective dimension in $\mathsf{Cond}(\mathsf{Ab})$.
Here we provide a negative answer to a stronger conjecture by proving that 
products of compact projective condensed anima are not, in general, compact, and we do so 
via an auxiliary result about classical sheaf categories that we feel is of interest 
in its own right.
In what follows, $\mathsf{ED}$ denotes the category of extremally 
disconnected compact Hausdorff spaces.

\begin{thmd}
For every finite field $K$ and $S \in \mathsf{ED}$, the constant sheaf $\mathcal{K}$ on $S$ is injective. If $S,T \in \mathsf{ED}$ are each \v{C}ech-Stone compactifications of discrete sets of cardinality at least $\aleph_\omega$, then for any field $K$ the injective dimension of the constant sheaf $\mathcal{K}$ on $S \times T$ is infinite.

In consequence, there exist compact projective condensed anima $S$ and $T$ whose product is not compact.
\end{thmd}


Let us turn now in greater detail to the paper's organization and the varying demands that it places upon its readers. 
Section \ref{sec:pro} divides into five subsections; of these, only \ref{subsect:condensedbackground}, which provides a brief review of derived and condensed settings,
and \ref{subsect:systemsAkl}, recording the fundamentals of the groups $\mathrm{lim}^n\,\mathbf{A}_{\kappa,\lambda}$, play a significant role in subsequent sections. Subsections \ref{subsect:pro-to-D} and the more down-to-earth \ref{subsect:alternative}, containing 
main portions of the arguments of Theorems A and B, respectively, provide motivation and context for the computations of \ref{subsect:systemsAkl}, while \ref{subsect:morenonvanishing} records a more set theoretically involved refinement of the latter.
Put differently, in Subsections \ref{subsect:condensedbackground}, \ref{subsect:systemsAkl}, and \ref{subsect:alternative}, we presume only a basic familiarity with homological algebra (\cite[Ch.\ 1--3]{Weibel}, say) and set theory,
 while in Subsections \ref{subsect:pro-to-D} and \ref{subsect:morenonvanishing} (both of which may be skimmed in a first reading), we assume a little more familiarity, respectively, with each.

Sections \ref{sec:additivity} and \ref{sect:products} may be read independently of one another; the first of these divides into two subsections, neither of which presumes any further background of the reader.
Its first subsection explores the implications of the results of Subsection \ref{subsect:systemsAkl} along the lines discussed before Theorem C above, which it also proves. A more optional second subsection contextualizes these implications within a broader discussion of strong homology.

Next, after a brief introduction, the two subsections of Section \ref{sect:products} contain 
our proof of Theorem D, with the first containing the proofs of its first paragraph and the second containing that of its second. These portions of the paper require a bit more of the reader.
Subsection \ref{subsect:sheaves} involves sheaf theory at roughly the level of \cite[Ch.\ 2]{Iversen}.
Subsection \ref{subsect:products}, like any discussion of condensed anima, then necessitates an at least occasionally $\infty$-category theoretic vocabulary; relevant references are provided along the way.
The paper then concludes with a list of open questions; among these, those amounting essentially to problems in infinitary combinatorics predominate.

We hope by this outline to have suggested the kinds of math which our text engages; its more fundamental roots in the more contemporary field of condensed mathematics, though, merit a few further words of comment.
Core references for condensed mathematics are the Scholze and Clausen--Scholze lecture notes \cite{CS1, CS2, CS3} and the online Clausen--Scholze lecture series \cite{Masterclass} and \cite{analytic_stacks}; multiple theses \cite{Aparicio_condensed_21, Asgeirsson, Mair} and course notes \cite{UChic,Columb,JHop} usefully supplement this material. In so rapidly expanding a field, however, reference lists take on a markedly provisional character. Contemporaneous with the development of condensed mathematics has been that of Barwick and Haine's closely related \emph{pyknotic mathematics}; among the fundamental references here are \cite{pyknotic} and the MSRI lecture series \cite{MSRI}, although our comments above on references proliferating faster than we can confidently track them continue to apply.
Below, we have uniformly framed our results in condensed terms, both for simplicity and for the essentially contingent reason that our first encounter with this world of ideas was by way of an email from Dustin Clausen and Peter Scholze; these results' translations to the pyknotic framework are, in general, straightforward.
For that initial contact and for much rich and generous conversation thereafter, we wish to underscore our thanks to Clausen and Scholze.
For further discussion and clarifications of the material of Sections \ref{subsect:pro-to-D} and \ref{subsect:products}, we are grateful to Lucas Mann as well.

Lastly, a word on some of our most fundamental notations. We write $P(X)$ for the power set, and $|X|$ for the cardinality, of a set $X$. If $X$ is partially ordered then we write $\mathrm{cf}(X)$ for the least cardinality of a cofinal subset of $X$, while if $\kappa$ is a cardinal then $\mathrm{Cof}(\kappa)$ denotes the class of ordinals $\xi$ for which $\mathrm{cf}(\xi)=\kappa$. If 
$\beta$ is an ordinal, then $\lim(\beta)$ will denote the set of all limit ordinals less than $\beta$. We note that we also use $\lim$ to denote the inverse limit functor; the correct interpretation will always be clear from context. If $X$ is well-ordered then $\mathrm{otp}(X)$ denotes the unique ordinal which is order-isomorphic to $X$. We write ${^\kappa}X$ for the collection of functions from $\kappa$ to $X$, although we will at times also exponentiate on the right. Relatedly, $[X]^n$ and $[X]^{<\omega}$ denote, respectively, the collections of size-$n$ subsets, and of finite subsets, of $X$; when $X$ is linearly ordered, we will identify these with ordered tuples in the natural way. If $\mathcal{C}$ is a category, then we write $X\in\mathcal{C}$ to indicate that $X$ is an object of $\mathcal{C}$.

\section{The category $\mathsf{D}^{\geq 0}\mathsf{(Cond(Ab))}$ and the systems $\mathbf{A}_{\kappa,\lambda}$}
\label{sec:pro}
\subsection{A brief survey of the relevant categories}
\label{subsect:condensedbackground}

Let us begin by recalling the notion of an \emph{abelian category}: a category $\mathcal{A}$ is abelian if it possesses
\begin{enumerate}
\item a zero object (i.e., an object which is both initial and terminal),
\item binary biproducts (i.e., coinciding binary products and sums), and
\item kernels and cokernels, such that
\item all of $\mathcal{A}$'s monomorphisms are kernels, and all of its epimorphisms are cokernels.
\end{enumerate}
The paradigmatic example, of course, is the category $\mathsf{Ab}$ of abelian groups, and we may view conditions (1) -- (4) as abstracting from $\mathsf{Ab}$ the categorical equipment for basic homological algebra.
More general examples include the category $\mathsf{Mod}_R$ of right $R$-modules for a fixed ring $R$, or the category of presheaves of abelian groups over a fixed space $X$ or category $\mathcal{D}$, with inverse systems of abelian groups over a fixed partial order $I$ a special case of the latter.
Letting $I$ range over the class of directed sets determines the objects --- frequently denoted
\[\textnormal{``}\lim_{i\in I}\!\textnormal{''}\,A_i\] --- of the \emph{pro-category} $\mathsf{Pro}(\mathsf{Ab})$ of abelian groups, whose morphisms are formally described by the formula
\begin{align}
\label{eq:limcolim}
\mathrm{Hom}_{\mathsf{Pro}(\mathsf{Ab})}\left(\textnormal{``}\lim_{i\in I}\!\textnormal{''}\,A_i,\textnormal{``}\lim_{j\in J}\!\textnormal{''}\,B_j\right)=\lim_{j\in J}\colim_{i\in I}\mathrm{Hom}_{\mathsf{Ab}}\left(A_i,B_j\right)
\end{align}
\noindent More generally and intuitively, $\mathsf{Pro}(\mathcal{C})$ freely adjoins cofiltered limits to its base category $\mathcal{C}$, as its notation suggests, and it, too, is abelian, if $\mathcal{C}$ itself is.
Moreover, if $\mathcal{C}$ possesses cofiltered limits then the inverse limit functor
\[\textnormal{``}\lim_{i\in I}\!\textnormal{''}\,C_i\,\mapsto\,\lim_{i\in I} C_i\]
is well-defined on $\mathsf{Pro}(\mathcal{C})$.
In particular (as a closer inspection of the morphisms of $\mathsf{Pro}(\mathcal{C})$ then implies), $\lim_{i\in J} C_i\cong\lim_{i\in I} C_i$ whenever $J$ is a cofinal subset of $I$. For more on the subject of pro-categories (often via their dual notion $\mathsf{Ind}(\mathcal{C}^{\mathrm{op}})=(\mathsf{Pro}(\mathcal{C}))^{\mathrm{op}}$, which plays, in turn, a significant conceptual role in Section \ref{sect:products} below), see \cite[I.8]{SGA4}, \cite[Appendix]{ArtinMazur}, \cite[VI]{Johnstone}, \cite[I.1]{Mardesic_strong_00}, or \cite[\S 6]{Kashiwara_Categories_06}.

By the phrase \emph{basic homological algebra} above, we largely mean the following sequence: any abelian category $\mathcal{A}$ gives rise to the abelian category $\mathsf{Ch}^{\star}(\mathcal{A})$ of cochain complexes in $\mathcal{A}$ (here the ${^\star}$ may be read to denote any of the standard boundedness conditions, or none at all); identifying chain homotopic chain maps $f^\bullet,g^\bullet:A^\bullet\to B^\bullet$ determines in turn the \emph{homotopy category} $\mathsf{K}^{\star}(\mathcal{A})$ of $\mathcal{A}$, and localizing $\mathsf{K}^{\star}(\mathcal{A})$ (or, equivalently, $\mathsf{Ch}^{\star}(\mathcal{A})$) with respect to the class of \emph{quasi-isomorphisms} $f^\bullet:A^\bullet\to B^\bullet$ --- i.e., of maps $f^\bullet$ inducing isomorphisms $f^n_*:\mathrm{H}^n(A^\bullet)\to\mathrm{H}^n(B^\bullet)$ of all the cohomology groups of $A^\bullet$, $B^\bullet$ --- determines the \emph{derived category} $\mathsf{D}^{\star}(\mathcal{A})$ of $\mathcal{A}$.
The latter two operations exhibit strong analogies with identifications of topological spaces up to \emph{homotopy equivalence} and \emph{weak homotopy equivalence}, respectively; as in both of those cases, they render legible relations otherwise invisible among the objects of the source category.
Observe next that any additive functor $F:\mathcal{A}\to\mathcal{B}$ between abelian categories naturally induces functors $\mathsf{Ch}(F):\mathsf{Ch}^{\star}(\mathcal{A})\to\mathsf{Ch}^{\star}(\mathcal{B})$ and $\mathsf{K}(F):\mathsf{K}^{\star}(\mathcal{A})\to\mathsf{K}^{\star}(\mathcal{B})$, but that there is, in general, no $G:\mathsf{D}^{\star}(\mathcal{A})\to\mathsf{D}^{\star}(\mathcal{B})$ making the following square commute:

\begin{center}

\begin{tikzcd}
\mathsf{K}^{\star}(\mathcal{A}) \arrow[d, "\mathsf{K}(F)"'] \arrow[r, "\ell_{\mathcal{A}}"] & \mathsf{D}^{\star}(\mathcal{A}) \arrow[d, "G", dashed] \\
\mathsf{K}^{\star}(\mathcal{B}) \arrow[r, "\ell_{\mathcal{B}}"']                            & \mathsf{D}^{\star}(\mathcal{B})                       
\end{tikzcd}
\end{center}

\noindent What may exist in its stead is a $G$ so that $G\circ\ell_{\mathcal{A}}$ optimally approximates $\ell_{\mathcal{B}}\circ K(F)$ (in the sense of forming a terminal or initial object in the appropriate category of functors $\mathsf{K}^{\star}(\mathcal{A})\to\mathsf{D}^{\star}(\mathcal{B})$ over or under $\ell_{\mathcal{B}}\circ K(F)$, respectively); such a $G$ is then the (left or right) \emph{derived functor} ($\mathrm{L}F$ or $\mathrm{R}F$, respectively) of $F$.
Functors of particular interest in what follows will be $\mathrm{Hom}:\mathcal{A}^{\mathrm{op}}\times\mathcal{A}\to\mathsf{Ab}$ and $\mathrm{lim}:\mathsf{Pro}(\mathcal{A})\to\mathcal{A}$ for abelian categories $\mathcal{A}$; the cohomology groups of their derived functors $\mathrm{RHom}(A,B)$ and $\mathrm{Rlim}\,\mathbf{X}$ define the classical expressions $\mathrm{Ext}^n(A,B)$ and $\mathrm{lim}^n\,\mathbf{X}$, respectively, for $n>0$.

A \emph{non-example} of an abelian category is the category $\mathsf{TopAb}$ of topological abelian groups: writing $\mathbb{R}^{\mathrm{disc}}$ and $\mathbb{R}$ for the group $\mathbb{R}$ endowed with the discrete and standard topologies, respectively, the map
\begin{align}
\label{eq:toyexample}\mathrm{id}:\mathbb{R}^{\mathrm{disc}}\to\mathbb{R}
\end{align}
witnesses the failure of condition (4) above.
The choice of $\mathbb{R}$ here is, of course, essentially arbitrary; the same issue arises for any fixed algebraic object admitting multiple topologies, some finer than others. These examples epitomize a tension between topological data and any practice of algebra partaking of the constructions of the previous paragraph; they are standard within the condensed literature and cue one of its main motivating questions: \emph{How to do algebra when rings/modules/groups carry a topology?} \cite[p.\ 6]{CS1}.
We turn now, with this as background, to a brief introductory overview of the condensed framework. This subsection's remainder is entirely drawn from \cite[I--II]{CS1}.

Write $\mathsf{CHaus}$ to denote the category of compact Hausdorff spaces and $\mathsf{ProFin}$ to denote its full subcategory of Stone spaces, i.e., of totally disconnected compact Hausdorff spaces; the notation derives from this subcategory's equivalence with the pro-category of finite sets (to see this, view the objects of the latter as inverse systems of discrete topological spaces and take their limits). Write $\mathsf{ED}\subset\mathsf{ProFin}$ for the category of extremally disconnected compact Hausdorff spaces, i.e., for the full subcategory of $\mathsf{CHaus}$ spanned by its projective objects (see \cite{Gleason}).
Letting $\kappa$ range through the cardinals filters these large categories by their essentially small subcategories $\mathsf{ED}_\kappa\subseteq\mathsf{ProFin}_\kappa$ spanned by their objects of topological weight less than $\kappa$.
The finite jointly surjective covers determine a Grothendieck topology $\tau$ on $\mathsf{ProFin}$, and we write $\tau$ also for its restriction to any subcategory of $\mathsf{ProFin}$.

\begin{definition}
\label{def:condset}
Fix an infinite cardinal $\kappa$ and let $\mathcal{C}=\mathsf{ProFin}_\kappa$; a \emph{$\kappa$-condensed set} $F$ is then a $\mathsf{Set}$-valued sheaf on the site $(\mathcal{C},\tau)$.
More concretely, $F\in\mathrm{Fun}(\mathcal{C}^{\mathrm{op}},\mathsf{Set})$ is $\kappa$-condensed iff it satisfies the following conditions:
\begin{enumerate}
\item $F(\varnothing)=*$ (i.e., $F(\varnothing)$ is a singleton, a terminal object of $\mathsf{Set}$);
\item the natural map $F(S\sqcup T)\to F(S)\times F(T)$ is a bijection for any $S,T\in\mathcal{C}$;
\item for any surjection $T\twoheadrightarrow S$ in $\mathcal{C}$ with induced maps $p_1,p_2$ from the fiber product $T\times_S T$ to $T$, the natural map $F(S)\to\{x\in F(T)\mid p_1^*(x)=p_2^*(x)\}$ is a bijection.
\end{enumerate}
Such $F$ form the objects of a category $\mathsf{Cond}(\mathsf{Set})_{\kappa}$, with morphisms the natural transformations between them.
For strong limit $\kappa\leq\lambda$ we then have natural embeddings $\mathsf{Cond}(\mathsf{Set})_{\kappa}\to\mathsf{Cond}(\mathsf{Set})_{\lambda}$, and the category $\mathsf{Cond}(\mathsf{Set})$ is the direct limit of this system of embeddings as $\kappa$ ranges over the class of strong limit cardinals.
\end{definition}
A frequently useful observation is the following: by coinitiality, its restriction to extremally disconnected sets forms a basis for the topology $\tau$, hence condensed sets may be equivalently defined by taking $\mathcal{C}=\mathsf{ED}_\kappa$ in Definition \ref{def:condset} above.
Over this more restricted category, the sheaf condition reduces to the first two conditions of Definition \ref{def:condset} (i.e., the third may be safely ignored), and this fact carries a number of pleasing consequences within the theory. The fact, on the other hand, that the category $\mathsf{ED}$ doesn't in general possess products is the source of some unpleasantness, and this deficiency forms the main concern of Section \ref{sect:products} below.

Let $X$ be a topological space which is T1, meaning that any point of $X$ forms a closed set.
Then the restricted Yoneda functor
$$\underline{X}:\mathsf{ProFin}^{\mathrm{op}}\to\mathsf{Set}:S\mapsto\mathrm{Cont}(S,X),$$
where $\mathrm{Cont}(S,X)$ denotes the set of continuous maps from $S$ to $X$,
is a condensed set, as the reader is encouraged to verify.
In fact, the maps $X\mapsto\underline{X}$ define a fully faithful embedding of the category of compactly generated weakly Hausdorff topological spaces into the category $\mathsf{Cond}(\mathsf{Set})$; this is the first fundamental fact about the theory (\cite[Prop.\ 1.7]{CS1}).
For the second, define the category $\mathsf{Cond}(\mathsf{Ab})$ either by replacing $\mathsf{Set}$ with $\mathsf{Ab}$ in Definition \ref{def:condset} or by restricting to the abelian group objects of $\mathsf{Cond}(\mathsf{Set})$. We then have \cite[Thm.\ 1.10]{CS1}:
\begin{theorem}
$\mathsf{Cond}(\mathsf{Ab})$ is a complete and cocomplete abelian category.
\end{theorem}
In particular, it is an abelian category into which the category, for example, of locally compact abelian groups embeds: within it, the latter's deficiencies epitomized by line \ref{eq:toyexample} above find resolution. And it is a category to which the derived machinery described above readily applies (in particular, it possesses enough projectives; see Section 4); with this observation, we arrive to the following subsection.
\subsection{The functor $\mathsf{Pro}(\mathsf{D}(\mathsf{Ab}))^{\mathtt{b}}\to\mathsf{D}(\mathsf{Cond}(\mathsf{Ab}))$}
\label{subsect:pro-to-D}
It is plain from the preceding discussion that the category of abelian groups fully faithfully embeds into $\mathsf{Cond}(\mathsf{Ab})$; so too, by way of the category of locally compact abelian groups, does the category of \emph{profinite} abelian groups.
Whether the category $\mathsf{Pro}(\mathsf{Ab})$ does as well is a natural next question; its most encompassing formulation is at the derived level, as follows:

\begin{question}
\label{ques:pro-D}
Is the natural map $\mathsf{Pro}(\mathsf{D}\mathsf{(Ab)})^{\mathtt{b}}\to \mathsf{D}\mathsf{(Cond(Ab))}$ given by
\begin{align}
\label{eq:Q1}
\mathrm{``lim}\textnormal{''}\,C_i\mapsto\mathrm{Rlim}\,\underline{C}_i\end{align}
fully faithful?
\end{question}
A few remarks are immediately in order. First: as, in general, it's only at the $\infty$-category theoretic level that the $\mathsf{Pro}( - )$ operation preserves derived structure (cf.\ \cite[\S 15.4]{Kashiwara_Categories_06}), it's at this level that Question \ref{ques:pro-D} should ultimately be posed.
We won't belabor this point, since it only momentarily affects our reformulation of the question below; more generally, as our introduction suggested, a relaxed or ``naive'' reading relationship (cf.\ \cite[Ch.\ I.1]{Cisinskietal}; the analogy, of course, is with naive set theory) to the $\infty$-category theoretic manipulations of both this subsection and Section \ref{subsect:products} is broadly compatible with this paper's main aims.
This is fortunate, since any rigorous review of the subject of $\infty$-categories is beyond the scope of this paper.
Lurie's \cite{LurieHTT} or its more recent online reworking \cite{kerodon} are standard one-stop references; for the subject of $\infty$-derived categories, see \cite[\S 1.3]{HA}.
For a highly economical review of much of that material, see Appendix A of \cite{Mann}, which we cite more particularly in Section \ref{sect:products} below.
In any case, a main point in the present context is that, when interpreted at the $\infty$-derived level, the $\mathrm{Rlim}$ of (\ref{eq:Q1}) is simply a limit; thus at this level, the functor (\ref{eq:Q1}) is plainly limit-preserving, and Question \ref{ques:pro-D} is, modulo size issues, one of whether $\mathsf{Pro}(\mathsf{D}\mathsf{(Ab)})^{\mathtt{b}}$ is a reflective subcategory of $\mathsf{D}\mathsf{(Cond(Ab))}$ \cite[\S 5.2.7]{LurieHTT}.\footnote{The issues in question are that the constituent categories aren't presentable, so that the  standard lemma \cite[Cor.\ 5.5.2.9]{LurieHTT} completing the characterization doesn't apply.}
We will expand on the question's significance in this subsection's conclusion below.

Second: we should clarify what it is that we mean by the superscript $\mathtt{b}$.
We begin by noting that the variant of Question \ref{ques:pro-D} given by omitting $\mathtt{b}$ admits an easy negative answer; thanks to Lucas Mann for the following example. For any complex $E$,
$$\mathrm{Hom}_{\mathsf{Pro(D(Ab))}}\left(``\prod_{\mathbb{N}}\textnormal{''}\,\mathbb{Z}[-n],E\right)=\bigoplus_\mathbb{N}\mathrm{Hom}_{\mathsf{D(Ab)}}(\mathbb{Z}[-n],E);$$
see \cite[2.6.2]{Kashiwara_Categories_06}.
Here $\mathbb{Z}[-n]$ denotes the cochain complex valued $\mathbb{Z}$ in the degree $n$ and $0$ elsewhere and the natural map $\bigoplus_{\mathbb{N}}\mathbb{Z}[-n]\to\prod_{\mathbb{N}}\mathbb{Z}[-n]$ is easily seen to be a quasi-isomorphism, with the consequence that
\begin{align*}
\mathrm{Hom}_{\mathsf{D(Cond(Ab))}}\left(\prod_{\mathbb{N}}\underline{\mathbb{Z}[-n]},\underline{E}\right) & =\mathrm{Hom}_{\mathsf{D(Cond(Ab))}}\left(\bigoplus_{\mathbb{N}}\underline{\mathbb{Z}[-n]},\underline{E}\right) \\ & =\prod_{\mathbb{N}}\mathrm{Hom}_{\mathsf{D(Cond(Ab))}}\left(\underline{\mathbb{Z}[-n]},\underline{E}\right).
\end{align*}
In each of the above cases the concluding $\mathrm{Hom}$ terms are $\mathrm{H}^n(E)$, hence if the cohomology of $E$ is nonvanishing in unboundedly many degrees, the two leftmost terms above are non-isomorphic, and this instance of the functor (\ref{eq:Q1}) isn't full.
Restricting attention to the full subcategory of $\mathsf{Pro(D(Ab))}$ spanned by inverse systems of uniformly bounded complexes, though, averts this issue, and it is this category which we have denoted above by $\mathsf{Pro(D(Ab))}^{\mathtt{b}}$.
So posed, Question \ref{ques:pro-D} reduces to a more computationally concrete and interesting question, by steps which we now describe.

First, note that one may without loss of generality restrict to complexes bounded from below by $0$, and hence to the derived category $\mathsf{D}^{\geq 0}$; observe next that by equation \ref{eq:limcolim}, our question is then that of whether for all inverse systems of uniformly bounded cochain complexes $(C_i)_{i\in I}$ and $(E_j)_{j\in J}$ in $\mathsf{D}^{\geq 0}(\mathsf{Ab})$ the natural map
$$\lim_{j\in J}\colim_{i\in I}\mathrm{Hom}_{\mathsf{Pro}(\mathsf{D}^{\geq 0}(\mathsf{Ab}))}(C_i,E_j)\to \mathrm{Hom}_{\mathsf{D}^{\geq 0}(\mathsf{Cond(Ab}))}\left(\lim_{i\in I} \underline{C}_i,\lim_{j\in J} \underline{E}_j\right)$$
is an isomorphism.
Again all operations are to be read as $\infty$-categorical. Factoring out the rightmost $\mathrm{lim}$ reduces our question to that of whether
$$\colim_{i\in I}\mathrm{Hom}_{\mathsf{Pro}(\mathsf{D}^{\geq 0}(\mathsf{Ab}))}(C_i,E)\to \mathrm{Hom}_{\mathsf{D}^{\geq 0}(\mathsf{Cond(Ab}))}\left(\lim_{i\in I} \underline{C}_i,\underline{E}\right)$$
is an isomorphism for every bounded $E$ in $\mathsf{D}^{\geq 0}(\mathsf{Ab})$ (note that by the example above, this is in general false if $E$ is unbounded).
By passing the $\colim$ back across the parentheses and recalling that $\mathrm{RHom}$ commutes with all finite limits and colimits, we may by \cite[4.4.2.7]{LurieHTT} further reduce our question to that of whether
\begin{align*}\mathrm{RHom}_{\mathsf{Pro}(\mathsf{D}^{\geq 0}(\mathsf{Ab}))}\left(\text{``}\prod_{i\in I}\text{''}\,C_i,E\right)=\mathrm{RHom}_{\mathsf{D}^{\geq 0}\mathsf{(Cond(Ab))}}\left(\prod_{i\in I} \underline{C}_i, \underline{E}\right)\end{align*}
for all $(C_i)_{i\in I}$ and $E$ as above.
Standard d\'{e}vissage arguments (cf.\ \cite[Prop.\ 15.4.2]{Kashiwara_Categories_06}) using the truncation functors $\tau^{\geq n}$ and fiber sequences $\tau^{<n}E\to E\to \tau^{\geq n}E$ allow us to further reduce to the case in which $E$ is concentrated in a single degree; similarly for the first coordinate.
Working with $\mathrm{RHom}$ allows us to further assume, without loss of generality, that that degree in both cases is zero.
We have thus reduced our question to the following.

\emph{For abelian groups $G_i$ $(i\in I)$ and $H$, does the following equality hold:}
\begin{align}\label{eq:Q1'}\mathrm{RHom}_{\mathsf{Pro}(\mathsf{D}^{\geq 0}(\mathsf{Ab}))}\left(\text{``}\prod_{i\in I}\text{''}\,G_i,H\right)=\mathrm{RHom}_{\mathsf{D}^{\geq 0}\mathsf{(Cond(Ab))}}\left(\prod_{i\in I} \underline{G}_i, \underline{H}\right)?\end{align}
This question further simplifies as follows. Observe first that free resolutions of arbitrary groups induce long exact sequences reducing the question to one for free groups (see the proof of Theorem \ref{thm:fromstacks} for more concrete argumentation along these lines); we may therefore assume above that each $G_i=\bigoplus_J\mathbb{Z}$ and $H=\bigoplus_K\mathbb{Z}$ for some sets $J$ and $K$.
In this case, the left-hand side is concentrated in the degree zero, for the reason that
\begin{align*}
\mathrm{RHom}_{\mathsf{Pro}(\mathsf{D}^{\geq 0}(\mathsf{Ab}))}\left(\text{``}\prod_{i\in I}\text{''}\,G_i,H\right)=\bigoplus_{i\in I}\mathrm{RHom}_{\mathsf{D}^{\geq 0}(\mathsf{Ab})}\left(G_i,H\right),
\end{align*}
together with the fact that each $G_i$ is free.
Moreover, the right-hand term readily identifies in the zeroth degree with $\bigoplus_I\prod_J H$.
It is then easy to see that this, in turn, equals the zeroth degree of the right-hand side of equation \ref{eq:Q1'}, namely $\mathrm{lim}\,\mathbf{A}_{I,J}[H]$, in the language of a reformulation which we now describe.

Recall that $[J]^{<\omega}$ denotes the collection of finite subsets of $J$.
For any function $f:I\to [J]^{<\omega}$, let $X(f)=\{(i,j)\in I\times J\mid j\in f(i)\}$, and let ${^I}([J]^{<\omega})$ denote the collection of all functions from $I$ to $[J]^{<\omega}$, ordered by setting $f\leq g$ if and only if $X(f)\subseteq X(g)$. Now observe that
$$\prod_{i\in I} G_i=\underset{f\in {^I}([J]^{<\omega})}{\mathrm{colim}} \prod_{i\in I} \mathbb{Z}^{f(i)}\,.$$
By way of this observation, one may rewrite the right-hand side of equation \ref{eq:Q1'} as
$$\mathrm{Rlim}_{f\in {^I}([J]^{<\omega})}
\,\mathrm{RHom}_{\mathsf{D}^{\geq 0}\mathsf{(Cond(Ab))}}\left(\prod_{i\in I} \underline{\mathbb{Z}^{f(i)}}, \underline{H}\right)$$
(cf.\ \cite[Prop.\ 3.6.3]{Prosmans}). Observe now that, since $\prod_{i\in I} \underline{\mathbb{Z}^{f(i)}}$ is projective in the subcategory $\mathsf{Solid}$ of $\mathsf{Cond(Ab)}$, and that the derived category of the former fully and faithfully embeds in that of the latter (\cite[Thm.\ 5.8]{CS1}),
the interior $\mathrm{RHom}$ above concentrates in degree zero, and may be computed in that degree to equal $\bigoplus_{i\in I}\bigoplus_{f(i)} H=\bigoplus_{X(f)} H$. This last point follows essentially from \cite[Prop.\ 97.7 or Cor.\ 94.5]{Fuchs_Infinite_73}.
We arrive in this way to the derived limits of the inverse systems $\mathbf{A}_{I,J}[H]$, the focus of the following subsection.
As will soon grow clear, these limits are substantially more complicated than the left-hand side expressions of line \ref{eq:Q1'}; as such, they exemplify the computational import of Question \ref{ques:pro-D}: it is asking whether, on a significant subcategory of $\mathsf{D}^{\geq 0}\mathsf{(Cond(Ab))}$, the $\mathrm{RHom}_{\mathsf{D}^{\geq 0}\mathsf{(Cond(Ab))}}$ functor admits a more transparent or manageable description.
That subcategory $\mathsf{Pro}(\mathsf{D}\mathsf{(Ab)})^{\mathtt{b}}$ is significant not only for formalizing questions about limits, but for its centrality (hinted at above) to the intermediate subcategory $\mathsf{D}^{\geq 0}(\mathsf{Solid})$, one of the most critical in condensed mathematics (see \cite[\S 5, 6]{CS1}, \cite[\S 2]{CS2}).
In essence, the above conversion was via the case of finite $J$, and the rather nontrivial computation of that case is a foundational result in the theory of solid abelian groups. Question \ref{ques:pro-D} may be read as one of how that computation extends.

\subsection{The systems $\mathbf{A}_{\kappa,\lambda}$}
\label{subsect:systemsAkl}
Readers skipping over \ref{subsect:pro-to-D} to this subsection should refer, when necessary, to the previous paragraph for any unfamiliar notation.

\begin{definition}
\label{def:Akappalambda}
Let $I$ and $J$ be nonempty sets; for any function $f:I\to [J]^{<\omega}$, let
$$A_f=\bigoplus_{X(f)}\mathbb{Z}\,.$$
For any $f\leq g$ in ${^I}([J]^{<\omega})$, there is a natural projection map
$$p_{fg}:A_g\to A_f;$$
these together comprise the inverse system
$$\mathbf{A}_{I,J}=\left(A_f,p_{fg},{^I}([J]^{<\omega})\right).$$
For any abelian group $H$, let $\mathbf{A}_{I,J}[H]$ denote the inverse system defined by replacing each instance of $\mathbb{Z}$ above with $H$; similarly for the systems defined in parallel below.
\end{definition}
We may now summarize the reasoning of Subsection \ref{subsect:pro-to-D} in the following terms:
\begin{proposition}
\label{prop:fullifvanishing}
If the natural map $\mathsf{Pro}(\mathsf{D}^{\geq 0}\mathsf{(Ab)})^{\mathtt{b}}\to \mathsf{D}^{\geq 0}\mathsf{(Cond(Ab))}$ is fully faithful then $\mathrm{lim}^n\,\mathbf{A}_{I,J}[H]=0$ for all free abelian groups $H$, sets $I$ and $J$, and $n>0$. \qed
\end{proposition}

\begin{remark}
\label{rmk:AklandA}
Let $\kappa=|I|$ and $\lambda=|J|$ and observe that bijections $I\to\kappa$ and $J\to\lambda$ determine an isomorphism of the systems $\mathbf{A}_{I,J}$ and $\mathbf{A}_{\kappa,\lambda}$, so that we may without loss of generality restrict our attention to systems of the latter form.
Thus for any nonzero cardinals $\kappa$ and $\lambda$ define inverse systems $\mathbf{B}_{\kappa,\lambda}=(B_f,p_{fg},{^\kappa}([\lambda]^{<\omega}))$ by letting $$B_f=\prod_{X(f)}\mathbb{Z}$$
and $\left(\mathbf{B}/\mathbf{A})_{\kappa,\lambda}=(B_f/A_f,p_{fg},{^\kappa}([\lambda]^{<\omega}\right))$, where the maps $p_{fg}$ are in each case the obvious projection maps.
Note that $\mathrm{lim}\,\mathbf{A}_{\kappa,\lambda}\cong\bigoplus_\kappa\prod_{\lambda}\mathbb{Z}$ and that $\mathrm{lim}\,\mathbf{B}_{\kappa,\lambda}\cong\prod_{\kappa\times\lambda}\mathbb{Z}$.

Write $\mathbf{X}\restriction\Lambda$ for the restriction of an inverse system $\mathbf{X}$ to a suborder $\Lambda$ of its index-set.
Letting $[\omega]^{\mathrm{in}}$ denote the set of strictly initial segments of $\omega$ then induces an identification of $\mathbf{A}_{\omega,\omega}\restriction {^\omega}([\omega]^{\mathrm{in}})$ with the inverse system denoted $\mathbf{A}$ in \cite{Mardesic_additive_88} and \cite{Bergfalk_simultaneously_23}, for example, and much intervening literature; since $[\omega]^{\mathrm{in}}$ is cofinal in $[\omega]^{<\omega}$, it follows that $\mathbf{A}_{\omega,\omega}$ and $\mathbf{A}$ are isomorphic in the category $\mathsf{Pro}(\mathsf{Ab})$.
The lemmas and definitions preceding Theorem \ref{thm:limsofAkappalambda} below are routine generalizations to the systems $\mathbf{A}_{\kappa,\lambda}$ of a sequence standard within the literature of $\mathbf{A}$.
\end{remark}

\begin{lemma}
\label{lem:Bvanishing}
$\mathrm{lim}^n\,\mathbf{B}_{\kappa,\lambda}=0$ for $n>0$.
\end{lemma}

\begin{proof}
Observe that $\mathrm{lim}\,\mathbf{B}_{\kappa,\lambda}$ surjects onto $\mathrm{lim}\,(\mathbf{B}_{\kappa,\lambda}\restriction\Lambda)$ for any downwards-closed suborder $\Lambda$ of ${^\kappa}([\lambda]^{<\omega})$; in other words, $\mathbf{B}_{\kappa,\lambda}$ is \emph{flasque} in the sense of \cite{Jensen_les}.
By way of this observation, the lemma follows from \cite[Th\'{e}or\`{e}me 1.8]{Jensen_les}.
\end{proof}
Consider now the short exact sequence of inverse systems
\begin{align}
\label{eq:ABBASES}
\mathbf{0}\to\mathbf{A}_{\kappa,\lambda}\to\mathbf{B}_{\kappa,\lambda}\to(\mathbf{B}/\mathbf{A})_{\kappa,\lambda}\to \mathbf{0}.\end{align}
The functor $\mathrm{lim}$ being left exact, its application to (\ref{eq:ABBASES}) induces a long exact sequence
$$0\to\mathrm{lim}\,\mathbf{A}_{\kappa,\lambda}\to\mathrm{lim}\,\mathbf{B}_{\kappa,\lambda}\to\mathrm{lim}\,(\mathbf{B}/\mathbf{A})_{\kappa,\lambda}\to \mathrm{lim}^1\,\mathbf{A}_{\kappa,\lambda}\to\mathrm{lim}^1\,\mathbf{B}_{\kappa,\lambda}\to\cdots$$
whose terms $\mathrm{lim}^n\,\mathbf{X}$ are the $n^{\mathrm{th}}$ cohomology groups of $\mathrm{Rlim}\,\mathbf{X}$ (with a natural identification of $\mathrm{lim}$ and $\mathrm{lim}^0$).
Lemma \ref{lem:Bvanishing} then implies that 
\begin{align}
\label{eq:lim1}
\mathrm{lim}^1\,\mathbf{A}_{\kappa,\lambda}\cong\frac{\mathrm{lim}\,(\mathbf{B}/\mathbf{A})_{\kappa,\lambda}}{\mathrm{im}(\mathrm{lim}\,\mathbf{B}_{\kappa,\lambda})}\end{align} and that 
\begin{align}
\label{eq:limn}
\mathrm{lim}^n\,\mathbf{A}_{\kappa,\lambda}\cong\mathrm{lim}^{n-1}\,(\mathbf{B}/\mathbf{A})_{\kappa,\lambda}\text{ for all }n>1.
\end{align}
These equations facilitate combinatorial reformulations of the conditions 
\[``\mathrm{lim}^n\,\mathbf{A}_{\kappa,\lambda}=0"
\]
for $n>0$ as in the definition and lemma below; in what follows, $=^*$ denotes equality modulo a finite set.
\begin{definition}
\label{def:coherence}
Let $\Phi=\langle\varphi_f:X(f)\to\mathbb{Z}\mid f\in {^\kappa}([\lambda]^{<\omega})\rangle$ be a family of functions.
\begin{itemize}
\item $\Phi$ is \emph{coherent} if $$\varphi_f=^*\varphi_g|_{X(f)}$$ for all $f\leq g$ in ${^\kappa}([\lambda]^{<\omega})$.
\item $\Phi$ is \emph{trivial} if there is a function $\psi:\kappa\times\lambda\to\mathbb{Z}$ such that $$\varphi_f=^*\psi|_{X(f)}$$ for all $f\in {^\kappa}([\lambda]^{<\omega})$.
\end{itemize}
For any $(n+1)$-tuple $\vec{f}$, write $\vec{f}^i$ for the $n$-tuple formed by omitting the $i^{\mathrm{th}}$ element of $\vec{f}$.
Also write $X(\vec{f})$ for $\bigcap_{f\in\vec{f}}X(f)$ and let $\Phi=\langle\varphi_{\vec{f}}:X(\vec{f})\to\mathbb{Z}\mid \vec{f}\in ({^\kappa}([\lambda]^{<\omega}))^n\rangle$ be a family of functions for some $n>1$.
Here and in all that follows, we assume without further comment that such $\Phi$ are \emph{alternating}, meaning that $\varphi_{\vec{f}}=\mathrm{sign}(\sigma)\,\varphi_{\sigma\cdot\vec{f}}$ for all permutations $\sigma$ of $n$, where $\sigma\cdot\vec{f}$ denotes the natural action of $\sigma$ on $\vec{f}$.
\begin{itemize}
\item $\Phi$ is \emph{$n$-coherent} if $$\sum_{i\leq n}(-1)^i\varphi_{\vec{f}^i}|_{X(\vec{f})}=^* 0$$
for all $\vec{f}\in ({^\kappa}([\lambda]^{<\omega}))^{n+1}$.
\item $\Phi$ is \emph{$n$-trivial} (or simply \emph{trivial} if the value of $n$ is clear from context) if there exists a $\Psi=\langle\psi_{\vec{f}}:X(\vec{f})\to\mathbb{Z}\mid \vec{f}\in ({^\kappa}([\lambda]^{<\omega}))^{n-1}\rangle$ such that
$$\sum_{i<n}(-1)^i\psi_{\vec{f}^i}|_{X(\vec{f})}=^* \varphi_{\vec{f}}$$
for all $\vec{f}\in ({^\kappa}([\lambda]^{<\omega}))^{n}$.
\end{itemize}
\end{definition}

The adaptations of these definitions to families of $H$-valued functions for an arbitrary $H \in \mathsf{Ab}$ are straightforward, and the following lemma will continue to apply.
For simplicity, however, we will tend in this subsection to proceed, as above, with a default codomain of $\mathbb{Z}$.
Note also that \emph{coherence} readily identifies in the above definition with the $n=1$ instance of \emph{$n$-coherence}; this permits more uniform statements, as in the lemma below.

\begin{lemma}
\label{lem:zeroiftriv}
For all $n>0$, $\mathrm{lim}^n\,\mathbf{A}_{\kappa,\lambda}=0$ if and only if every $n$-coherent family of functions indexed by ${^\kappa}([\lambda]^{<\omega})$ is trivial.
\end{lemma}
\begin{proof} In the case of $n=1$, elements of $\mathrm{lim}\,(\mathbf{B}/\mathbf{A})_{\kappa,\lambda}$ may be viewed as collections of $=^*$-equivalence classes $\langle [\varphi_f]\mid f\in{^\kappa}([\lambda]^{<\omega})\rangle$, where $\langle\varphi_f\mid f\in {^\kappa}([\lambda]^{<\omega})\rangle$ is a coherent family of functions.
Since $\mathrm{lim}\,\mathbf{B}_{\kappa,\lambda}=\prod_{\kappa\times\lambda}\mathbb{Z}$ and the map $\mathrm{lim}\,\mathbf{B}_{\kappa,\lambda}\to\mathrm{lim}\,(\mathbf{B}/\mathbf{A})_{\kappa,\lambda}$ takes a function $\psi:\kappa\times\lambda\to\mathbb{Z}$ to the collection $\langle[\psi|_{X(f)}]\mid f\in {^\kappa}([\lambda]^{<\omega})\rangle$ of equivalence classes of its restrictions, the assertion follows from equation \ref{eq:lim1}.

The case of higher $n$ is close in spirit, only a bit more tedious; since the argument is so essentially that given in \cite[Section 2.1]{Bergfalk_simultaneously_21}, readers are referred there for details.
The fundamental point is that, by equation \ref{eq:limn}, elements of $\mathrm{lim}^n\,\mathbf{A}_{\kappa,\lambda}$ admit representation as elements of $\mathrm{lim}^{n-1}\,(\mathbf{B}/\mathbf{A})_{\kappa,\lambda}$ which, when construed as cohomology classes in the associated Roos complex, give rise to the representations of Definition \ref{def:coherence}.
\end{proof}
We now record the most basic facts about the groups $\mathrm{lim}^n\,\mathbf{A}_{\kappa,\lambda}$; these will more than suffice for the applications structuring this and the following section.
For further (and subtler) results, see Theorem \ref{thm:limnAkl_further} below.
\begin{theorem}
\label{thm:limsofAkappalambda}
The following hold for the inverse systems $\mathbf{A}_{\kappa,\lambda}$:
\begin{enumerate}
\item If $\kappa$ or $\lambda$ is finite then $\mathrm{lim}^n\,\mathbf{A}_{\kappa,\lambda}=0$ for all $n>0$.
\item If $\kappa=\lambda=\aleph_0$ then it is consistent with the \textsf{ZFC} axioms that $\mathrm{lim}^n\,\mathbf{A}_{\kappa,\lambda}=0$ for all $n>0$.
\item If $\kappa=\lambda=\aleph_0$ and $n>0$ then it is consistent with the \textsf{ZFC} axioms that $\mathrm{lim}^n\,\mathbf{A}_{\kappa,\lambda}\neq 0$.
\item If $\mu\geq\kappa$ and $\nu\geq\lambda$ and $n\geq 0$ then $\mathrm{lim}^n\,\mathbf{A}_{\kappa,\lambda}$ forms a subgroup of $\mathrm{lim}^n\,\mathbf{A}_{\mu,\nu}$.
\item If $\kappa=\aleph_0$ and $\lambda=\aleph_1$ then $\mathrm{lim}^1\,\mathbf{A}_{\kappa,\lambda}\neq 0$.
\end{enumerate}
Moreover, for any nontrivial abelian group $H$, these results all equally hold if $\mathbf{A}_{\kappa,\lambda}$ is replaced with $\mathbf{A}_{\kappa,\lambda}[H]$.
\end{theorem}

The following corollary now establishes Theorem A from the introduction.

\begin{corollary}
\label{cor:notfull}
The functor of equation \ref{eq:Q1} is not full.
\end{corollary}

\begin{proof}
This is immediate from item (5) of Theorem \ref{thm:limsofAkappalambda}, together with Proposition \ref{prop:fullifvanishing}.
\end{proof}

\begin{proof}[Proof of Theorem \ref{thm:limsofAkappalambda}] 
Item (1) is immediate from Definition \ref{def:coherence} and Lemma \ref{lem:zeroiftriv}, for if $\kappa$ is finite then any associated $n$-coherent function is trivialized by a family of zero functions (or by a single zero function, if $n=1$), while if $\lambda$ is finite then the partial order ${^\kappa}([\lambda]^{<\omega})$ possesses a maximum.

By \cite{Bergfalk_simultaneously_23} and Remark \ref{rmk:AklandA}, item (2) holds in any forcing extension given by adding $\beth_\omega$-many Cohen reals; see also \cite{Bergfalk_simultaneously_21} for an alternate deduction of its consistency from the existence of a weakly compact cardinal. Similarly, item (3) is the main result of \cite{Velickovic_non_21}; more precisely, item (3) was therein shown to hold for $\mathbf{A}_{\aleph_0,\aleph_0}[H]$ for any fixed nontrivial abelian group $H$.
For a strong generalization of item (2) to arbitrary abelian groups $H$, see \cite[Theorem 1.2]{Bannister_additivity_23} and the discussion following the proof of Theorem \ref{thm:fromstacks} below.

For item (4), fix $n\geq 0$ and observe that for any $\mu\geq\kappa$ and $\nu\geq\lambda$, any $n$-coherent $\Phi=\langle\varphi_{\vec{f}}:X(\vec{f})\to\mathbb{Z}\mid\vec{f}\in {^\kappa}([\lambda]^{<\omega})\rangle$ naturally extends to an $n$-coherent $\Psi=\langle\psi_{\vec{f}}:X(\vec{f})\to\mathbb{Z}\mid\vec{f}\in {^\mu}([\nu]^{<\omega})\rangle$ which is trivial if and only if $\Phi$ is, an extension which determines, in turn, an embedding of $\mathrm{lim}^n\,\mathbf{A}_{\kappa,\lambda}$ into $\mathrm{lim}^n\,\mathbf{A}_{\mu,\nu}$.
Concretely, for any $f:\mu\to[\nu]^{<\omega}$ define $g:\kappa\to [\lambda]^{<\omega}$ by $g(\xi)=f(\xi)\cap\lambda$ for all $\xi\in\kappa$, a correspondence extending in an obvious fashion to $n$-tuples $\vec{f}\in({^\mu}([\nu]^{<\omega}))^n$.
Let then $\psi_{\vec{f}}\,|_{X(\vec{g})}=\varphi_{\vec{g}}$ and $\psi_{\vec{f}}\,|_{X(\vec{f})\backslash X(\vec{g})}=0$; the verification that $\Psi$ is as claimed is straightforward and left to the reader.

For item (5), begin by fixing a family $E=\langle e_\alpha:\alpha\to\omega\mid\alpha<\omega_1\rangle$ of finite-to-one functions which is nontrivially coherent in the classical sense that $e_\beta |_\alpha=^* e_\alpha$ for all $\alpha<\beta<\omega_1$ but for no $e:\omega_1\to\omega$ does $e |_\alpha =^* e_\alpha$ for all $\alpha<\omega_1$.
(Among the more prominent such families are the fiber maps $e_\alpha=\rho_1(\,\cdot\,,\alpha)$ of the $\rho_1$ or $\rho$ functions of \cite{Todorcevic_Walks_07}, for example, but interested readers might equally construct such an $E$ by transfinite recursion.)
For any $f:\omega\to [\omega_1]^{<\omega}$, let $\mathrm{sp}(f)=\mathrm{sup}\bigcup_{i\in\omega} f(i)$ and let
\begin{equation*}
\varphi_f(i,\xi)=
    \begin{cases}
        1 & \text{if } e_{\mathrm{sp}(f)}(\xi)=i \\
        0 & \text{otherwise}
    \end{cases}
\end{equation*}
for all $(i,\xi)\in X(f)$.
To complete the proof, it will suffice to show that $\Phi=\langle\varphi_f\mid f\in {^\omega}([\omega_1]^{<\omega})\rangle$ is nontrivially coherent.
To aid in seeing this, consider an auxiliary family $T=\langle\tau_\gamma\mid \gamma<\omega_1\rangle$ of functions $\tau_\gamma:\omega\times\gamma\to\{0,1\}$, each of which is the characteristic function of the reflection of the graph of $e_\gamma$; more concisely, $\tau_\gamma(i,\xi)=1$ if and only if $e_\gamma(\xi)=i$.
It is then clear from the observation that $\varphi_f=\tau_{\mathrm{sp}(f)}|_{X(f)}$ that $\Phi$, by way of $T$, inherits the coherence of $E$.

For the nontriviality of $\Phi$, simply observe that $\varphi |_{X(f)}\neq^*\varphi_f$ for any $\varphi_f$ whose domain $X(f)$ contains the $\{(k,\alpha_k)\mid k\in Y\}$ of the following claim:

\begin{claim}
For any $\varphi:\omega\times\omega_1\to\mathbb{Z}$ there exist a $\gamma<\omega_1$ and infinite $Y\subseteq \omega$ and $\alpha_k<\gamma$ with $\tau_\gamma(k,\alpha_k)\neq\varphi(k,\alpha_k)$ for all $k\in Y$.
\end{claim}
\begin{proof}
If not, then for some cofinal $\Gamma\subseteq\omega_1$ and $\ell\in\omega$ we would have $\tau_\gamma|_{[\ell,\omega)\times\gamma}=\varphi|_{[\ell,\omega)\times\gamma}$ for all $\gamma\in\Gamma$.
Let then $e(\xi)=n$ if and only if $n\geq\ell$ and $\varphi(n,\xi)=1$; this defines a
 partial function $e:\omega_1\to\omega$, any extension of which to all of $\omega_1$ trivializes $E$ (since each $e_\gamma\in E$ is finite-to-one), a contradiction. \renewcommand{\qedsymbol}{\twoqedbox}
\end{proof}
\let\qed\relax
\end{proof}

\subsection{An alternative framing}
\label{subsect:alternative}

Relations between the main expressions of Sections \ref{subsect:pro-to-D} and \ref{subsect:systemsAkl} may be more neatly summarized as follows. The theorem is due to Clausen and Scholze, and appears (in essence) in the closing minutes of the fourth lecture in their \emph{Analytic Stacks} series \cite{analytic_stacks}.
We fill in the details of its proof. Note that this establishes clauses (1) and (2) 
of Theorem B.

\begin{theorem}
\label{thm:fromstacks}
For any cardinals $\mu>0$ and $\lambda\geq\aleph_0$, the following assertions are equivalent:
  \begin{enumerate}
    \item $\lim^n \mb{A}_{\aleph_0,\lambda}[\bigoplus_{\mu} 
    \bb{Z}] = 0$ for all $n>0$.
    \item $\mathrm{RHom}_{\mathsf{D}^{\geq 0}\mathsf{(Cond(Ab))}}
    (\prod_{\omega} \bigoplus_\lambda \underline{\bb{Z}}, \bigoplus_{\mu} 
    \underline{\bb{Z}})$ is concentrated in degree zero.
    \item Whenever $\mb{M} = (M_i, \pi_{i,j}, \omega)$ is an inverse sequence  of abelian groups 
    of cardinality at most $\lambda$ whose transition maps $\pi_{i,j} : M_{j} \to M_i$ are all surjective, 
    and $H$ is an abelian group of cardinality at most $\mu$, we have
$$\mathrm{Ext}_{\mathsf{Cond(Ab)}}^n\left(\lim_{i < \omega}\,\underline{M_i},\underline{H}\right)=\colim_{i < \omega}\,\mathrm{Ext}_{\mathsf{Cond(Ab)}}^n\left(\underline{M_i},\underline{H}\right)$$
for all $n \geq 0$.
  \end{enumerate}
\end{theorem}

\begin{proof}
  An argument that $(1) \Rightarrow (2)$ concluded Section \ref{subsect:pro-to-D}; let us briefly recall it, emphasizing its reversibility, thereby establishing $(2) \Rightarrow (1)$.
  
Again, since 
  \[
    \prod_{\omega} \bigoplus_\lambda \bb{Z} = 
    \colim_{f \in {^{\omega}([\lambda]^{<\omega})}} 
    \prod_{i < \omega} \bb{Z}^{f(i)},
  \]
the expression in clause 2 may be written as
  \[
    \mathrm{Rlim}_{f \in {^{\omega}([\lambda]^{<\omega})}}
    \mathrm{RHom}_{\mathsf{D}^{\geq 0}\mathsf{(Cond(Ab))}}\left(\prod_{i < \omega} 
    \underline{\bb{Z}^{f(i)}}, \bigoplus_\mu \underline{\bb{Z}}\right).
  \]
  Since $\prod_{i < \omega} \underline{\bb{Z}^{f(i)}}$ is projective in 
  $\mathsf{Solid}$, the $\mathrm{RHom}$ in the above expression concentrates 
  in degree zero, where it is equal to $\bigoplus_{X(f)} \bigoplus_{\mu} \bb{Z}$. 
  The clause 2 expression thus reduces to 
  \[
    \mathrm{Rlim}_{f \in {^{\omega}([\lambda]^{<\omega})}} 
    \bigoplus_{X(f)} \bigoplus_{\mu} \bb{Z} = \mathrm{Rlim} 
    \,\mb{A}_{\aleph_0, \lambda}\left[\bigoplus_{\mu} \bb{Z}\right],
  \]
hence concentrates in degree zero if and only if the rightmost expression just above does; this shows $(1) \Leftrightarrow (2)$.
 
     Let us now show that $(2) \Rightarrow (3)$. To this end, fix 
  $\mb{M}$ and $H$ as in the statement of (3). By resolving $H$ by 
  free groups of cardinality $\mu$, we reduce to the case in which 
  $H = \bigoplus_\mu \bb{Z}$. 
  Turning to the first coordinate, fix a resolution 
  \begin{equation}
  \label{eq:basicKFM}
  0 \ra K_i \ra F_i \ra M_i \ra 0
  \end{equation}
  of each $M_i$ in which each $F_i$ and $K_i$ are free groups of size $\lambda$ (i.e., each 
  $F_i$ and $K_i$ are of the form $\bigoplus_\lambda \bb{Z}$) and observe that since products are 
  exact in $\CondAb$ (\cite[Thm.\ 2.2 (AB4*)]{CS1}), these resolutions assemble into a short exact sequence
  \begin{equation}
  \label{eq:prodSES}
    0 \ra \prod_{\omega} \bigoplus_\lambda \underline{\bb{Z}} \ra 
    \prod_{\omega} \bigoplus_\lambda \underline{\bb{Z}} \ra 
    \prod_{i<\omega} \underline{M_i} \ra 0
  \end{equation}
of condensed abelian groups with natural projections to the condensations of each short exact sequence as in line \ref{eq:basicKFM}.
Applying $\mathrm{Hom}( - , \bigoplus_{\mu} \underline{\bb{Z}})$ to all of them determines a family of long exact sequences; under our assumptions, these carry the following implications:
\begin{enumerate}[label=(\alph*)]
\item $\mathrm{Ext}^n(\prod_{i < \omega} \underline{M_i}, \bigoplus_\mu \underline{\bb{Z}}) 
    = 0$, and hence
    \begin{equation}
\label{eq:prodcoprodequality}    
    \mathrm{Ext}^n\left(\prod_{i < \omega} \underline{M_i}, 
  \bigoplus_\mu \underline{\bb{Z}}\right) \cong \bigoplus_{i < \omega} \mathrm{Ext}^n\left(\underline{M_i}, 
  \bigoplus_\mu \underline{\bb{Z}}\right)\end{equation} for all $n>1$.
\item Equation \ref{eq:prodcoprodequality} in fact holds for $n=0$ and $n=1$ as well.
\end{enumerate}
The first is an easy consequence of the theorem's clause 2, while the second further leverages the slenderness of $\bigoplus_\mu \bb{Z}$ \cite[\S 94]{Fuchs_Infinite_73}.
For the aforementioned projection maps from (\ref{eq:prodSES}) to the condensations of (\ref{eq:basicKFM}) induce a chain map from the sum of the long exact sequences given by $\mathrm{Hom}( - , \bigoplus_{\mu} \underline{\bb{Z}})$ of the latter to that given by $\mathrm{Hom}( - , \bigoplus_{\mu} \underline{\bb{Z}})$ of the former, and slenderness tells us that the two constituents $$\bigoplus_{\omega}\,\mathrm{Hom}\left(\bigoplus_\lambda 
    \underline{\bb{Z}}, \bigoplus_\mu \underline{\bb{Z}}\right)\to\mathrm{Hom}\left(\prod_{\omega} \bigoplus_\lambda \underline{\bb{Z}}, \bigoplus_\mu \underline{\bb{Z}}\right)$$
    of this chain map are each isomorphisms. This, together with the Five Lemma and item (a), implies that the two other nontrivial constituents of this map, namely the $n=0$ and $n=1$ instances of 
        \begin{equation*}   
    \bigoplus_{i < \omega} \mathrm{Ext}^n\left(\underline{M_i}, 
  \bigoplus_\mu \underline{\bb{Z}}\right)\to\mathrm{Ext}^n\left(\prod_{i < \omega} \underline{M_i}, 
  \bigoplus_\mu \underline{\bb{Z}}\right)\end{equation*}
are isomorphisms as well.

To upgrade these isomorphisms from products and coproducts to limits and colimits, consider, for any nonzero $j\leq\omega$, the shift map
$$\mathrm{sh}:\prod_{i<1+j}\underline{M_i}\to\prod_{i<j}\underline{M_i}:(x_i)\mapsto(\pi_{i,i+1}(x_{i+1}))$$
and observe that for any finite such $j$, both rows of the natural commutative diagram
\[
\begin{tikzcd}
0 \arrow[r] & \lim_{i<\omega}\,\underline{M_i} \arrow[r] \arrow[d] & \prod_{i<\omega}\underline{M_i} \arrow[r, "\mathrm{id}-\mathrm{sh}"] \arrow[d] & \prod_{i<\omega}\underline{M_i} \arrow[r] \arrow[d] & 0 \\
0 \arrow[r]           & \underline{M_j} \arrow[r]                              & \prod_{i< j+1}\underline{M_i} \arrow[r, "\mathrm{id}-\mathrm{sh}"]          & \prod_{i< j}\underline{M_i} \arrow[r]            & 0          
\end{tikzcd}
\]
are exact, since the transition maps of $\mathbf{M}$ are all surjective.
Again applying $\mathrm{Hom}( - , \bigoplus_{\mu} \underline{\bb{Z}})$ induces a family of long exact sequences, and a chain map from the colimit of those associated to the lower row to that associated to the upper row above.
Since every second and third constituent of this chain map is, by the previous paragraph, an isomorphism of the form
$$\colim_{j<\omega}\,\mathrm{Ext}^n\left(\prod_{i<j}\,\underline{M_i},\bigoplus_\mu\underline{\bb{Z}}\right)\cong\bigoplus_{i < \omega} \mathrm{Ext}^n\left(\underline{M_i}, 
  \bigoplus_\mu \underline{\bb{Z}}\right)\xrightarrow{\cong}\mathrm{Ext}^n\left(\prod_{i < \omega} \underline{M_i}, 
  \bigoplus_\mu \underline{\bb{Z}}\right),$$
  again the Five Lemma implies that all the other constituents of this chain map, namely the natural morphisms $$\colim_{j<\omega}\,\mathrm{Ext}^n\left(\underline{M_j},\bigoplus_\mu\underline{\bb{Z}}\right)\to \mathrm{Ext}^n\left(\mathrm{lim}_{i<\omega}\,\underline{M_i},\bigoplus_\mu\underline{\bb{Z}}\right)$$
for each $n\in\omega$, are isomorphisms as well. This shows $(2) \Rightarrow (3)$.
But since $\prod_{\omega} 
  \bigoplus_\lambda \underline{\bb{Z}}$ admits expression as the limit of a sequence as 
  in clause 3, we in fact have $(2) \Leftrightarrow (3)$, and this concludes the proof 
  of the theorem.
  \end{proof}

Let us briefly discuss the consistency of the (equivalent) assertions in Theorem 
\ref{thm:fromstacks}, first considering the case in which $\lambda = \aleph_0$. 
In recent work, Bannister considers the additivity 
of derived limits for a particular type of inverse system known as an \emph{$\Omega_\kappa$-system} 
(where the parameter $\kappa$ denotes an arbitrary infinite cardinal). In particular, he proves 
in \cite[Thm.\ 1.2]{Bannister_additivity_23} that, for any fixed $\kappa$, derived limits are 
additive for $\Omega_\kappa$-systems in any forcing extension obtained by adding at least 
$\beth_\omega(\kappa)$-many Cohen reals to a model of $\mathsf{ZFC}$. The assertion that 
$\lim^n \mb{A}_{\aleph_0,\aleph_0}[H] = 0$ for all $n > 0$ and all abelian groups $H$ is a special 
case of the additivity of derived limits for $\Omega_{\aleph_0}$-systems; therefore,  
for $\lambda = \aleph_0$, the conditions 
in Theorem \ref{thm:fromstacks} hold simultaneously for all $\mu > 0$ in any forcing extension 
obtained by adding at least $\beth_\omega$-many Cohen reals. In particular, starting with a model 
of $\mathsf{GCH}$, one sees that they are consistent with $2^{\aleph_0} = \aleph_{\omega+1}$. 

Furthermore, a result of Casarosa and Lambie-Hanson \cite{casarosa_lh} shows that, again for 
$\lambda = \aleph_0$, $\aleph_{\omega+1}$ is the \emph{minimum} value of the continuum compatible 
with the conditions holding simultaneously for all $\mu$. They prove that, if the dominating 
number $\mathfrak{d}$ equals $\aleph_n$ for some $1 \leq n < \omega$, then 
$\lim^n \mb{A}_{\aleph_0,\aleph_0}[\bigoplus_{\omega_n} 
\mathbb{Z}] \neq 0$. In particular, if the conditions in Theorem \ref{thm:fromstacks} hold 
for $\lambda = \aleph_0$ and all cardinals $\mu < \aleph_\omega$, then we must have 
$2^{\aleph_0} > \aleph_\omega$.
A more axiomatic approach to these results is possible as well; the assumption $[\mathrm{PH}_n$ for all $n\geq 0]$, where $\mathrm{PH}_n$ denotes the $n^{\mathrm{th}}$ \emph{partition hypothesis} of \cite{BBMT}, also implies Theorem \ref{thm:fromstacks} for $\lambda=\aleph_0$ and all $\mu>0$ simultaneously.
This, though, is a logically stronger assumption than those discussed above: by \cite[Cor.\ 8.21]{BBMT}, its consistency strength is exactly that of the existence of a weakly compact cardinal.

On the other hand, if $\lambda > \aleph_0$, then the conditions of Theorem \ref{thm:fromstacks} are inconsistent even when 
$\mu = 1$, by clauses (5) and (4) of Theorem \ref{thm:limsofAkappalambda}.

We end this subsection by highlighting a particular consequence of the 
assertions in Theorem \ref{thm:fromstacks} for the duality of 
condensed Banach spaces and condensed Smith spaces in non-archimedean 
settings. For concreteness, we work over $\bb{Q}_p$ for some fixed 
prime $p$. As noted in \cite[Lecture 6]{analytic_stacks}, the derived 
category $\mathsf{D(Solid}_{\bb{Q}_p})$ of solid $\bb{Q}_p$-vector 
spaces is a full subcategory of the derived category $\mathsf{D(Solid)}$ 
of solid abelian groups, which in turn is a full subcategory of the 
the derived category $\mathsf{D(Cond}(\mathsf{Ab}))$ of condensed 
abelian groups.

Classically, a Smith space is a complete compactly generated locally 
convex topological vector space that has a universal compact set. Smith 
spaces were introduced in \cite{smith}, where they were shown to be in 
stereotype duality with Banach spaces.

This Banach--Smith duality persists in the condensed setting. Recall 
that $\bb{Q}_p = \bb{Z}_p[\frac{1}{p}]$, where the latter expression is 
defined to be 
\[
  \colim (\bb{Z}_p \xrightarrow{\times p} \bb{Z}_p \xrightarrow{\times p} \bb{Z}_p \xrightarrow{\times p} \cdots ).
\]
As shown in \cite[Lemma 3.8]{rj_rc} (see also \cite[Lecture 6]{analytic_stacks}, which more closely matches our notation), 
the solid $\bb{Q}_p$-Banach spaces 
are precisely the solid $\bb{Q}_p$-vector spaces of the form
\[
  \Big(\bigoplus_I \bb{Z}_p\Big)^{\wedge}_p\left[\frac{1}{p}\right] = 
  \left(\lim_{n<\omega} \Big(\bigoplus_I \bb{Z}_p/p^n \bb{Z}_p\Big)\right)\left[\frac{1}{p}\right] 
\] for some index set $I$,
while the solid $\bb{Q}_p$-Smith spaces are precisely the solid 
$\bb{Q}_p$-vector spaces of the form
\[
  \Big(\prod_I \bb{Z}_p\Big)\left[\frac{1}{p}\right]
\]
for some index set $I$ (readability dictates that for the remainder of this section we forego the underline notation for condensed images of rings).
It is shown in \cite[Lemma 3.10]{rj_rc} that the
classical Banach--Smith duality extends to the solid $\bb{Q}_p$-vector space 
setting via the internal $\mathrm{Hom}$ functor $\underline{\mathrm{Hom}}(-,-): (\mathsf{Solid}_{\bb{Q}_p})^{\mathrm{op}} \times 
\mathsf{Solid}_{\bb{Q}_p} \ra \mathsf{Solid}_{\bb{Q}_p}$ (see \cite[p.\ 13]{CS1} for basic definitions of condensed internal $\mathrm{Hom}$ functors). In particular, 
there it is shown that, for any index set $I$, we have
\begin{itemize}
    \item $\underline{\mathrm{Hom}}\left((\bigoplus_I \bb{Z}_p)^{\wedge}_p\left[\frac{1}{p}\right], \bb{Q}_p\right) = (\prod_I \bb{Z}_p)\left[\frac{1}{p}\right]$, and
    \item $\underline{\mathrm{Hom}}\left((\prod_I \bb{Z}_p)\left[\frac{1}{p}\right], 
    \bb{Q}_p\right) = (\bigoplus_I \bb{Z}_p)^{\wedge}_p\left[\frac{1}{p}\right]$.
\end{itemize}
As $\prod_I\mathbb{Z}_p$ is, moreover, a projective solid $\mathbb{Z}_p$-module, it is essentially immediate that the second of these equalities holds even at the level of the derived category $\mathsf{D(Solid}_{\bb{Q}_p})$, i.e., even with the 
$\underline{\mathrm{Hom}}$ replaced by $\mathrm{R}\underline{\mathrm{Hom}}$; see \cite[Remark 3.11]
{rj_rc}. The question of whether this holds for the first bullet (and hence of whether this is a \emph{derived} duality relation), though, is more delicate. We end this subsection with a brief proof of the observation of Clausen and Scholze (again, see \cite[Lecture 6]{analytic_stacks}) that, under the conditions 
of Theorem \ref{thm:fromstacks} with $\lambda = \aleph_0$, this question also has an affirmative 
answer in the context of 
\emph{separable} solid $\bb{Q}_p$-Banach spaces, i.e., in the situation 
in which the index set $I$ is countable. The following theorem yields clause (3) 
of Theorem B:

\begin{theorem}
  Suppose that $\lim^n \mb{A}_{\aleph_0,\aleph_0}[H] = 0$ for all $n > 0$ and all abelian groups $H$. 
  Then, for all primes $p$, we have 
  \[
    \mathrm{R}\underline{\mathrm{Hom}}\left(\Big(\bigoplus_\omega \bb{Z}_p\Big)^{\wedge}_p\left[\frac{1}{p}\right], 
    \bb{Q}_p \right) = \Big(\prod_\omega \bb{Z}_p\Big)\left[\frac{1}{p}\right].
  \]
\end{theorem}

\begin{proof}
  By \cite[Lemma 3.10]{rj_rc}, we have $\underline{\mathrm{Hom}}\left(
  (\bigoplus_\omega \bb{Z}_p)^{\wedge}_p\left[\frac{1}{p}\right], 
  \bb{Q}_p \right) = (\prod_\omega \bb{Z}_p)\left[\frac{1}{p}\right]$.
  Therefore, it suffices to show that $\mathrm{R}\underline{\mathrm{Hom}}\left(
  (\bigoplus_\omega \bb{Z}_p)^{\wedge}_p\left[\frac{1}{p}\right], \bb{Q}_p \right)$
  is concentrated in degree zero.
Note that
  \[
    \Big(\bigoplus_\omega \bb{Z}_p\Big)^{\wedge}_p\left[\frac{1}{p}\right] \cong \lim_{n < \omega} 
    \Big(\bigoplus_\omega \bb{Q}_p/(p^n \bb{Z}_p)\Big)
  \]
  and that $\bigoplus_\omega \bb{Q}_p/(p^n \bb{Z}_p)$ is countable and discrete for each $n < \omega$. 
  Similarly, we have $\bb{Q}_p \cong \lim_{m < \omega} \bb{Q}_p/(p^n \bb{Z}_p)$. Therefore, for all 
  $S \in \mathsf{ED}$, we have
  \begin{align}
  \label{eq:Banach-Smith}
    & \mathrm{R\underline{Hom}}\left(\Big(\bigoplus_\omega \bb{Z}_p\Big)^{\wedge}_p\left[\frac{1}{p}\right], 
    \bb{Q}_p \right)(S) \nonumber \\ &= \underset{m<\omega}{\mathrm{Rlim}} \:\mathrm{R\underline{Hom}}\left(\lim_{n < \omega} 
    \Big(\bigoplus_\omega \bb{Q}_p/(p^n \bb{Z}_p)\Big), \bb{Q}_p/(p^m \bb{Z}_p)\right)(S) \nonumber \\ 
    &= \underset{m<\omega}{\mathrm{Rlim}} \:\mathrm{RHom}\left(\lim_{n < \omega} 
    \Big(\bigoplus_\omega \bb{Q}_p/(p^n \bb{Z}_p)\Big), \mathrm{Cont}(S, \bb{Q}_p/(p^m \bb{Z}_p))\right).
  \end{align}
For the second equality, see the third paragraph of \cite[\S 4]{TangProfinite} (or, a little more concretely, argue as for \cite[Lem.\ 11.1]{BLHS}). Since $S \in \mathsf{ED}$ and $\bb{Q}_p/(p^m \bb{Z}_p)$ is discrete, we know that 
  $\mathrm{Cont}(S, \bb{Q}_p/(p^m \bb{Z}_p))$ (with the compact-open topology) is a discrete abelian group 
  for all $m < \omega$. 
  Hence clause (3) of Theorem \ref{thm:fromstacks} applies to the interior RHom of line \ref{eq:Banach-Smith}, entailing that
  \begin{align*}
    & \mathrm{H}^i\mathrm{RHom}\left(\lim_{n<\omega} 
    \Big(\bigoplus_\omega \bb{Q}_p/(p^n \bb{Z}_p)\Big), \mathrm{Cont}(S, \bb{Q}_p/(p^m \bb{Z}_p))\right) \\ 
    & = \underset{n<\omega}{\mathrm{colim}} \:\mathrm{Ext}^i\left( 
    \Big(\bigoplus_\omega \bb{Q}_p/(p^n \bb{Z}_p)\Big), \mathrm{Cont}(S, \bb{Q}_p/(p^m \bb{Z}_p))\right) 
    \\ &= \underset{n<\omega}{\mathrm{colim}} \prod_\omega 
    \mathrm{Ext}^i\left(\bb{Q}_p/(p^n \bb{Z}_p), \mathrm{Cont}(S, \bb{Q}_p/(p^m \bb{Z}_p))\right)
  \end{align*}
  for all $i\geq 0$.
  Both arguments inside $\mathrm{Ext}^i$ in the last line above are (discrete) abelian 
  groups, and $\mathrm{Cont}(S, \bb{Q}_p/(p^m \bb{Z}_p))$ is divisible, and hence injective. 
  Thus 
  \[
    \mathrm{Ext}^i\left(\bb{Q}_p/(p^n \bb{Z}_p), \mathrm{Cont}(S, \bb{Q}_p/(p^m \bb{Z}_p))\right) = 0
  \]
  for all $i>0$ and the $\mathrm{Rlim}$ of line \ref{eq:Banach-Smith} is of an inverse sequence of abelian groups whose transition maps are all surjective. 
  In consequence, it, too, concentrates in the zeroth degree, and this completes the proof.
\end{proof}

\subsection{Consistent nonvanishing in higher degrees}
\label{subsect:morenonvanishing}

In the presence of additional assumptions, item (5) of Theorem \ref{thm:limsofAkappalambda} does generalize to higher degrees; a main result in this direction is Theorem \ref{thm:limnAkl_further} below.
Here a few preliminary words may be of value.

First, this being the paper's most combinatorially involved result, readers may wish to defer a reading of its proof until they've digested the argument of Theorem \ref{thm:nonzero_cohomology_on_opens}; its proof in fact will reference that argument at one point.
That said, these references are essentially to motifs which readers familiar with $n$-coherence arguments will already know.
Among the most central of these motifs is the following theorem \cite{Goblot}; this result is frequently instrumental, as in Theorems \ref{thm:limnAkl_further} and \ref{thm:nonzero_cohomology_on_opens}, in the construction of higher nontrivially coherent families.

\begin{theorem}[Goblot 1970]
\label{thm:Goblot} Let $I$ be a directed partial order indexing an inverse system $\mathbf{X}$ of abelian groups. If $\cf(I) \leq \aleph_k$, then $\mathrm{lim}^n\,\mathbf{X}=0$ for all $n>k+1$.
\end{theorem}

\noindent If the transitive maps of $\mathbf{X}$ are all surjective, this vanishing holds even for all $n>k$, as observed in \cite[Theorem 2.4]{Velickovic_non_21}; in what follows, by \emph{Goblot's Theorem} we will mean both Theorem \ref{thm:Goblot} and this corollary.

This theorem brings us to a second point, which is the fundamental relationship of the functors $\mathrm{lim}^n$ to the cardinals $\aleph_n$.
Since it will grow clearer in the course of this paper, there's no need to linger on this point, except to note that Theorem \ref{thm:limnAkl_further} is exemplary in two respects: (1) In it, each rise in the degree of the derived limit corresponds to a rise in the cofinality of a nonvanishing system; by Goblot's result, this is essentially unavoidable. (2) The question of whether the theorem's hypotheses may be weakened to the \textsf{ZFC} axioms is an open and good one (and is recorded in our conclusion). This brings us to a further item, which is that of the theorem's hypotheses; in it, for example, a notion of nontrivial $n$-coherence slightly different from that of Definition \ref{def:coherence} is invoked.
Both, along with the main mechanism of Theorem \ref{thm:nonzero_cohomology_on_opens}, fall within a more general framework which we pause to record; here $\mathcal{I}^n$ denotes the cartesian product of an ideal $\mathcal{I}$ with itself $n$ times.
\begin{definition}
\label{def:coh2}
For any $n>0$ and abelian group $H$ and ideals $\mathcal{I}$, $\mathcal{J}$ on a set $Y$, a family
$$\Phi=\langle\varphi_{\vec{I}}:\bigcap\vec{I}\to H\mid\vec{I}\in\mathcal{I}^n\rangle$$
is \emph{$n$-coherent (modulo $\mathcal{J}$)} if
$$\mathrm{supp}\left(\sum_{i\leq n}(-1)^i\varphi_{\vec{I}^i}\big|_{\bigcap\vec{I}}\right)\in\mathcal{J}$$
for all $\vec{I}\in \mathcal{I}^{n+1}$.

If $n=1$ then $\Phi$ is \emph{trivial} if there exists a $\psi:Y\to H$ such that
$$\mathrm{supp}\left(\psi\big|_{I}-\varphi_I\right)\in\mathcal{J}$$
for all $I\in\mathcal{I}$.
Similarly, if $n>1$ then $\Phi$ is \emph{trivial} if there exists a $$\Psi=\langle\psi_{\vec{I}}:\bigcap\vec{I}\to H\mid\vec{I}\in \mathcal{I}^{n-1}\rangle$$
such that
$$\mathrm{supp}\left(\sum_{i\leq n}(-1)^i\psi_{\vec{I}^i}\big|_{\bigcap\vec{I}}-\varphi_{\vec{I}}\right)\in\mathcal{J}$$
for all $\vec{I}\in \mathcal{I}^n$.

For any $\mathcal{I}$, $\mathcal{J}$, and $H$ as above, let $\mathbf{X}[\mathcal{I},\mathcal{J},H]=(X_I,\pi_{I,J}, \mathcal{I})$ denote the $\mathcal{I}$-indexed inverse system with terms
$$X_I=\langle\varphi:I\to H\mid\mathrm{supp}(\varphi)\in\mathcal{J}\rangle$$
and transition maps
$$\pi_{I,J}:X_J\to X_I:\varphi\mapsto \varphi\big|_I$$
for each $I\subseteq J$ in $\mathcal{I}$.
\end{definition}
\begin{example}
If $Y=\kappa\times\lambda$, $\mathcal{I}=\{X(f)\mid f\in{^\kappa}([\lambda]^{<\omega})\}$, and $\mathcal{J}=[Y]^{<\omega}$ then $\mathbf{X}[\mathcal{I},\mathcal{J},H]$ equals the system $\mathbf{A}_{\kappa,\lambda}[H]$ of Definition \ref{def:Akappalambda}.
\end{example}
\begin{lemma}
\label{lem:nontrivfamsarenontrivlimsagain}
For all $n>0$ and $\mathcal{I}$, $\mathcal{J}$, and $H$ as in Definition \ref{def:coh2} above,
$$\operatorname{lim}^n \mathbf{X}[\mathcal{I},\mathcal{J},H]\neq 0$$ if and only the associated $n$-coherent families of functions are all trivial.
\end{lemma}
\begin{proof}
The argument is essentially identical to that of Lemma \ref{lem:zeroiftriv}.\end{proof}
\begin{example}
For any fixed regular cardinal $Y=\lambda$, letting $H=\mathbb{Z}$, $\mathcal{J}=[Y]^{<\omega}$, and $\mathcal{I}$ equal the bounded ideal on $\lambda$ (within which the initial segments of $\lambda$ are cofinal) recovers some of the most fundamental set theoretic notions of coherence.
The $n=1$ case of Definition \ref{def:coh2} is that strongly associated to Todorcevic's walks apparatus \cite{Todorcevic_Walks_07}; moreover,
$$\operatorname{lim}^n \mathbf{X}[\mathcal{I},\mathcal{J},H]\cong\mathrm{H}^n(\lambda;\mathcal{H})$$
for all $n>0$, where the righthand terms are the sheaf cohomology groups of $\lambda$ (endowed with its order topology) with respect to the constant sheaf associated to $H$ (the argument is much as for \cite[Thm.\ 3.2]{IHW}).
\end{example}
The argument of Theorem \ref{thm:limnAkl_further} below will simultaneously involve notions of coherence from our previous two examples; we won't notationally belabor their distinction, though, as the dangers of confusing them are marginal.

Turning lastly to the theorem's set theoretic hypotheses: good references for the approachability ideal and the principle $\diamondsuit(S)$ are \cite{Eisworth} and \cite{Rinot_diamond}, respectively; that said, what we need of them will appear, simply, in the course of the theorem's proof.
We note that the hypotheses of the theorem are relatively weak. 
They all hold, for instance, in G\"{o}del's constructible universe 
$\mathrm{L}$ or various other canonical inner models. In addition, the 
failure of the approachability assumption in the main statement of the theorem 
has large cardinal strength precisely equal to a greatly Mahlo cardinal 
(see \cite{Mitchell_approachability}). In particular, if there fails to be 
a stationary subset of $S^{\aleph_{n+1}}_{\aleph_n}$ in $I[\aleph_{n+1}]$, then 
$\aleph_{n+1}$ must be greatly Mahlo in $\mathrm{L}$ (here, given 
cardinals $\kappa < \lambda$, $S^\lambda_\kappa$ denotes the set 
$\{\alpha < \lambda \mid \cf(\alpha) = \kappa\}$). 
The hypotheses of the theorem 
therefore hold, for instance, in the forcing extension of $\mathrm{L}$ formed by 
adding $\aleph_\omega$-many Cohen reals, a model in which, by results of 
\cite{Bergfalk_simultaneously_23} and \cite{Bannister_additivity_23}, we have
\[
  \mathrm{lim}^{n+1}\,\mathbf{A}_{\aleph_0, \aleph_0}\left[H\right]=0
\]
for every $n < \omega$ and every abelian group $H$.

\begin{theorem}
\label{thm:limnAkl_further}
  Suppose that $1 \leq n < \omega$ and there is a stationary subset $S \subseteq 
  S^{\aleph_{n+1}}_{\aleph_n}$ in the approachability ideal $I[\aleph_{n+1}]$.
  Then 
  \[
    \mathrm{lim}^{n+1}\,\mb{A}_{\aleph_n, \aleph_{n+1}}\left[\bigoplus_{\omega_{n+1}} 
    \bb{Z}\right] \neq 0.
  \]
  If, additionally, $\diamondsuit(S)$ holds and there is a nontrivial 
  $n$-coherent family $\langle \tau_{\vec{\eta}} : \bigwedge \vec{\eta} \rightarrow 
  \bb{Z} \mid \vec{\eta} \in (\omega_n)^n \rangle$, then 
  $\mathrm{lim}^{n+1}\, \mb{A}_{\aleph_n, \aleph_{n+1}} \neq 0$.
\end{theorem}

\begin{remark}
\label{rmk:cofinality}
Below, we will at times invoke what might be seen either by the description of the lim functor $\mathsf{Pro}(\mathcal{C})\to\mathcal{C}$ in Section \ref{subsect:condensedbackground} or, in any particular case, by hand, namely: \emph{only cofinal phenomena matter in analyses of $\mathrm{lim}^n$ or $n$-coherence.}
\end{remark}

\begin{proof}[Proof of Theorem \ref{thm:limnAkl_further}]
  Let $\mu := \aleph_n$, $H := \bigoplus_{\omega_{n+1}} \bb{Z}$, and 
  $\Lambda := {^\mu}([\mu^+]^{<\omega})$. We will begin by constructing 
  an $(n+1)$-coherent (modulo the ideal of finite subsets of $\mu \times \mu^+$) family $\Phi' = \langle \varphi'_{\vec{\beta}} : \mu \times 
  \bigwedge \vec{\beta} \ra H' \mid \vec{\beta} \in (\mu^+)^{n+1} \rangle$, where 
  $H'$ will be either $\bb{Z}$ or $H$, depending on whether or not we are 
  assuming the additional hypotheses in the statement of the theorem. Fix a 
  sequence $\langle c_\alpha \mid \alpha < \mu^+ \rangle$ witnessing 
  that $S \in I[\mu^+]$. In particular, we may assume that each 
  $c_\alpha$ is a subset of $\mu^+$ and, for all $\gamma \in S$, there is a 
  club $d_\gamma$ in $\gamma$ such that 
  \begin{itemize}
      \item $\otp(d_\gamma) = \mu$; and
      \item for all $\beta < \gamma$, there is an $\alpha < \gamma$ such that 
        $d_\gamma \cap \beta = c_\alpha$.
  \end{itemize}
  We can assume that $\otp(c_\alpha) < \mu$ for all $\alpha < \mu^+$.
  If $\diamondsuit(S)$ holds, then we may fix a sequence $\langle A_\gamma \mid 
  \gamma \in S \rangle$ such that
  \begin{itemize}
      \item for each $\gamma \in S$, $A_\gamma$ is of the form 
      \[
        \left\langle \psi^\gamma_b : \bigcap_{\alpha \in b} c_\alpha \ra \bb{Z} 
        ~ \middle| ~ b \in \gamma^n \right\rangle;
      \]
      \item for every sequence of the form 
      \[
        \left\langle \psi_b : \bigcap_{\alpha \in b} c_\alpha \ra \bb{Z} 
        ~ \middle| ~ b \in (\mu^+)^n \right\rangle,
      \]
      there are stationarily many $\gamma \in S$ such that $A_\gamma = 
      \langle \psi_b \mid b \in \gamma^n \rangle$.
  \end{itemize}

  We now define $\varphi'_{\vec{\beta}}$ for $\vec{\beta} \in (\mu^+)^{n+1}$. Since 
  we require $\Phi'$ to be alternating, it suffices to define $\varphi'_{\vec{\beta}}$ 
  for strictly increasing sequences $\vec{\beta} = \langle \beta_i \mid i \leq n \rangle$.
  We do so by recursion on $\beta_n$. For terminological convenience, let ``Case 1" denote 
  the case in which we are not assuming any additional hypotheses and $H' = \bigoplus_{\omega_{n+1}} 
  \bb{Z}$, and let ``Case 2" denote the case in which the additional hypotheses are being 
  assumed and $H' = \bb{Z}$. As part of our recursion hypotheses, we will require that, if we 
  are in Case 1, then, for all increasing $\vec{\beta} \in (\mu^+)^{n+1}$ and all 
  $(\eta, \alpha) \in \mu \times \beta_0$, we have $\varphi'_{\vec{\beta}}(\eta, \alpha) \in 
  \bigoplus_{\beta_n + \mu} \bb{Z}$. If we are in Case 1, fix a nontrivial $n$-coherent 
  family $\langle \tau_{\vec{\eta}} : \bigwedge \vec{\eta} \rightarrow \bigoplus_{\mu} 
  \bb{Z} \mid \vec{\eta} \in (\omega_n)^n \rangle$ (such exist by a slight modification of the argument of Theorem \ref{thm:nonzero_cohomology_on_opens} below; alternatively, see \cite[Theorem 7.6]{TFOA}). Recall that, in Case 2, we already have a nontrivial 
  $n$-coherent family $\langle \tau_{\vec{\eta}} : \bigwedge \vec{\eta} \rightarrow 
  \bb{Z} \mid \vec{\eta} \in (\omega_n)^n \rangle$. Given $\gamma < \mu^+$, we define a 
  ``shift homomorphism" $\sh_\gamma : \bigoplus_\mu \bb{Z} \ra \bigoplus_{\gamma + \mu} \bb{Z}$ 
  as follows: for all $x \in \bigoplus_\mu \bb{Z}$, let $\sh_\gamma(x)$ be the unique 
  $y \in \bigoplus_{\gamma + \mu} \bb{Z}$ such that $\supp(y) = \{\gamma + \eta \mid 
  \eta \in \supp(x)\}$ and, for all $\eta \in \supp(x)$, we have $y(\gamma + \eta) = x(\eta)$.
  
  Fix $\gamma < \mu^+$, and assume that we have defined $\langle \varphi'_{\vec{\beta}} \mid 
  \vec{\beta} \in \gamma^{n+1} \rangle$. We now define $\varphi'_{\vec{\beta}}$ for 
  all increasing $\vec{\beta}$ with $\beta_n = \gamma$. By Goblot's Theorem, we can fix a family $\langle \sigma^\gamma_{\vec{\beta}}
  : \beta_0 \ra H'
  \mid \vec{\beta} \in \gamma^n \rangle$ witnessing that $\langle \varphi'_{\vec{\beta}} \mid 
  \vec{\beta} \in \gamma^{n+1} \rangle$ is trivial. In Case 1, our recursion hypothesis 
  implies that, for each $\vec{\beta} \in \gamma^{n+1}$, $\varphi'_{\vec{\beta}}$ maps 
  into $\bigoplus_{\beta_n + \mu} \bb{Z}$, so we can assume that, for each $\vec{\beta} \in 
  \gamma^n$, $\sigma^\gamma_{\vec{\beta}}$ maps into $\bigoplus_{\xi_\gamma} \bb{Z}$, where 
  $\xi_\gamma = \sup\{\beta + \mu \mid \beta < \gamma\}$. In particular, $\xi_\gamma \leq \gamma 
  + \mu$ and, if $\cf(\gamma) = \mu$, then $\xi_\gamma = \gamma$. 
  If $\gamma \notin S$, then simply let $\varphi'_{\vec{\beta} ^\frown \langle \gamma 
  \rangle} := (-1)^n \sigma^\gamma_{\vec{\beta}}$ for all $\vec{\beta} \in \gamma^n$. It is 
  routine to verify that this maintains coherence and, in Case 1, our recursion hypothesis.

  If $\gamma \in S$, then we do a bit more work. Since $d_\gamma$ is cofinal in $\gamma$, 
  it will suffice by Remark \ref{rmk:cofinality} to define $\varphi'_{\vec{\beta} ^\frown \langle \gamma \rangle}$ for 
  $\vec{\beta} \in (d_\gamma)^n$. Let 
  $\langle \delta^\gamma_\eta \mid \eta < \mu \rangle$ be the increasing enumeration of $d_\gamma$.
  
  Suppose first that we are in Case 1. Fix $\vec{\beta} \in (d_\gamma)^n$, and let 
  $\vec{\eta} \in \mu^n$ be such that $\vec{\beta} = \langle \delta_{\eta_i} \mid i < n \rangle$.
  For each $(\xi, \alpha) \in \mu \times \beta_0$, let
  \[
    \varphi'_{\vec{\beta}^\frown \langle \gamma \rangle}(\xi, \alpha) := 
    \begin{cases}
        (-1)^n \sigma^\gamma_{\vec{\beta}}(\xi,\alpha) + \sh_\gamma(\tau_{\vec{\eta}}(\xi)) 
        & \text{if } \xi < \eta_0 \text{ and } \alpha = \delta^\gamma_\xi \\ 
        (-1)^n \sigma^\gamma_{\vec{\beta}}(\xi,\alpha) & \text{otherwise.}
    \end{cases}
  \]
  Again, the verification that this maintains coherence and our recursion hypothesis 
  is routine, if somewhat tedious.

  Suppose now that we are in Case 2. We begin by defining auxiliary families 
  $\bar{\Phi}^\gamma = \langle 
  \bar{\varphi}^\gamma_{\vec{\eta}} : \eta_0 \ra \bb{Z} \mid \vec{\eta} \in \mu^{n+1} 
  \rangle$, $\bar{\Sigma}^\gamma = \langle \bar{\sigma}^\gamma_{\vec{\eta}}:\eta_0 \ra \bb{Z} \mid 
  \vec{\eta} \in 
  \mu^n \rangle$, and $\bar{\Psi}^\gamma = \langle \bar{\psi}^\gamma_{\vec{\eta}}: \eta_0 \ra 
  \bb{Z} \mid \vec{\eta} \in \mu^n \rangle$. For the first, temporarily 
  fix $\vec{\eta} \in \mu^{n+1}$, and let $\vec{\beta} \in \gamma^{n+1}$ 
  be given by letting $\beta_i = \delta^\gamma_{\eta_i}$ for all $i \leq n$. 
  Now, for all $\xi < \eta_0$, let $\bar{\varphi}^\gamma_{\vec{\eta}}(\xi) := 
  \varphi'_{\vec{\beta}}(\xi, \delta^\gamma_\xi)$. For the latter two, fix 
  $\vec{\eta} \in \mu^n$, and let $\vec{\beta} \in \gamma^n$ be given by 
  letting $\beta_i = \delta^\gamma_{\eta_i}$ for all $i < n$. First, for all 
  $\xi < \eta_0$, let $\bar{\sigma}^\gamma_{\vec{\eta}}(\xi) = 
  \sigma^\gamma_{\vec{\beta}}(\xi, \delta^\gamma_\xi)$. Finally, to define 
  $\bar{\psi}^{\gamma}_{\vec{\eta}}$, define $b \in \gamma^n$ by letting 
  $b(i)$ be such that $d_\gamma \cap \beta_i = c_{b(i)}$ for all $i < n$, and 
  let $\bar{\psi}^{\gamma}_{\vec{\eta}}(\xi) = \psi^\gamma_b(\delta^\gamma_\xi)$ 
  for all $\xi < \eta_0$.

  By construction, $\bar{\Phi}^\gamma$ is $(n+1)$-coherent, and 
  $\bar{\Sigma}^\gamma$ trivializes $\bar{\Phi}^\gamma$. 
  If $\bar{\Psi}^\gamma$ does not trivialize $\bar{\Phi}^\gamma$, then 
  proceed exactly as in the case
  in which $\gamma \notin S$, i.e., set $\varphi'_{\vec{\beta}^\frown \langle \gamma \rangle} 
  := (-1)^n \sigma^\gamma_{\vec{\beta}}$ for all $\vec{\beta} \in \gamma^n$.
  Suppose now that $\bar{\Psi}^\gamma$ \emph{does} trivialize $\bar{\Phi}^\gamma$. 
  It follows that the family $\bar{\Sigma}^\gamma - \bar{\Psi}^\gamma = 
  \langle \bar{\sigma}^\gamma_{\vec{\eta}} - \bar{\psi}^\gamma_{\vec{\eta}} 
  \mid \vec{\eta} \in \mu^n \rangle$ is $n$-coherent. If $\bar{\Sigma}^\gamma - \bar{\Psi}^\gamma$ is nontrivial, then 
  again proceed as in the case in which $\gamma \notin S$. On the other hand, 
  if $\bar{\Sigma}^\gamma - \bar{\Psi}^\gamma$ is trivial, then proceed as 
  follows. Suppose that $\vec{\beta} \in (d_\gamma)^n$, and let $\vec{\eta} \in 
  \mu^n$ be such that $\beta_i = \delta^\gamma_{\eta_i}$ for all $i < n$. 
  Then, for all $(\xi, \alpha) \in \mu \times \beta_0$, set
  \[
    \varphi'_{\vec{\beta}^\frown \langle \gamma \rangle}(\xi,\alpha) := 
    \begin{cases}
        (-1)^n \sigma^\gamma_{\vec{\beta}}(\xi,\alpha) + \tau_{\vec{\eta}}(\xi) & 
        \text{if } \xi < \eta_0 \text{ and } \alpha = \delta^\gamma_\xi \\
        (-1)^n \sigma^\gamma_{\vec{\beta}}(\xi,\alpha) & \text{otherwise.}
    \end{cases}
  \]
  Once again, verifying coherence is routine. This completes the construction of 
  $\Phi'$.

  We now use $\Phi'$ to generate an $(n+1)$-coherent family $\Phi = \langle \varphi_{\vec{f}} 
  : X(\vec{f}) \ra H' \mid \vec{f} \in \Lambda^{n+1} \rangle$. For each $f \in \Lambda$, let 
  $\beta_f := \ssup(\bigcup\{f(\eta) \mid \eta < \mu\})$.\footnote{For a set of ordinals $x$, 
  $\ssup(x)$ denotes the \emph{strong supremum} of $x$, i.e., $\sup\{\alpha + 1 \mid \alpha 
  \in x\}$.} Given $\vec{f} \in \Lambda^{n+1}$, let $\vec{\beta}_{\vec{f}} \in 
  (\mu^+)^{n+1}$ be given by $\vec{\beta}_{\vec{f}}(i) = \beta_{f_i}$ for $i \leq n$. 
  Finally, let $\varphi_{\vec{f}} := \varphi'_{\vec{\beta}_{\vec{f}}} \restriction X(\vec{f})$.
  The $(n+1)$-coherence of $\Phi$ follows immediately from the $(n+1)$-coherence of 
  $\Phi'$. 

  It remains to show that $\Phi$ is nontrivial. Assume that $n > 1$; the 
  case $n = 1$ is similar but easier. Suppose for the sake of contradiction 
  that $\Psi = \langle \psi_{\vec{f}} : X(\vec{f}) \ra H' \mid \vec{f} \in 
  \Lambda^n \rangle$ witnesses that $\Phi$ is trivial. For each $\alpha < \mu^+$, 
  define a function $f_\alpha \in \Lambda$ by letting 
  \[
    f_\alpha(\eta) := 
    \begin{cases}
        \{c_\alpha(\eta)\} & \text{if } \eta < \otp(c_\alpha) \\
        \emptyset & \text{otherwise.}
    \end{cases}
  \]
  If $b \in (\mu^+)^{\leq n}$ is such that, for all $i < j < |b|$, we have 
  $c_{b(i)} \sqsubseteq c_{b(j)}$, then let $\vec{f}_b \in \Lambda^{|b|}$ 
  be defined by letting $\vec{f}_b(i) = f_{b(i)}$ for all $i < |b|$, and, if 
  $|b| = n$, then define 
  $\psi^*_b : c_{b(0)} \ra H'$ by letting $\psi^*_b(c_{b(0)}(\eta)) = 
  \psi_{\vec{f}_b}(\eta, c_{b(0)}(\eta))$ for all $\eta < \otp(c_{b(0)})$.
  For all other $b \in (\mu^+)^n$, simply let $\psi^*_b : \bigcap_{\alpha \in b} 
  c_\alpha \ra H'$ be the constant function, taking value $0$.

  Assume first that we are in Case 1. Define a function $h:(\mu^+)^n \ra 
  \mu^+$ by letting 
  \[
    h(b) = \sup\{\varepsilon \mid \exists \beta \in \bigcap_{\alpha \in b} 
    c_\alpha ~ [\varepsilon \in \supp(\psi^*_b(\beta)]\}
  \]
  for all $b \in (\mu^+)^n$. Note that this is well-defined, since each 
  $c_\alpha$ has cardinality less than $\mu$ and each $\psi^*_b$ maps into 
  $\bigoplus_{\mu^+} \bb{Z}$. Now fix $\gamma \in S$ such that 
  $h(b) < \gamma$ for all $b \in \gamma^n$. Recall that $\langle \delta^\gamma_\eta 
  \mid \eta < \mu \rangle$ is the increasing enumeration of $d_\gamma$. 
  Define a function $g \in \Lambda$ by letting $g(\xi) := \{\delta^\gamma_\xi\}$ 
  for all $\xi < \mu$, and note that $\beta_g = \gamma$. For each 
  $\eta < \mu$, define $g_\eta \in \Lambda$ by letting 
  \[
    g_\eta(\xi) = 
    \begin{cases}
        \{\delta^\gamma_\xi\} & \text{if } \xi < \eta \\
        \emptyset & \text{otherwise.}
    \end{cases}
  \]
  If $\eta < \mu$ is a limit ordinal, then $\beta_{g_\eta} = \delta^\gamma_\eta$; 
  moreover, there is $\alpha < \gamma$ such that $\{\delta^\gamma_\xi  \mid 
  \xi < \eta\} = c_\alpha$, and hence $g_\eta = f_\alpha$.

  Temporarily fix an increasing $\vec{\eta} \in (\lim(\mu))^n$, let 
  $\vec{g}_{\vec{\eta}} \in \Lambda^n$ be equal to $\langle g_{\eta_i} \mid 
  i < n \rangle$, and let $\vec{\beta} \in d_\gamma^n$ equal 
  $\langle \delta^\gamma_{\eta_i} \mid i < n \rangle$.  By construction, for 
  all $\xi < \eta_0$, we have 
  \[
    \varphi_{\vec{g}_{\vec{\eta}}{}^\frown \langle g \rangle}(\xi, \delta^\gamma_\xi) 
    = \varphi'_{\vec{\beta}^\frown \langle \gamma \rangle}(\xi, \delta^\gamma_\xi) 
    = (-1)^n \sigma^\gamma_{\vec{\beta}}(\xi,\delta^\gamma_\xi) + \sh_\gamma(\tau_{\vec{\eta}}(\xi)).
  \]
  By our assumption about $\Psi$, it follows that, for all but finitely many 
  $\xi < \eta_0$, we have
  \[
    (-1)^n \psi_{\vec{g}_{\vec{\eta}}}(\xi, \delta^\gamma_\xi) +
    \sum_{i < n}(-1)^n \psi_{\vec{g}_{\vec{\eta}}^i{}^\frown \langle g 
    \rangle}(\xi, \delta^\gamma_\xi) = (-1)^n \sigma^\gamma_{\vec{\beta}}(\xi,\delta^\gamma_\xi) + \sh_\gamma(\tau_{\vec{\eta}}(\xi)).
  \]
  By our choice of $\gamma$ and $\sigma^\gamma_{\vec{\beta}}$, respectively, 
  the supports of $\psi_{\vec{g}_{\vec{\eta}}}(\xi, \delta^\gamma_\xi)$ 
  and $\sigma^\gamma_{\vec{\beta}}(\xi,\delta^\gamma_\xi)$ are contained in 
  $\gamma$. It follows that, for all but finitely many $\xi < \eta_0$, we have
  \[
    \sum_{i < n}(-1)^n \psi_{\vec{g}_{\vec{\eta}}^i{}^\frown \langle \gamma 
    \rangle}(\xi, \delta^\gamma_\xi) \restriction [\gamma, \gamma + \mu) = 
    \sh_\gamma(\tau_{\vec{\eta}}(\xi))
  \]

  For each $\vec{\zeta} \in (\lim(\mu))^{n-1}$, let 
  $\vec{g}_{\vec{\zeta}} = \langle g_{\zeta_i} \mid i < n-1 \rangle$, and let
  $\nu_{\vec{\zeta}} : \zeta_0 \ra \bigoplus_{\mu} \bb{Z}$ be the unique function 
  such that
  \[
    \sh_\gamma(\nu_{\vec{\zeta}}(\xi) = \psi_{\vec{g}_{\vec{\zeta}}{}^\frown 
    \langle g \rangle}(\xi, \delta^\gamma_\xi) \restriction [\gamma, 
    \gamma + \mu).
  \]
  But then the calculations above show that the family $\langle 
  \nu_{\vec{\zeta}} \mid \vec{\zeta} \in (\lim(\mu))^{n-1} \rangle$ witnesses 
  that the family $\langle \tau_{\vec{\eta}} \mid \vec{\eta} \in 
  (\lim(\mu))^n \rangle$ is trivial. Since $\lim(\mu)$ is cofinal in $\mu$, 
  this implies that the entire family $\langle \tau_{\vec{\eta}} \mid 
  \vec{\eta} \in \mu^n \rangle$ is trivial, which is a contradiction.  

  Now assume that we are in Case 2, and fix a $\gamma \in S$ such 
  that $A_\gamma = \langle \psi^*_b \mid b \in \gamma^n \rangle$. 
  Define $g$ and $g_\eta$ for $\eta < \mu$ as in the previous case. Let 
  $\bar{\Phi}^\gamma$, $\bar{\Sigma}^\gamma$, and $\bar{\Psi}^\gamma$ be 
  as in stage $\gamma$ of the construction of $\Phi'$. By our choice of 
  $\gamma$, it follows that $\bar{\Psi}^\gamma$ trivializes 
  $\bar{\Phi}^\gamma$. Suppose first that $\bar{\Sigma}^\gamma - \bar{\Psi}^\gamma$ 
  is nontrivial. Fix $\vec{\eta} \in (\lim(\mu))^n$, and let $b \in \gamma^n$ be such that 
  $c_{b(i)} = \{\delta^\gamma_\xi \mid \xi < \eta_i\}$ for all $i < n$. Unraveling the 
  definitions and constructions yields the following facts.
    \begin{enumerate}
      \item For all $\xi < \eta_0$, we have $\psi^*_{b}
      (\delta^\gamma_\xi) = \bar{\psi}^\gamma_{\vec{\eta}}(\xi)$.
      \item For all $\xi < \eta_0$, we have $\varphi_{\vec{f}_b{}^\frown 
      \langle g \rangle}(\xi, \delta^\gamma_\xi) = 
      (-1)^n\bar{\sigma}^\gamma_{\vec{\eta}}(\xi)$.
      \item For all but finitely many $\xi < \eta_0$, we have 
        \[
          (-1)^n \bar{\psi}^\gamma_{\vec{\eta}}(\xi) + 
          \sum_{i < n} (-1)^i \psi_{\vec{f}_b^i{}^\frown \langle g 
          \rangle} (\xi, \delta^\gamma_\xi) = (-1)^n 
          \bar{\sigma}^\gamma_{\vec{\eta}}(\xi).
        \]
    \end{enumerate}
    For $\vec{\zeta} \in (\lim(\mu))^{n-1}$, define $\nu_{\vec{\zeta}}:\zeta_0 \ra 
    \bb{Z}$ as follows. Let $\vec{g}_{\vec{\zeta}} = \langle g_{\zeta_i} \mid 
    i < n-1 \rangle$, and let $\nu_{\vec{\zeta}}(\xi) = 
    \psi_{\vec{g}_{\vec{\zeta}}{}^{\frown} \langle g \rangle}(\xi, \delta^\gamma_\xi)$
    for all $\xi < \zeta_0$. Then item (3) above implies that 
    $\langle (-1)^n \nu_{\vec{\zeta}} \mid \vec{\zeta} \in (\lim(\mu))^{n-1} \rangle$ 
    trivializes $\bar{\Sigma}^\gamma - \bar{\Psi}^\gamma$ restricted to 
    $\lim(\mu)$, which is a 
    contradiction.

    Finally, assume that $\bar{\Sigma}^\gamma - \bar{\Psi}^\gamma$ is trivial, 
    as witnessed by $\langle \rho_{\vec{\zeta}} \mid \vec{\zeta} \in 
    \mu^{n-1} \rangle$. Fix $\vec{\eta} \in (\lim(\mu))^n$ and $b \in \gamma^n$ as 
    in the previous paragraph. Item (1) of that paragraph is unchanged, and 
    items (2) and (3) become
    \begin{enumerate}
      \item[(2')] For all $\xi < \eta_0$, we have $\varphi_{\vec{f}_b{}^\frown 
      \langle g \rangle}(\xi, \delta^\gamma_\xi) = 
      (-1)^n\bar{\sigma}^\gamma_{\vec{\eta}}(\xi) + \tau_{\vec{\eta}}(\xi)$.
      \item[(3')] For all but finitely many $\xi < \eta_0$, we have 
        \[
          (-1)^n \bar{\psi}^\gamma_{\vec{\eta}}(\xi) + 
          \sum_{i < n} (-1)^i \psi_{\vec{f}_b^i{}^\frown \langle g 
          \rangle} (\xi, \delta^\gamma_\xi) = (-1)^n 
          \bar{\sigma}^\gamma_{\vec{\eta}}(\xi) + \tau_{\vec{\eta}}(\xi).
        \]
    \end{enumerate}
    Rearranging item (3') and using the triviality of $\bar{\Sigma}^\gamma 
    - \bar{\Psi}^\gamma$, we see that, for all but finitely many $\xi < \eta_0$, 
    we have
    \[
      (-1)^{n-1} \sum_{i < n} (-1)^i \rho_{\vec{\eta}^i}(\xi) + 
      \sum_{i < n} (-1)^i \psi_{\vec{f}_b^i{}^\frown \langle g 
      \rangle} (\xi, \delta^\gamma_\xi) = \tau_{\vec{\eta}}(\xi).
    \]
    For all $\vec{\zeta} \in (\lim(\mu))^{n-1}$, define $\nu_{\vec{\zeta}}:\zeta_0 \ra 
    \bb{Z}$ as follows. Let $\vec{g}_{\vec{\zeta}} = \langle g_{\zeta_i} \mid 
    i < n-1 \rangle$, and let $\nu_{\vec{\zeta}}(\xi) = (-1)^{n-1} 
    \rho_{\vec{\zeta}}(\xi) + \psi_{\vec{g}_{\vec{\zeta}}{}^\frown \langle 
    g \rangle}(\xi, \delta^\gamma_\xi)$ for all $\xi < \zeta_0$. Then 
    the above calculation implies that $\langle \nu_{\vec{\zeta}} \mid 
    \vec{\zeta} \in (\lim(\mu))^{n-1} \rangle$ trivializes $\langle \tau_\eta \mid 
    \vec{\eta} \in (\lim(\mu))^n \rangle$, yielding the final contradiction.
\end{proof}

\begin{corollary}
  If $\mathrm{V} = \mathrm{L}$, then 
  $\mathrm{lim}^{n+1}\, \mb{A}_{\aleph_n, \aleph_{n+1}}\neq 0$ for all $n<\omega$.
\end{corollary} 

\begin{proof}
The case of $n=0$ is the \textsf{ZFC} result Theorem \ref{thm:limsofAkappalambda}(5).
For higher $n$, the relevant conditions of Theorem \ref{thm:limnAkl_further} all hold in G\"{o}del's constructible universe $\mathrm{L}$. To see this, first note that $
\mathsf{GCH}$ holds in $\mathrm{L}$ and therefore, by \cite[Theorem 3.40]{Eisworth}, 
$S^{\aleph_{n+1}}_{\aleph_n} \in I[\aleph_{n+1}]$ for all $1 \leq n < \omega$. By 
a theorem of Jensen \cite{jensen_constructible}, $\diamondsuit(S^{\aleph_{n+1}}_{\aleph_n})$ holds in $\mathrm{L}$. Finally, there exists a nontrivial 
  $n$-coherent family $\langle \tau_{\vec{\eta}} : \bigwedge \vec{\eta} \rightarrow 
  \bb{Z} \mid \vec{\eta} \in (\omega_n)^n \rangle$ in $\mathrm{L}$ by \cite[Cor. 3.27]{BLHCoOI}.
\end{proof}





\section{The additivity of derived limits and of strong homology}
\label{sec:additivity}

\subsection{Additivity}
\label{subsect:additivity}

The question of the additivity of strong homology has, over the past 35 years, attracted a substantial literature; in chronological order, 
main results therein include the following:
\begin{enumerate}
\item The continuum hypothesis implies that strong homology is not additive, 
not even on the category of closed subsets of $\mathbb{R}^2$ \cite{Mardesic_additive_88}.
\item There exist non-metrizable \textsf{ZFC} counterexamples to the additivity of strong homology \cite{Prasolov_non_05}.
\item It is consistent with the \textsf{ZFC} axioms that strong homology is additive on the category of locally compact separable metric spaces \cite{Bannister_on_22}.
\end{enumerate}
A primary aim of this section is to outline a more direct and cohesive approach to these theorems: each may be regarded as a result about the additivity of \emph{higher derived limits}, one moreover which is instantiated by the groups $\mathrm{lim}^n\,\mathbf{A}_{\kappa,\lambda}$ for suitable cardinals $\kappa$ and $\lambda$.
We will show more precisely that items (1), (2) and (3) above correspond in strong senses to items (3), (5), and (2) of Theorem \ref{thm:limsofAkappalambda}, respectively; each in this view is at heart a statement of infinitary combinatorics.

For concreteness, we begin by recalling what it means for a functor to be additive. We record also a definition of strong homology which we elucidate in this section's conclusion; readers seeking a fuller treatment of the subject are referred to \cite{Mardesic_additive_88} or \cite{Mardesic_strong_00}.

\begin{definition}
A homology theory $\mathrm{H}_q:\mathsf{Top}\to\mathsf{Ab}$ $(q\geq 0)$ is \emph{additive} if for every natural number $q$ and every family $\{X_\alpha\,|\,\alpha\in A\}$ of topological spaces, the map
 $$i^{*}_q:\bigoplus_A \mathrm{H}_q(X_\alpha)\rightarrow \mathrm{H}_q\Big(\coprod_A X_\alpha\Big)$$
 induced by the inclusion maps $i_\alpha:X_\alpha\hookrightarrow\coprod_A X_\alpha$ is an isomorphism.
 
Similarly, a functor $F:\mathsf{Pro(Ab)}\to\mathsf{Ab}$ is \emph{additive} if for every family $\{\mathbf{X}_\alpha\,|\,\alpha\in A\}$ of inverse systems of abelian groups, the map
$$i^{*}:\bigoplus_A F(\mathbf{X}_\alpha)\rightarrow F\Big(\coprod_A \mathbf{X}_\alpha\Big)$$
 induced by the inclusion maps $i_\alpha:\mathbf{X}_\alpha\hookrightarrow\coprod_A \mathbf{X}_\alpha$ is an isomorphism.

Adaptations of these definitions to subcategories $\mathcal{C}$ of the source categories $\mathsf{Top}$ and $\mathsf{Pro(Ab)}$ are straightforward; note, however, that additivity in these contexts is only a question of those sums of $\mathcal{C}$-objects which also fall within $\mathcal{C}$.

The \emph{strong homology group} $\overline{\mathrm{H}}_q$ of a topological space $X$ is the $q^{\mathrm{th}}$ homology group of the homotopy limit of the system of chain complexes given by the application of the singular chain functor to any representative of the \emph{strong shape} $\mathsf{sSh}(X)\in\mathsf{Ho}(\mathsf{Pro(}\mathsf{Pol}))$ of $X$ \cite{CordierStrong}; in plainer English, it is the most programmatically \emph{strong shape invariant} extension of \emph{Steenrod homology} from the category of metric compacta to that of all topological spaces.
\end{definition}

We turn now to our central examples.
For $n>0$ let $B^n$ denote an $n$-dimensional open ball, and let $S^n$, as usual, denote an $n$-dimensional sphere.
For any cardinal $\lambda>0$ let $\mathbf{Y}^{n,\lambda}$ denote the inverse system $(Y^{n,\lambda}_a,p_{ab},[\lambda]^{<\omega}\backslash\{\varnothing\})$ in which for all nonempty $a\subseteq b$ in $[\lambda]^{<\omega}$ the space $Y^{n,\lambda}_a$ is the $a$-indexed wedge sum $\bigvee_{a}S^n$ and the bonding map $p_{ab}:Y^{n,\lambda}_b\to Y^{n,\lambda}_a$ is the identity on $\bigvee_a S^n$, sending all other points in $\bigvee_b Y^{n,\lambda}$ to the wedge's basepoint.
\begin{lemma}
\label{lemma:resolution}
For any $n,\lambda>0$ the space $Y^{n,\lambda}=\mathrm{lim}\,\mathbf{Y}^{n,\lambda}$ is homeomorphic to the one-point compactification of $\coprod_\lambda B^n$; in consequence, the system $\mathbf{Y}^{n,\lambda}$ defines a resolution of this space in the sense of \cite{Mardesic_strong_00} or \cite{Mardesic_additive_88}.
\end{lemma}
\begin{proof}
The second claim follows from the first together with \cite[Thm. 6.20]{Mardesic_strong_00}. For the first, observe simply that $Y^{n,\lambda}$ is compact, that the $\alpha^{\mathrm{th}}$ summand of $\coprod_\lambda B^n$ naturally identifies with the complement of the basepoint in $Y^{n,\lambda}_{\{\alpha\}}$, and that the complement of the union of these $B^n$-images in $Y^{n,\lambda}$ is exactly one point.
\end{proof}
When $\lambda=\aleph_0$, of course, $Y^{n,\lambda}$ is a \emph{compact bouquet of $n$-spheres}, familiarly known as an \emph{$n$-dimensional Hawaiian earring}.
The following theorem generalizes the computations in \cite{Mardesic_additive_88} of the $\lambda=\aleph_0$ case to arbitrary cardinalities.
\begin{theorem}
\label{thm:homology_groups}
For any nonzero cardinals $\kappa$ and $\lambda$, let $\mathbf{A}_{\kappa,\lambda}$ denote the inverse system of Definition \ref{def:Akappalambda} above. Then for any $n>0$, the strong homology groups of the space $Y^{n,\lambda}$ and its sums are as follows:
\begin{equation*}
\overline{\mathrm{H}}_q(Y^{n,\lambda})=
    \begin{cases}
        0 & \text{if } q\neq 0,n\\
        \prod_\lambda\mathbb{Z} & \text{if } q=n\\
        \mathbb{Z} & \text{if } q=0
    \end{cases}
\end{equation*}
and
\begin{equation*}
\overline{\mathrm{H}}_q\left(\coprod_\kappa Y^{n,\lambda}\right)=
    \begin{cases}
        0 & \text{if } q>n\\
        \mathrm{lim}^{n-q}\,\mathbf{A}_{\kappa,\lambda} & \text{if } 0<q\leq n\\
        \mathrm{lim}^n\,\mathbf{A}_{\kappa,\lambda}\oplus\bigoplus_\kappa\mathbb{Z} & \text{if } q=0.
    \end{cases}
\end{equation*}
In particular, if strong homology is additive then $\mathrm{lim}^s\,\mathbf{A}_{\kappa,\lambda}=0\textit{ for all }s,\kappa,\lambda>0$.
\end{theorem}
\begin{proof}
The last assertion follows from the first simply by letting $n$ range above $0$ and letting $q$ range between them.
For the first, observe that by Lemma \ref{lemma:resolution} the arguments of \cite[\S 5-6]{Mardesic_additive_88} fully generalize to arbitrary cardinalities $\kappa$ and $\lambda$, with only the partial exception of the (implicit) invocation of the Mittag-Leffler principle at the top of page 736.
In our context, the key point at that step is that for every $q\geq 0$ the derived limits $\mathrm{lim}^s$ of the system of singular homology groups $\mathrm{H}_q(Y^{n,\lambda}_b)$ $(b\in [\lambda]^{<\omega})$ associated to $\mathbf{Y}^{n,\lambda}$ vanish in all degrees $s>0$. When $q\neq n$, though, this is clear, since these groups are either constantly $0$ or $\mathbb{Z}$, while if $q=n$, this system is $\mathbf{A}_{1,\lambda}$, whereupon the assertion follows immediately from item (1) of Theorem \ref{thm:limsofAkappalambda}.
\end{proof}

The main result of \cite{Prasolov_non_05}, and our Theorem C, 
now follow as an immediate corollary:

\begin{corollary}
\label{cor:notadd}
There exists a non-metrizable \textsf{ZFC} counterexample to the additivity of strong homology.
\end{corollary}

\begin{proof}
Our counterexample is a countable sum of copies of $Y^{n,\aleph_1}$. By Theorem \ref{thm:homology_groups}, whenever $p=n-1$, the additivity of strong homology over this sum would require that $\mathrm{lim}^1\,\mathbf{A}_{\aleph_0,\aleph_1}=0$, but we know from Theorem \ref{thm:limsofAkappalambda} that this limit is nonzero.
\end{proof}
As noted, $Y^{n,\aleph_1}$ is a significantly simpler counterexample to additivity than that appearing in \cite{Prasolov_non_05} (which had derived, in turn, from \cite{Mardesic_nonvanishing_96}); it is also, in contrast to that counterexample, compact. See our conclusion for the question of whether \emph{metrizable} \textsf{ZFC} counterexamples to additivity exist.
There exists no \emph{locally compact separable metric} \textsf{ZFC} counterexample to the additivity of strong homology since, by \cite{Bannister_on_22}, the only obstructions to additivity on this class of spaces are any nonvanishing derived limits $\mathrm{lim}^n$ $(n>0)$ of a family of inverse systems which sufficiently resemble $\mathbf{A}_{\aleph_0,\aleph_0}$ so that their higher limits all vanish for the same reasons that the latter's do in the model of \cite{Bergfalk_simultaneously_21} (and likewise in the model of \cite{Bergfalk_simultaneously_23}, by \cite{Bannister_additivity_23}).
When $\mathrm{lim}^1\,\mathbf{A}_{\aleph_0,\aleph_0}\neq 0$, on the other hand, strong homology fails to be additive even on closed subspaces of $\mathbb{R}^2$, as witnessed by a countable sum of earrings $Y^{1,\aleph_0}$, and it was the derivation of this scenario from the continuum hypothesis which formed the main result of \cite{Mardesic_additive_88}.

These examples, like those of \cite{Prasolov_non_05}, may be regarded as materializing a more fundamental failure of additivity, namely one at the level of the derived limit functors.
More particularly, \cite[Theorem 1]{Prasolov_non_05} is also immediate within our framework, this time from Theorem \ref{thm:limsofAkappalambda}.
\begin{corollary}
\label{cor:nonadditive}
The functor $\mathrm{lim}^n:\mathsf{Pro(Ab)}\to\mathsf{Ab}$ is not additive for either $n=1$ or $n=2$.
\end{corollary}
Note in contrast that the functor $\mathrm{lim}^0=\mathrm{lim}:\mathsf{Pro(Ab)}\to\mathsf{Ab}$ \emph{is} additive, as follows from the observation beginning the proof below.\footnote{The additivity of \v{C}ech homology is argued in \cite[Theorem 9]{Mardesic_additive_88}; the proof applies in its essentials to the functor $\mathrm{lim}$ as well, but readers may find it at least as efficient to supply the argument themselves.}
\begin{proof}
The essential observation is that the sum, in a pro-category, of a set of inverse systems is represented by a system indexed by the product of their index sets (with terms the naturally associated sum of objects; see \cite[pp. 502-503]{Prasolov_non_05} for details). In particular, $\mathbf{A}_{\kappa,\lambda}=\coprod_\kappa\mathbf{A}_{1,\lambda}$ for all cardinals $\kappa$ and $\lambda$, so the inequality $\mathrm{lim}^1\,\mathbf{A}_{\aleph_0,\aleph_1}\neq 0 = \bigoplus_\omega\mathrm{lim}^1\,\mathbf{A}_{1,\aleph_1}$ witnesses the $n=1$ instance of the corollary.
For the $n=2$ case, let $\mathbf{C}_{\kappa,\lambda}=(C_f,q_{fg},{^\kappa}([\lambda]^{<\omega})$ and $\mathbf{D}_{\kappa,\lambda}=(D_f,r_{fg},{^\kappa}([\lambda]^{<\omega})$, where
$$C_f=\bigoplus_{(\kappa\times\lambda)\backslash X(f)}\mathbb{Z}$$
and $D_f=\bigoplus_{\kappa\times\lambda}\mathbb{Z}$, and the morphisms $q_{fg}$ and $r_{fg}$ are the natural inclusion and identity maps, respectively.
Note that $\mathbf{D}_{\kappa,\lambda}$ is flasque, so that the short exact sequences $\mathbf{C}_{\kappa,\lambda}\to\mathbf{D}_{\kappa,\lambda}\to\mathbf{A}_{\kappa,\lambda}$ induce isomorphisms $\mathrm{lim}^1\,\mathbf{A}_{\kappa,\lambda}\to\mathrm{lim}^2\,\mathbf{C}_{\kappa,\lambda}$, converting the previous inequality to $\mathrm{lim}^2\,\mathbf{C}_{\aleph_0,\aleph_1}\neq 0 = \bigoplus_\omega\mathrm{lim}^2\,\mathbf{C}_{1,\aleph_1}$.
\end{proof}
An oddity of these results, of course, is their restriction to the degrees $n\leq 2$. By Goblot's Theorem, results for higher $n$ will be, just as in Section \ref{sec:pro}, at least in part about the combinatorics of $\aleph_{n-1}$, and this indeed is much of the interest of Question \ref{ques:add_of_lims} below.
In particular, the question of whether the functor $\mathrm{lim}^3:\mathsf{Pro(Ab)}\to\mathsf{Ab}$ is consistently additive appears to be closely related to the question of whether the group $\mathrm{H}^2(\omega_2;\mathbb{Z})$ can vanish in any model of the \textsf{ZFC} axioms.\footnote{See \cite[Question 9.3]{TFOA} and \cite[Question 7.7]{Bergfalk_simultaneously_23} and \cite[Question 11.2]{IHW} for general formulations of this question, which is one of the most important in this research area.}

Of comparable interest is the question \emph{on what classes $\mathcal{C}$ may the functors $\mathrm{lim}^n$ be additive?} After all, the \textsf{ZFC} obstructions recorded in Theorem \ref{thm:limsofAkappalambda} and Corollary \ref{cor:nonadditive} only arise for systems of ``height'' $\lambda \geq \aleph_1$; it is natural, in consequence, to wonder if $\mathrm{lim}^n$ may nevertheless be additive for restricted or unrestricted sums of towers (i.e., countable inverse sequences) of abelian groups, and it's indeed in these terms that the theorem \cite[Thm.\ 1.2]{Bannister_additivity_23} referenced in Section \ref{subsect:alternative} is framed.
In fact (as noted in \cite[\S 7]{Bannister_additivity_23}), this result coupled with the argument of \cite[Thm.\ 5.1]{Be17} implies the consistency of $\mathrm{lim}^1\,\mathbf{A}_{\kappa,\aleph_0}[H]=0$ for all $\kappa$ and abelian groups $H$, and the question of whether similar implications hold in higher degrees appears to be a deep one; for recent developments, see \cite{Bannister_All, Bergfalk_Casarosa}.


As indicated, for whatever light it may shed on these and related computations, we conclude this section with a broader discussion of strong homology; readers uninterested in the latter should proceed without delay to Section \ref{sect:products}.

\subsection{Strong homology}
\label{subsect:stronghomology}

Works on the additivity problem have, following \cite{Mardesic_additive_88} and \cite{Mardesic_strong_00}, tended to emphasize the axiomatic and, to a lesser degree, computational features of strong homology, somewhat eliding more conceptual formulations in the process (see \cite[\S 1]{Bannister_on_22} for a recent example).
The cumulative effect has been a distorted picture, one in which strong homology figures mainly (1) as a sort of compromise-formation between the requirements of a homology functor and the unruly realities of geometric topology (represented by a standard roster of motivating examples of ``bad local behavior'': the Warsaw circle, Hawaiian earrings, solenoids, and so on), or (2) as an elaborate, if not \emph{ad hoc}, exactness-recovering corrective to \v{C}ech homology, or (3) as both.
There is an element of truth to each of these images, but taken alone, they are seriously misleading as well, insofar as strong homology figures within them as essentially isolated from the more structural and homotopy theoretic emphases of the contemporary mainstream of algebraic topology.
The reality, though, is that those emphases and strong homology \emph{grew up together} in large part out of the arrival, through the 1970s, of model category, homotopy (co)limit, and homotopy coherent technologies \cite{BoardmanVogt, BousfieldKan, CordierCoherent}, and it is in these terms, at least as much as (1) and (2) above, that strong homology should be understood; that \emph{strong shape}, the background homotopy theory of strong homology, is a conspicuously $\infty$-categorical entity is just one manifestation of this heritage (see, e.g., \cite[\S 7.1.6] {LurieHTT} and \cite{GuntherSemi}).
We briefly sketch a more balanced, but nevertheless classical, account below.

Surely the best-known general invariant of a topological space $X$ is its \emph{homotopy type}; put differently, the most prominent relation refining the identity relation on topological spaces is that of \emph{homotopy equivalence}.
For all its virtues, this relation is, from some perspectives, weak; natural extensions of this relation determine further invariants of a topological space $X$, such as its \emph{strong shape} and its \emph{shape}.\footnote{We note in passing that these extensions are, in spirit, orthogonal to that of \emph{weak homotopy equivalence}: while the latter identifies the Warsaw circle, for example, with a point, the former identify it with a circle. These more intuitive identifications are the source both of the name \emph{shape} and of much initial interest in its theory, concerns emphasized in \cite{CordierPorter} and recapitulated (alongside many fundamental techniques) in the more contemporary field of persistent homology.}
More precisely, we have a sequence of categories
$\mathsf{Ho(Top)}\to\mathsf{sSh(Top)}\to\mathsf{Sh(Top)}$
corresponding to the aforementioned invariants, respectively, together with commuting functors from $\mathsf{Top}$ to each of them.
The middle of these admits multiple incarnations, or models (within the \emph{coherent homotopy category} $\mathsf{CH}(\mathsf{Pro}(\mathsf{Top}))$, for example; see \cite[pp. 90, 2--3, and Ch.\ 1]{Mardesic_strong_00}); most classical among these is that of \cite{EdwardsHastings}, whereby the above diagram takes the form
\begin{equation}
\label{eq:shape_categories}
\mathsf{Ho(Top)}\xrightarrow{F}\textsf{Ho(Pro(Top))}\xrightarrow{G}\textsf{Pro(Ho(Top))},\end{equation}
with each of the latter two terms naturally construed as categories of \emph{systems of approximations to a topological space} $X$.\footnote{To be clear, the homotopy category of the first and third terms is, in the terminology of \cite[p.\ 395]{MayPonto}, the ``naive'' one, corresponding to the Hurewicz model structure on the category of topological spaces; that of the middle term derives from either of the classes of weak equivalences discussed in \cite{PorterTwo}. For the functor $F$, see \cite[\S 8]{Mardesic_strong_00}; on objects, $G$ factors through the natural map $\textsf{Pro}(\mathsf{Top})\to\mathsf{Pro}(\mathsf{Ho(Top)})$.}
Such systems are only informative, of course, when they take values in some restricted class (or \emph{dense subcategory}) of ``nice'' topological spaces, the categories $\mathsf{Pol}$ of polyhedra or of absolute neighborhood retracts (ANRs) being standard historical choices. And indeed, the \emph{Vietoris nerve} and \emph{\v{C}ech nerve} constructions on the objects of $\mathsf{Top}$ may each naturally be viewed as functors to the $\textsf{Ho(Pro(Pol))}$ and $\textsf{Pro(Ho(Pol))}$ subcategories, respectively, of the incarnations of $\mathsf{sSh(Top)}$ and $\mathsf{Sh(Top)}$ in equation (\ref{eq:shape_categories}) above (see \cite{Gunther}).
As is well known, within the \v{C}ech construction, the plurality of maps witnessing the refinement relation between two open coverings of a space induce a single homotopy class of maps between the associated polyhedra, and it is for this reason that the shape functor takes values in the pro-category of $\mathsf{Ho(Top)}$.

The further passage to a homology theory then depends on the assignment of (homotopy classes of) chain complexes to the objects of the terms of equation (\ref{eq:shape_categories}).
The canonical choice in the case of $\mathsf{Ho(Top)}$ is, of course, the \emph{singular chain} functor $X\mapsto C^{\mathrm{sing}}_\bullet(X)$, the crucial point here being that $C^{\mathrm{sing}}_\bullet$ maps homotopy equivalent spaces to homotopy equivalent chain complexes; the homology groups of $C^{\mathrm{sing}}_\bullet(X)$ are the \emph{singular homology groups} of $X$.
On the other hand, applying $C^{\mathrm{sing}}_\bullet$ to the $\mathsf{Top}$ and $\mathsf{Ho(Top)}$ terms of the two subsequent terms of line (\ref{eq:shape_categories}) yields objects in the homotopy category of the pro-category of chain complexes and the pro-category of the homotopy category of chain complexes, respectively. The natural means for consolidating these systems into a single graded abelian group are the \emph{homotopy inverse limit} and the \emph{inverse limit}, respectively --- although, limits not in general existing in homotopy categories, it's in the latter case standard to interchange steps, i.e., to take the limit of the homology groups of the chain complexes of the elements of the \v{C}ech nerve of $X$; the resultant groups are the \emph{\v{C}ech homology groups} of $X$.\footnote{This parallels, of course, the standard construction of \v{C}ech cohomology groups; that the latter are both shape and strong shape invariant is essentially immediate.}

The \emph{strong homology groups} of a topological space $X$ are, in contrast, the homology groups of the chain complex arising from the application of the composition of the \emph{homotopy inverse limit} (holim) and \emph{singular} functors to the strong shape of $X$; to sum up, for any fixed space $X$ and $q\geq 0$,
\begin{align*}
\mathrm{H}_q(X) & = h_q(C^{\mathrm{sing}}_\bullet(X)),\\
\overline{\mathrm{H}}_q(X) & = h_q(\mathrm{holim}(C^{\mathrm{sing}}_\bullet(sS(X)))),\textnormal{ and }\\
\check{\mathrm{H}}_q(X) & = \mathrm{lim}\,h_q(C^{\mathrm{sing}}_\bullet(S(X))),\end{align*}
where $\mathrm{H}_q$ denotes singular homology, $S:\mathsf{Top}\to\mathsf{Sh(Top)}$ and $sS:\mathsf{Top}\to\mathsf{sSh(Top)}$ denote the shape (see \cite[I.4.2, Thm.\ 3]{ShapeTheory}) and strong shape functors, respectively, and $h_q$ denotes the $q^{\mathrm{th}}$ homology group of a chain complex.
Cordier (who characterizes the $\mathrm{holim}$ and $\mathrm{lim}^s$ functors as complementary responses to the ``deficiencies'' of $\mathrm{lim}$ \cite[p.\ 35]{CordierStrong}) derives the existence of a homotopy limit in the category of non-negative chain complexes of abelian groups from the Dold-Kan correspondence together with its existence in the category of simplicial abelian groups; he shows moreover that it takes the concrete form of precisely that total complex figuring so prominently in Chapter 4 of \cite{Mardesic_strong_00}.

Let us conclude with a few summary remarks.
\begin{itemize}
\item Generalized framings of the strong shape functor are of a wider significance than we have so far suggested; see in particular Hoyois's \emph{Higher Galois theory} \cite[Def.\ 2.3, Rmk.\ 2.13]{Hoyois}, where this functor, ranging over $\infty$-topoi and denoted $\Pi_\infty$, carries the name \emph{fundamental pro-$\infty$-groupoid}; shape functors arise naturally in both the condensed and pyknotic settings as well, and are areas of active research.\footnote{See forthcoming work by Mair; see also \cite[II.4]{exodromy} for further references.}


\item Recall that Eilenberg and Steenrod showed, via a geometric realization of a simple inverse sequence of groups possessing a nonvanishing $\mathrm{lim}^1$ ``that the \v{C}ech `homology theory' with integer coefficients is not exact,'' or more precisely that \emph{no integral homology theory on the category of compact pairs is simultaneously continuous (meaning that it commutes with inverse limits) and exact} \cite[p.\ 265]{EilenbergSteenrod}.
Consider then what is sometimes framed (in a sense made precise by a Milnor short exact sequence \cite[Thm.\ 21.9]{Mardesic_strong_00}) as rectifying the non-exactness of $\check{\mathrm{H}}_\bullet$, namely the strong homology functor $\overline{\mathrm{H}}_\bullet$.
On the category of compact metric spaces, $\overline{\mathrm{H}}_\bullet$ is fully axiomatized by the Eilenberg-Steenrod axioms together with a relative homeomorphism and cluster axiom \cite{Milnor1960}; it is in consequence thereon equivalent to the homology theories of Steenrod \cite{SteenrodCycles}, Borel-Moore \cite{BorelMoore}, Sitnikov, and others \cite{Massey78}.
Where these theories may differ is in their extension to the broader category of locally compact spaces; of the central question of whether the Steenrod-Sitnikov theory embodying the most naive strategy of extension (direct limits) is strong shape invariant, Sklyarenko writes that it ``will likely be a test not only for Steenrod-Sitnikov homology, but also for strong shape theory itself'' (\cite{Sklyarenko}, quoted in \cite{Melikhov}).
It is reasonable to wonder whether materializations like $\coprod_\omega Y^{n,\aleph_1}$ of nonvanishing derived limits may, much as in the classical \v{C}ech case, play some role in the answer.
\end{itemize}

\section{Products of compact projectives}
\label{sect:products}

Recall the introductory account of condensed mathematics concluding Section \ref{subsect:condensedbackground} above.
For motivation, as is standard, we began by emphasizing the shortcomings of, for example, the category of Hausdorff topological abelian groups, together with its faithful embedding into the abelian category $\mathsf{Cond(Ab)}$ as one far-reaching kind of remedy.
Note, though, that the provision of an abelian --- and even a symmetric monoidal abelian --- category structure extending that of some ill-behaved preabelian category isn't, in and of itself, particularly novel: Yoneda embeddings into abelian sheaf categories will always have these features (cf.\ comments, e.g., at \cite[p.\ 9]{CS2}).
The distinction of condensed categories begins with something further, namely the conjunction of such features with the existence of a class of \emph{compact projective generators}; we record the relevant definitions just below.
Let us more immediately recall, though, that this class derives from the existence \emph{in the underlying category \textsf{ProFin} of the condensed site} of a generating class of projective objects, and that the latter, as noted in Section \ref{subsect:condensedbackground} above, are precisely the extremally disconnected compact Hausdorff spaces, or $\mathsf{ED}$ spaces, for short.
Note for future reference that the $\mathsf{ED}$ spaces also admit characterization as the retracts of the \v{C}ech--Stone compactifications $\beta X$ of discrete spaces $X$ (these $\beta X$ are accordingly termed the \emph{free} objects of $\mathsf{CHaus}$ in \cite{Rainwater}); recall as well the description of any such $\beta X$ as $(\{\mathcal{U}\subseteq P(X)\mid\mathcal{U}\textnormal{ is an ultrafilter on }X\},\tau)$,
where the topology $\tau$ on $\beta X$ is generated by the sets $N_I:=\{\mathcal{U}\in\beta X\mid I\in\mathcal{U}\}$ for $I\subseteq X$.

Returning to our more general discussion, let us quote from \cite{CS2} the basic definitions:
\begin{quote}
Let $\mathcal{C}$ be a category that admits all small colimits. Recall that an object $X\in\mathcal{C}$ is compact (also called finitely presented) if $\mathrm{Hom}(X,-)$ commutes with filtered colimits. An object $X\in\mathcal{C}$ is projective if $\mathrm{Hom}(X,-)$ commutes with reflexive coequalizers\footnote{Note that this corresponds with the more standard definitions in abelian contexts.} [\dots] 
Let $\mathcal{C}^{\mathrm{cp}}\subset\mathcal{C}$ be the full subcategory of compact projective objects.
\end{quote}

\noindent Fundamental examples include the following (as in item (5) below, in this section we will discard the underline convention for condensed images of a space):
\begin{quote}
\begin{enumerate}
\item If $\mathcal{C}=\mathsf{Set}$ is the category of sets, then $\mathcal{C}^{\mathrm{cp}}$ is the category of finite sets, which generates $\mathcal{C}$ under small colimits.
\item If $\mathcal{C}=\mathsf{Ab}$ is the category of abelian groups, then $\mathcal{C}^{\mathrm{cp}}$ is the category of finite[ly generated] free abelian groups, which generates $\mathcal{C}$ under small colimits. [\dots]
\item[(4)] If $\mathcal{C}=\mathsf{Cond(Set)}$ is the category of condensed sets, then $\mathcal{C}^{\mathrm{cp}}$ is the category of extremally disconnected profinite sets, which generates $\mathcal{C}$ under small colimits.
\item[(5)] If $\mathcal{C}=\mathsf{Cond(Ab)}$ is the category of condensed abelian groups, then $\mathcal{C}^{\mathrm{cp}}$ is the category of direct summands of $\mathbb{Z}[S]$ for extremally disconnected $S$, which generates $\mathcal{C}$ under small colimits. \cite[pp.\ 74-75]{CS2}
\end{enumerate}
\end{quote}
Extrapolating, the compact projective objects of a given category appear to incarnate the \emph{finite} and the \emph{extremally disconnected} in, respectively, classical and condensed settings; observe, however, that only the first of these classes is closed under (finite) products.
\begin{proposition}
\label{prop:Rudin}
If two extremally disconnected spaces $S$ and $T$ are both infinite, then their product is not extremally disconnected.
\end{proposition}
This proposition traces at the latest to Walter Rudin around 1960, by way of the attribution in \cite{Curtis}, and induces, in turn, a number of effects and questions within the field of condensed mathematics.
By item (4) above, for example, it implies that $\mathsf{Cond(Set)}^{\mathrm{cp}}$ is not closed under products.
Subtler questions attach to item (5).
There, $\mathbb{Z}[ - ]:\mathsf{Cond(Set)}\to\mathsf{Cond(Ab)}$ denotes the left adjoint to the forgetful functor $\mathsf{Cond(Ab)}\to\mathsf{Cond(Set)}$, and within $\mathsf{Cond(Ab)}$, the relation $\mathbb{Z}[S]\otimes\mathbb{Z}[T]=\mathbb{Z}[S\times T]$ renders $\mathbb{Z}[ - ]$-images of products of $\mathsf{ED}$ spaces computationally ubiquitous.
These are not, in general, projective, although the main result in this direction requires more argument than one might at first expect; the proof given in \cite{CS3}, for example, answers an open question from \cite{Aviles} in Banach space theory\footnote{That question, of whether the Banach space $C(\beta\mathbb{N}\times\beta\mathbb{N})$ is separably injective, typifies how this stratum of projectivity questions dualizes to injectivity questions in Banach settings; see \cite[Prop.\ 3.17]{CS3}.} along the way:

\begin{proposition}[Clausen--Scholze 2022]
\label{prop:projnotprods}
For any infinite sets $X$ and $Y$, the tensor product $\mathbb{Z}[\beta X]\otimes\mathbb{Z}[\beta Y]=\mathbb{Z}[\beta X\times\beta Y]$ is not projective.
\end{proposition}

See \cite[Appendix to III and pp.\ 26--27]{CS3} for further discussion of these issues; the latter includes, for example, the following observation: by item (5), $\mathsf{Cond(Ab)}$ is rich in projective objects, i.e., in $A=\mathbb{Z}[S]$ for which $\mathrm{Ext}^i_{\mathsf{Cond(Ab)}}(A,-)=0$ for all $i>0$.
For such to be \emph{internally} projective, though, would entail that the internal Ext functors
$$\underline{\mathrm{Ext}}^i_{\mathsf{Cond(Ab)}}(\mathbb{Z}[S],-)(T)=\mathrm{Ext}^i_{\mathsf{Cond(Ab)}}(\mathbb{Z}[S\times T],-)=0$$
for all $\mathsf{ED}$ spaces $T$ and $i>0$. Thus $\mathbb{Z}[\beta X]$ is not internally projective for any infinite set $X$, by Proposition \ref{prop:projnotprods}.

More refined analyses take account of the \emph{degree} of failure of a condensed abelian group to be projective; put differently, results like Proposition \ref{prop:projnotprods} may be viewed as partial answers to the general question \emph{For $\mathsf{ED}$ spaces $S$ and $T$, what is the projective dimension of $\mathbb{Z}[S\times T]$?}
We record the question of whether this dimension may uniformly be finite in our conclusion below.
In the present section, we address a strong variant of this prospect, one formulated at the level of the category $\mathsf{Cond(Ani)}$ of condensed anima.
This, heuristically, is the condensed category of homotopy types of topological spaces; for its precise $\infty$-category theoretic definition, see Section \ref{subsect:products}. The more immediate point is that, just as in items (4) and (5) above, the $\mathsf{ED}$ spaces define the compact projective objects of $\mathsf{Cond(Ani)}$.
This is the reason that we may frame our main contribution to these analyses in the following classical terms, yielding clause (1) and (2) of Theorem D.
\begin{theorem}
\label{thm:injdimconstantfield}
For any finite field $K$ and $\mathsf{ED}$ space $S$, the constant sheaf $\mathcal{K}$ on $S$ is injective. In contrast, if $\mathsf{ED}$ spaces $S$ and $T$ are each \v{C}ech-Stone compactifications of sets of cardinality at least $\aleph_\omega$, then for any field $K$ the injective dimension of the constant sheaf $\mathcal{K}$ on their product is infinite.
\end{theorem}
This theorem is, of course, evocative of both Proposition \ref{prop:Rudin} and Proposition \ref{prop:projnotprods}, but with a striking jump both in the threshold cardinality for unruly products from $\omega$ to $\aleph_\omega$.
The contours of all but the contributing Theorem \ref{thm:nonzero_cohomology_on_opens} of its argument are due to Peter Scholze, along with the deduction of a main consequence in the condensed setting, establishing clause (3) of Theorem D \cite{ScholzePersComm}:
\begin{theorem}
\label{thm:productsanima}
Products of compact projective condensed anima are not, in general, compact.
\end{theorem}
We divide our account of these results into two subsections. In the first, after a brief review of the relevant sheaf theory, we establish Theorem \ref{thm:injdimconstantfield}; in the second, after a brief review of the category of condensed anima, we establish Theorem \ref{thm:productsanima}. 


\subsection{Sheaves on $\mathsf{ED}$ spaces and their products}
\label{subsect:sheaves}

Recall that a \emph{presheaf of $R$-modules} $\mathcal{F}$ on a topological space $X$ is a contravariant functor from the lattice $(\tau(X),\subseteq)$ of open sets of $X$ to the category $R$-$\mathsf{Mod}$.
Write $r_{V,U}$ for the map $\mathcal{F}(U)\to\mathcal{F}(V)$ corresponding to the inclusion of $V$ into $U$, and recall that $\mathcal{F}$ is a \emph{sheaf} if it additionally satisfies the condition that \emph{for all open covers $\bigcup_{i\in I}U_i=U$ of elements $U$ of $\tau(X)$ and collections $\{s_i\in\mathcal{F}(U_i)\mid i\in I\}$ such that}
$$r_{U_i\cap U_j,U_i}(s_i)=r_{U_i\cap U_j,U_j}(s_j)$$
\emph{for all $i,j\in I$, there exists exactly one $s\in\mathcal{F}(U)$ satisfying} $$r_{U_i,U}(s)=s_i$$
\emph{for all $i\in I$}.
A fundamental class of examples are the so-called \emph{constant sheaves} assigning to each $U\in\tau(X)$ the $R$-module of \emph{locally constant functions} from $U$ to some fixed $R$-module $M$.
Within this subsection, the role of $R$ will tend to be played by a field. 

\begin{notation}
If $K$ is a field or vector space, then we write $\mathcal{K}$ for the associated constant sheaf on a topological space $X$; in cases of potential ambiguity, we append a subscript to indicate the base space, writing, for example, $\mathcal{K}_X$.

A calligraphic font is standard for two other classes of objects arising below as well; to distinguish these, we reserve the letter $\mathcal{I}$ for ideals, and $\mathcal{J}$ for injective resolutions or their constituent modules.
\end{notation}

We recall a few further fundamental points before turning to our main argument; for a fuller review, see, for example, \cite[Chapter 2]{Iversen}.

\begin{definition} The \emph{stalk} of a presheaf $\mathcal{F}$ at a point $x\in X$ is
$$\mathcal{F}_x:=\varinjlim_{U\ni x} \mathcal{F}(U).$$
Inferences about sheaves $\mathcal{G}$ on $X$ frequently reduce to inferences about their stalks $\mathcal{G}_x$ $(x\in X)$; relatedly, the \emph{sheafification} of a presheaf $\mathcal{F}$ may be defined simply as that sheaf $\widetilde{\mathcal{F}}$ on $X$ which is equipped with a map $\mathcal{F}\to\widetilde{\mathcal{F}}$ inducing an isomorphism $\mathcal{F}_x\to\widetilde{\mathcal{F}}_x$ of each stalk.

Given sheaves $\mathcal{F}$ and $\mathcal{G}$ on $X$ and $Y$, respectively, and a map $f:X\to Y$, the \emph{direct image} (or \emph{pushforward}) $f_*\mathcal{F}$ of $\mathcal{F}$ is the sheaf on $Y$ mapping each $U\in\tau(Y)$ to $\mathcal{F}(f^{-1}(U))$.
Left adjoint to $f_*$ is the \emph{inverse image} (or \emph{pullback}) functor $f^*$ sending $\mathcal{G}$ to the sheafification of the presheaf on $X$ given by
$$U\mapsto\varinjlim_{V\supseteq f[U]} \mathcal{G}(V).$$
When $f$ is the inclusion of a locally closed subspace $X$ into $Y$, we have as well an $f_!$ sending $\mathcal{F}$ to the sheaf $f_!\mathcal{F}$ on $Y$ with stalks
\begin{equation*}
(f_!\mathcal{F})_y=
    \begin{cases}
        \mathcal{F}_y & \text{if } y \in X\\
        0 & \text{if } y \notin X.
    \end{cases}
\end{equation*}
In this case, $f^*$ is left-inverse to $f_!$.

Suppose now that both $\mathcal{F}$ and $\mathcal{G}$ are sheaves of $K$-modules for some fixed field $K$ and that $Y$ is a point.
In this case, $\mathcal{G}$ naturally identifies with a $K$-vector space $W$ and $f^*\mathcal{G}$ is simply the constant sheaf $\mathcal{W}_X$, while $f_*\mathcal{F}$ is the vector space $\Gamma(X;\mathcal{F}):=\mathcal{F}(X)$ of \emph{global sections} of $\mathcal{F}$ over $X$. The \emph{sheaf cohomology $K$-modules} $\mathrm{H}^q(X;\mathcal{F})$ are, by definition, the evaluations at $\mathcal{F}$ of the right derived functors of this $\Gamma(X;\,\cdot\,)$ viewed as a functor to $K$-$\mathsf{Mod}$ from the abelian category of sheaves of $K$-modules on $X$.
By way of the aforementioned adjunction, we have as well the equation
$$\mathrm{Ext}_K^q(\mathcal{K},\mathcal{F})=\mathrm{H}^q(X;\mathcal{F})$$
with the left-hand side computed, of course, in the category concluding the previous sentence \cite[2.7.5]{Iversen}.
Therein, injective objects admit the standard definition; there are enough of them, and the injective dimension $\mathrm{injdim}_X(\mathcal{F})$ of any sheaf $\mathcal{F}$ is the minimal length of an injective resolution of $\mathcal{F}$.
\end{definition}
A useful injectivity criterion is the following:
\begin{lemma}
\label{clm:injective_criterion}
For any field $K$, a sheaf $\mathcal{F}$ of $K$-vector spaces over a topological space $X$ is injective if and only if $\mathrm{Ext}_K^1(\iota_{!}\mathcal{K},\mathcal{F})=0$ for every inclusion $\iota$ of a locally closed subset into $X$.
\end{lemma}
\begin{proof} The forward implication is standard. For the reverse implication, suppose that $\mathrm{Ext}_K^1$ vanishes in the manner described and let $f:\mathcal{A}\to\mathcal{F}$ and $g:\mathcal{A}\hookrightarrow\mathcal{B}$ be morphisms of sheaves over $X$; we need to show that $f$ extends to an $h:\mathcal{B}\to\mathcal{F}$.
To that end, write $\mathcal{A}'\leq\mathcal{A}''$ if $\mathcal{A}'$ is a subsheaf of $\mathcal{A}''$, regard $\mathcal{A}$ as a subsheaf of $\mathcal{B}$, and consider the natural partial ordering of pairs $(\mathcal{A}',h')$ for which $\mathcal{A}\leq\mathcal{A}'\leq\mathcal{B}$ and $h':\mathcal{A}'\to\mathcal{F}$ satisfies $f=h'\circ g$.
Since this order satisfies the conditions of Zorn's Lemma, our task reduces to showing that if there exists for some such $(\mathcal{A}',h')$ an $x\in S$ with $\mathcal{A}_x'\lneq\mathcal{B}_x$ then there exists an $\mathcal{A}''$ with $\mathcal{A}'\lneq\mathcal{A}''\leq\mathcal{B}$ and an $h'':\mathcal{A}''\to\mathcal{F}$ properly extending $h'$.

To that end, fix such an $x$ and a nonzero $b\in\mathcal{B}_x\backslash\mathcal{A}_x'$ and a representative $s\in\mathcal{B}(U)$ of $b$ for some open $U\in S$ containing $x$. Observe (by checking stalks) that $s$ generates a sheaf on $U$ which is isomorphic to, and which we are therefore justified in identifying with, $\mathcal{K}$.
Let $\varepsilon$ denote the inclusion $U\to X$ and observe that $\mathcal{A}'\cap\varepsilon_{!}\mathcal{K}\cong \eta_{!}\mathcal{K}$ for some inclusion $\eta$ of an open set $V$ into $X$.
Write $\iota$ for the inclusion of the locally closed set $U\backslash V$ into $X$ and apply $\mathrm{Hom}_K(-,\mathcal{F})$ to the short exact sequence
$$0\to\mathcal{A}'\to\mathcal{A}'+\varepsilon_{!}\mathcal{K}\to\iota_{!}\mathcal{K}\to 0$$
of sheaves on $X$. Our premise, applied to the fragment of the induced long exact sequence
$$\mathrm{Hom}_K(\mathcal{A}'+\varepsilon_{!}\mathcal{K},\mathcal{F})\xrightarrow{e}\mathrm{Hom}_K(\mathcal{A}',\mathcal{F})\to\mathrm{Ext}_K^1(\iota_{!}\mathcal{K},\mathcal{F}),$$
implies that $e$ is surjective, and hence that $h'$ extends to an an element $h'':\mathcal{A}'+\varepsilon_{!}\mathcal{K}\to\mathcal{F}$ of $e^{-1}(h')$. This $h''$ and $\mathcal{A}''=\mathcal{A}'+\varepsilon_{!}\mathcal{K}$ conclude the argument.
\end{proof}
Let us see how this criterion implies the first part of Theorem \ref{thm:injdimconstantfield} (compare also the proof of \cite[Lemma VII.1.7]{Fargues_Scholze_Geometrization}). 
\begin{lemma}
\label{lem:finitefieldsinjective}
For any finite field $K$ and $\mathsf{ED}$ space $S$, the constant sheaf $\mathcal{K}$ on $S$ is injective.
\end{lemma}
\begin{proof}
Observe first that since $K$ is compact, any continuous map from an open $U\subseteq S$ to $K$ uniquely extends to a continuous map $\beta U\to K$.
Similarly, since $S$ is $\mathsf{ED}$, the inclusion $U\hookrightarrow S$ uniquely extends to an embedding of $\beta U$ as an open and closed subspace of $S$.
Together these facts imply that the restriction map $\mathcal{K}(S)\to\mathcal{K}(U)$ is surjective.
When the latter holds, as here, for all open $U\subseteq S$, we say that the sheaf in question is \emph{flasque} (the term's application to inverse systems above was a special instance); the interest of this property is that flasque resolutions are acyclic and therefore convenient for cohomology computations.
It is easy to see that $\iota^*$ preserves this property for any such $\iota: U\hookrightarrow S$, and thus that $\mathrm{Ext}^1_K(\iota_{!}\mathcal{K},\mathcal{K})=\mathrm{H}^1(U;\iota^*\mathcal{K})=0$ for all such $\iota$ (the first equality follows from the $\iota_{!}$--$\iota^*$ adjunction together with the fact that such $\iota^*$ are exact and preserve injectives \cite[II.6]{Iversen}).
By Lemma \ref{clm:injective_criterion}, it remains only to see that $\mathrm{Ext}^1_K(\iota_{!}\mathcal{K},\mathcal{K})=0$ holds for all inclusions $\iota$ of \emph{locally closed} subsets $L$ into $S$ as well.

To that end, observe that any such $L$ is of the form $U\backslash V$ for some open $V\subseteq U\subseteq X$; letting $\eta$ and $\varepsilon$ again denote the inclusions of $V$ and $U$, respectively, we then have a short exact sequence
$$0\to\eta_{!}\mathcal{K}\to\varepsilon_{!}\mathcal{K}\to\iota_{!}\mathcal{K}\to 0$$
(cf.\ \cite[p.\ 112]{Iversen}). Applying $\mathrm{Hom}_K(-,\mathcal{K})$ together with the reasoning above yields the long exact sequence fragment
$$\mathrm{Hom}_K(\varepsilon_{!}\mathcal{K},\mathcal{K})\xrightarrow{f}\mathrm{Hom}_K(\eta_{!}\mathcal{K},\mathcal{K})\to\mathrm{Ext}_K^1(\iota_{!}\mathcal{K},\mathcal{K})\to 0.$$
Flasqueness implies that $f:\mathcal{K}(U)\to\mathcal{K}(V)$ is surjective, and thus that$$\mathrm{Ext}^1_K(\iota_{!}\mathcal{K},\mathcal{K})=0.$$
As $L$ was arbitrary, this completes the proof.
\end{proof}
For the second part of Theorem \ref{thm:injdimconstantfield}, we first prove a sequence of preliminary lemmas.
\begin{lemma}
\label{lemma:fininjdim_pushforward}
Let $K$ be a field and $S$ and $T$ be $\mathsf{ED}$ spaces and $\pi:S\times T\to S$ denote the projection map. Let $w(T)$ denote the cardinality of the smallest basis, or what is also known as the topological weight, of $T$. If the constant sheaf $\mathcal{K}$ on $S\times T$ is of finite injective dimension then so too is its direct image $\pi_{*}\mathcal{K}$, and the latter is in any case isomorphic to the constant sheaf $\mathcal{W}$ on $S$, where $W$ is the $K$-vector space $\bigoplus_{w(T)} K$.
\end{lemma}

\begin{proof}
For any $s\in S$, consider the stalk
$$(\pi_*\mathcal{K})_s=\varinjlim_{s\in U}\mathcal{K}(U\times T)$$
of $\pi_*\mathcal{K}$ over $s$.
By the compactness of $T$, for every locally constant $f:U\times T\to K$ there exists a neighborhood $V$ of $s$ and locally constant $g:T\to K$ such that $f(s',t)=g(t)$ for all $(s',t)\in V\times T$.
It follows that $(\pi_*\mathcal{K})_s$ equals the $K$-vector space of continuous functions $T\to K$ (with discrete codomain); moreover, the latter is isomorphic to $W=\bigoplus_{w(T)} K$ by the argument of \cite[Corollary 97.7]{Fuchs_Infinite_73}.
The lemma's concluding claim is now immediate: since the natural map from the constant sheaf $\mathcal{W}$ on $S$ to $\pi_*\mathcal{K}$ induces isomorphisms of stalks, it is an isomorphism of sheaves.

The preceding stalk computation holds, in fact, much more generally: it is a $q=0$ instance of the fact that for any $s\in S$ and sheaf $\mathcal{F}$ on $S\times T$, the stalk over $s$ of the right derived functor of the pushforward of $\mathcal{F}$
$$(\mathrm{R}^q\pi_*\mathcal{F})_s= \mathrm{H}^q(\pi^{-1}(s);\iota^{*}\mathcal{F})$$
for any $q\geq 0$, where $\iota^*$ denotes the pullback of the inclusion $\pi^{-1}(s)\hookrightarrow S\times T$ \cite[Theorem 17.2]{MilneLEC}.
Since pullback functors are exact and $\mathrm{H}^q(T;-)=0$ for any $q>0$ and profinite $T$ \cite[Theorem 5.1]{Wiegand}, this in turn implies that $\pi_*$ is exact.
The lemma's first claim now follows from the general fact that pushforwards preserve injectives \cite[II.4.1.3]{Iversen}, since this, when combined with the foregoing, implies that the image $\pi_*\mathcal{J}$ of a finite injective resolution $\mathcal{J}$ of $\mathcal{K}$ is a finite injective resolution of $\pi_*\mathcal{K}$.
\end{proof}
By way of this and the following lemma, assertions about the injective dimension of $\mathcal{K}$ on products of $\mathsf{ED}$ spaces --- and, hence, about the productivity of compact projective condensed anima --- convert to cohomology computations on open subsets of single $\mathsf{ED}$ spaces. 
\begin{lemma}
\label{lemma:eq_conds}
For every finite field $K$ and $\mathsf{ED}$ space $S$ and $K$-vector space $W$, the following conditions are equivalent:
\begin{enumerate}
\item The constant sheaf $\mathcal{W}$ on $S$ is of finite injective dimension.
\item There exists an $n\in\mathbb{N}$ such that $\mathrm{H}^q(U;\mathcal{W})=0$ for any open subset $U$ of $S$ and $q>n$.
\end{enumerate}
\end{lemma}
More exact relationships will emerge in the course of our proof below: the injective dimension of a $\mathcal{W}$ as in (1), for example, plays the role of an $n$ as in (2); an $n$ as in (2), on the other hand, implies that $\mathrm{injdim}_S(\mathcal{W})\leq n+1$.
Note also that our assumptions on $K$ and $S$ are unneeded for the implication $(1)\Rightarrow (2)$.
\begin{proof}
$(1)\Rightarrow (2)$: Much as in preceding arguments, the point is that the pullback functor $\iota^*$ associated to the inclusion $\iota:U\hookrightarrow S$ is both exact and preserves injectives.
Hence $\iota^*$ applied to a length-$n$ injective resolution of the constant sheaf $\mathcal{W}$ over $S$ is a length-$n$ injective resolution of the constant sheaf $\mathcal{W}=\iota^*\mathcal{W}$ over $U$; this implies (2).

$(2)\Rightarrow (1)$: Fix an $n$ as in (2) and a resolution
\begin{equation}
\label{eq:resolution1}
0\to\mathcal{W}\to\mathcal{J}^0\to\cdots\to\mathcal{J}^{n-1}\to\mathcal{F}\to 0
\end{equation}
of the sheaf $\mathcal{W}$ over $S$ in which each of the $\mathcal{J}^i$ are injective (if $n=0$, we just have the resolution $0\to\mathcal{W}\to\mathcal{F}\to 0$). By assumption, for every $q>n$ and inclusion $\iota$ of an open subset $U$ into $S$, 
\begin{align}
\label{eq:2implies1Ext}
\nonumber 0 = \mathrm{H}^{q}(U;\mathcal{W}) & = \mathrm{H}^{q}(U;\iota^*\mathcal{W}) \\
\nonumber & = \mathrm{Ext}_K^{q}(\mathcal{K},\iota^*\mathcal{W}) \\
\nonumber & = \mathrm{Ext}_K^{q}(\iota_{!}\mathcal{K},\mathcal{W}) \\
& = \mathrm{Ext}_K^{q-n}(\iota_{!}\mathcal{K},\mathcal{F}).
\end{align}
Here the penultimate equality follows just as in Lemma \ref{clm:injective_criterion} from the $\iota_{!}$--$\iota^*$ adjunction together with the fact that such $\iota^*$ are exact and preserve injectives, and the last equality is via a standard dimension-shifting argument, such as appears, for example, in \cite[Exercise 2.4.3]{Weibel}.
The vanishing of line \ref{eq:2implies1Ext} for all $q\geq n$ and open $U\subseteq S$ doesn't suffice for the conclusion of Lemma \ref{clm:injective_criterion}, but it's enough in this context for something almost as good, namely that $\mathrm{injdim}_S(\mathcal{F})\leq 1$.
To see this, fix a resolution
\begin{equation}
\label{SES1}
0\to\mathcal{F}\to\mathcal{J}^n\to\mathcal{H}\to 0
\end{equation}
with $\mathcal{J}^n$ injective.
Application of $\mathrm{Hom}_K(\iota_{!}\mathcal{K},-)$ for open inclusions $\iota$ as above then induces long exact sequences telling us two things: (i) that $\mathrm{Ext}^r_K(\iota_{!}\mathcal{K},\mathcal{H})=0$ for all such $\iota$ and $r>0$, and (ii) that $\mathcal{H}$ inherits the flasqueness of $\mathcal{J}^n$.
From these points, we conclude exactly as in Lemma \ref{lem:finitefieldsinjective} that $\mathcal{H}$ is injective.
Splicing (\ref{SES1}) to (\ref{eq:resolution1}) along $\mathcal{F}$ then defines a length-$(n+1)$ injective resolution
$$
0\to\mathcal{W}\to\mathcal{J}^0\to\cdots\to\mathcal{J}^{n-1}\to\mathcal{J}^n\to\mathcal{H}\to 0
$$
of $\mathcal{W}$.
\end{proof}
Against this background, we turn to the groups $\mathrm{H}^q(U;\mathcal{W})$ for such $W$ and open subsets $U \subseteq S$ as appear in Lemmas \ref{lemma:fininjdim_pushforward} and \ref{lemma:eq_conds} above; since such $U$ admit a convenient characterization when $S$ is the \v{C}ech-Stone compactification $\beta X$ of some discrete space $X$, we will focus our attention on $S$ of this sort. Note that there is no loss of generality in doing so, since every $\mathsf{ED}$ space $T$ is a retract of some such $\beta X$.\footnote{\label{ftnt:retracts} More precisely, such a retract descends to one of open sets $r^{-1}(U)\to U$; applying $\mathrm{H}^q(-;\mathcal{W})$ to the identity map $U\hookrightarrow r^{-1}(U)\to U$ then shows that if $n$ witnesses item 2 of Lemma \ref{lemma:eq_conds} for $\beta X$ then it does for $T$ as well. Thus for deriving a contradiction, via Lemma \ref{lemma:eq_conds}, from the premises of Lemma \ref{lemma:fininjdim_pushforward}, we may indeed restrict attention to the free $\mathsf{ED}$ spaces $\beta X$ for $X$ discrete.} The characterization just alluded to is the correspondence, given by Stone duality, between open subsets $U$ of $\beta X$ and ideals $\mathcal{I}$ on $X$ \cite[Lemma 19.1]{Halmos}. More precisely, any such $U$ equals
$\bigcup_{I\in\mathcal{I}_U} N_I$
for some ideal $\mathcal{I}_U\subseteq P(X)$; note next that each such $N_I$ corresponds, via the map $\mathcal{U}\mapsto\mathcal{U}\cap P(I)$, to the set of ultrafilters on $I$, and that this map in fact defines a homeomorphism $N_I\to\beta I$. 
In consequence, for any field $K$ and vector space $W=\bigoplus_\kappa K$ and open $U\subseteq\beta X$, we have the following:
\begin{align}
\mathrm{H}^q(U;\mathcal{W}) & = \mathrm{H}^q(\mathrm{colim}_{I\in\mathcal{I}_U} \,\beta I;\mathcal{W}) \label{eq:Rqlim0} \\
& = \mathrm{R}^q\mathrm{lim}_{I\in\mathcal{I}_U}\,\mathrm{H}^0(\beta I;\mathcal{W}) \label{eq:Rqlim1} \\
& = \mathrm{R}^q\mathrm{lim}_{I\in\mathcal{I}_U}\,\bigoplus_\kappa \mathrm{H}^0(\beta I;\mathcal{K}) \label{eq:Rqlim2} \\
& = \mathrm{R}^q\mathrm{lim}_{I\in\mathcal{I}_U}\,\bigoplus_\kappa \prod_I K
\label{eq:Rqlim3}
\end{align}
Here the fourth and third equalities follow from the definition and the compactness of $\beta I$, respectively; the first is immediate from our discussion above. The second may be viewed as an instance of \cite[Prop. 3.6.3]{Prosmans}; more precisely (since directed colimits are exact functors on sheaf categories), the right hand side of line \ref{eq:Rqlim0} is the $q^{\mathrm{th}}$ cohomology group of what may be written as
\begin{align*}
\mathrm{RHom}(\mathrm{colim}_{I\in\mathcal{I}_U}\,\iota_{!}\mathcal{K}_{\beta I},\mathcal{W}_U) & = \mathrm{Rlim}_{I\in\mathcal{I}_U}(\mathrm{RHom}(\iota_{!}\mathcal{K}_{\beta I},\mathcal{W}_U))\\
& = \mathrm{Rlim}_{I\in\mathcal{I}_U}(\mathrm{RHom}(\mathcal{K}_{\beta I},\iota^*\mathcal{W}_U)),
\end{align*}
where $\iota:\beta I\to U$ is the inclusion, the first equality is by way of the aforementioned reference, and the second is just as in the arguments of Lemmas \ref{lem:finitefieldsinjective} and \ref{lemma:eq_conds}.
As noted above, the cohomology groups $\mathrm{Ext}^r(\mathcal{K}_{\beta I},\iota^*\mathcal{W}_U)=\mathrm{H}^r(\beta I;\mathcal{W})$ of the last line's interior term vanish in all degrees $r>0$. We arrive in this way to the expression in line \ref{eq:Rqlim1} above.

The reward for these analyses is the expression appearing in line \ref{eq:Rqlim3}.
By way of it, we will argue the following theorem.
When coupled with Lemmas \ref{lemma:fininjdim_pushforward} and \ref{lemma:eq_conds} (or, more precisely, with the argument of its (1)$\Rightarrow$(2)), this will imply the second part of Theorem \ref{thm:injdimconstantfield}.
\begin{theorem}
\label{thm:nonzero_cohomology_on_opens}
Fix a field $K$ and $K$-vector space $W=\bigoplus_\kappa K$ for some $\kappa\geq\aleph_\omega$.
For any set $X$ with $|X|\geq\aleph_\omega$ and $n\in\mathbb{N}$, there exists an open $U\subseteq\beta X$ such that $\mathrm{H}^n(U;\mathcal{W})\neq 0$.
\end{theorem}
The remainder of this section will principally be taken up with a proof and partial converse of this theorem (see also Remark \ref{rmk:betternontrivs} below for some refinements), in which several of the preceding sections' themes reappear.
Our argument will revolve around a further instance of the generalized $n$-coherence appearing in Definition \ref{def:coh2}; this is the instance, for any cardinal $\kappa$, field $K$, set $X$, and ideal $\tilde{\mathcal{I}}$ on $X$, given by letting the definition's
\begin{itemize}
    \item $Y=\kappa\times X$,
    \item $H=K$,
    \item $\mathcal{I}$ be the ideal on $Y$ generated by $\{\kappa\times I\mid I\in\tilde{\mathcal{I}}\}$, and
    \item $\mathcal{J}$ be the ideal on $Y$ generated by $\{s\times X\mid s\in [\kappa]^{<\omega}\}.$
\end{itemize}
Let $\mathbf{Y}[\kappa,X,\tilde{\mathcal{I}}, K]$ denote the inverse system 
$\mathbf{X}[\mathcal{I}, \mathcal{J}, H]$ derived from this data as in 
Definition \ref{def:coh2}. Note that the set $\{\kappa \times I \mid I \in 
\tilde{I}\}$ is cofinal in $\mathcal{I}$ and that
\[
  X_{\kappa \times I} = \bigoplus_\kappa \prod_I K
\]
for any $I \in \tilde{I}$.
For any $n>0$, the associated notion of $n$-coherence is then that of a family
$$\Phi=\left\langle\varphi_{\vec{I}}:\kappa\times\bigcap\vec{I}\to K\mid\vec{I}\in\tilde{\mathcal{I}}^n\right\rangle$$
such that for any $\vec{I}\in \tilde{\mathcal{I}}^{n+1}$, the function
$$\sum_{i\leq n}(-1)^i\varphi_{\vec{I}^i}(\xi,\,\cdot\,)\big|_{\bigcap\vec{I}}$$
is nonzero for only finitely many $\xi\in\kappa$.
Triviality is similarly defined, and it is immediate from Lemma \ref{lem:nontrivfamsarenontrivlimsagain} that for all $n>0$ and $\kappa$ and $K$ and $\tilde{\mathcal{I}}$ as above, 
\[
  \operatorname{lim}^n \mathbf{Y}[\kappa,X,\tilde{I},K] = 0
\]
if and only if the associated $n$-coherent families of functions are all trivial.

\begin{proof}[Proof of Theorem \ref{thm:nonzero_cohomology_on_opens}]
By equations \ref{eq:Rqlim0} through \ref{eq:Rqlim3}, it will suffice to show that for any $K$, $\kappa$, $X$, and $n$ as in the statement of the theorem, there exists an ideal $\mathcal{I}_n$ on $X$ such that
$\operatorname{lim}^n \mathbf{Y}[\kappa,X,\mathcal{I}_n,K] \neq 0$.
For this it will suffice to exhibit for each $n>0$ a nontrivial $n$-coherent family of functions
$$\Phi=\left\langle\varphi_{\vec{\alpha}}:\kappa\times\bigwedge\vec{\alpha}\to K\mid\vec{\alpha}\in (\omega_n)^n\right\rangle;$$
to see this, identify $X$ with its cardinality $|X|=\lambda\geq\aleph_\omega$ and let $\mathcal{I}_n$ denote the bounded ideal on $\omega_n\subset\lambda$ and apply Lemmas \ref{lem:nontrivfamsarenontrivlimsagain} and Remark \ref{rmk:cofinality}. 
Much as in the proof of item 5 of Theorem \ref{thm:limsofAkappalambda}, such families $\Phi$ may be efficiently derived from existing constructions in the set theoretic literature;\footnote{Functions $\varphi_{\vec{\alpha}}$ as desired may be read off from the functions $\mathtt{f}_n$ of Section 6 of \cite{TFOA} (with $K$ replacing $\mathbb{Z}$ in the definition of the latter) or, relatedly, from the walks functions $r_2^n$ of Sections 7 and 9 of \cite{IHW}.} since in the present case, however, the relevant machinery is both more involved and less well-known, it seems preferable here to construct such families directly.
We will do so by recursion on $n$. Henceforth fix a field $K$.

For the base case $n=1$, the functions $\tau_\alpha$ appearing in the proof of item 5 of Theorem \ref{thm:limsofAkappalambda} furnish a nontrivial coherent $\Phi=\langle\varphi_\alpha:\kappa\times\alpha\to K\mid\alpha\in\omega_1\rangle$ with $\kappa=\aleph_0$, as the reader may verify.
In order to better fix both our notation and a template for higher $n$, though, we will record a more hands-on construction of such a $\Phi$, one in which $\kappa=\omega_1$.
For this purpose, fix a sequence $\langle C_\alpha\subseteq\alpha\mid\alpha\in\mathrm{Cof}(\omega)\cap\omega_1\rangle$ in which $\mathrm{otp}(C_\alpha)=\omega$ and $\mathrm{sup}(C_\alpha)=\alpha$ for every countable limit ordinal $\alpha$; for each such $\alpha$, let $\chi_\alpha(0,-):\alpha\to K$ be the characteristic function of $C_\alpha$ and let $\chi_\alpha:\omega_1\times\alpha\to K$ otherwise equal $0$.
For any $\psi:\kappa\times\alpha\to K$ and $\gamma<\kappa$,  define $\psi^{(\gamma)}:\kappa \times \alpha 
\rightarrow K$ by letting $\psi^{(\gamma)}(\gamma+\eta,\xi)=\psi(\eta,\xi)$ for all $(\eta,\xi)\in\kappa\times\alpha$ and letting $\psi^{(\gamma)}$ otherwise equal $0$.
We proceed now to construct by recursion on $\alpha$ a nontrivial coherent $\Phi=\langle\varphi_\alpha:\omega_1\times\alpha\to K\mid\alpha\in\omega_1\rangle$, as desired.

Our recursive assumption at stage $\beta$ is that $\Phi\restriction\beta:=\langle\varphi_\alpha:\omega_1\times\alpha\to K\mid\alpha\in\beta\rangle$ is defined and, moreover, that
$$\{\eta<\omega_1\mid\varphi_\alpha(\eta,\xi)\neq 0\text{ for some }\xi<\alpha<\beta\}\subseteq\beta;$$
henceforth we reserve the notation $\mathrm{supp}(-)$ (or here, $\mathrm{supp}(\Phi\restriction\beta)$) for a family of self-explanatory variations on the bracketed expression above. There are two possibilities:
\begin{enumerate}
\item $\beta=\alpha+1$ for some $\alpha$. In this case, let $\varphi_\beta(\eta,\xi)=\varphi_\alpha(\eta,\xi)$ for $\xi<\alpha$ and otherwise equal $0$.
\item $\beta$ is a limit ordinal. In this case, by Goblot's Theorem, there exists a trivialization $\psi:\omega_1\times\beta\to K$ of $\Phi\restriction\beta$; since $\mathrm{supp}(\Phi\restriction\beta)\subseteq\beta$ we may assume $\mathrm{supp}(\psi)\subseteq\beta$ as well. Let $\varphi_\beta=\psi+\chi^{(\beta)}_\beta$.
\end{enumerate}
This completes our account of stage $\beta$; note that our recursive assumption is conserved.

The $\Phi$ so constructed is plainly coherent, so suppose for contradiction that some $\psi:\omega_1\times\omega_1\to K$ trivializes $\Phi$.
Then for each $\beta\in\mathrm{Cof}(\omega)\cap\omega_1$ there exists $\xi_\beta<\beta$ with $\psi(\beta,\xi_\beta)\neq 0$; by Fodor's Lemma, $\xi_\beta$ equals some fixed $\xi$ for uncountably many $\beta$.
For any $\alpha>\xi$, though, $\mathrm{supp}(\varphi_\alpha)\subseteq\alpha+1$ implies that $\varphi\neq^*\psi$, and this is the desired contradiction.

We now show how the existence of nontrivial $n$-coherent families $$\Psi=\left\langle\psi_{\vec{\alpha}}:\omega_n\times\bigwedge\vec{\alpha}\to K\mid\vec{\alpha}\in (\omega_n)^n\right\rangle$$
implies the existence of nontrivial $(n+1)$-coherent families
$$\Phi=\left\langle\varphi_{\vec{\alpha}}:\omega_{n+1}\times\bigwedge\vec{\alpha}\to K\mid\vec{\alpha}\in (\omega_{n+1})^{n+1}\right\rangle.$$
As above, we construct $\Phi$ in stages $\Phi\restriction\beta=\langle\varphi_{\vec{\alpha}}:\omega_{n+1}\times\bigwedge\vec{\alpha}\to K\mid\vec{\alpha}\in \beta^{n+1}\rangle$ ranging through $\beta<\omega_{n+1}$, maintaining the condition that $\mathrm{supp}(\Phi\restriction\beta)\subseteq\beta$ for all $\beta\in\mathrm{Cof}(\omega_n)\cap\omega_{n+1}$, and 
$\mathrm{supp}(\Phi\restriction\beta)\subseteq\beta + \omega_n$ for all 
other $\beta \in \omega_{n+1}$, as we go.
Again we distinguish two cases:
\begin{enumerate}
\item $\mathrm{cf}(\beta)\neq\aleph_n$. By Goblot's Theorem, there exists an alternating trivialization $\Theta=\langle\theta_{\vec{\alpha}}:\omega_{n+1}\times\bigwedge\vec{\alpha}\to K\mid\vec{\alpha}\in\beta^n\rangle$ of $\Phi\restriction\beta$ with $\mathrm{supp}(\Theta)=\mathrm{supp}(\Phi\restriction\beta)$.
Extend $\Phi\restriction\beta$ to $\Phi\restriction\beta+1$ by letting $\varphi_{\vec{\alpha}^\frown\langle\beta\rangle}=(-1)^{n+1}\theta_{\vec{\alpha}}$ for all $\vec{\alpha}\in\beta^n$; the requirement that $\Phi$ be alternating then determines on $\Phi\restriction\beta+1$ on any remaining indices.
\item $\mathrm{cf}(\beta)=\aleph_n$. In this case, our assumptions together with Remark \ref{rmk:cofinality} ensure the existence of a nontrivial $n$-coherent $\Psi=\langle\psi_{\vec{\alpha}}:\omega_n\times\bigwedge\vec{\alpha}\to K\mid\vec{\alpha}\in (\beta)^n\rangle$ with $\mathrm{supp}(\Psi)\subseteq\omega_n$, as well as a $\Theta$ trivializing $\Phi\restriction\beta$ as before.
Extend $\Phi\restriction\beta$ to $\Phi\restriction\beta+1$ by letting $\varphi_{\vec{\alpha}^\frown\langle\beta\rangle}=(-1)^{n+1}\theta_{\vec{\alpha}}+\psi^{(\beta)}_{\vec{\alpha}}$ for all $\vec{\alpha}\in\beta^n$.
\end{enumerate}
This completes our account of stage $\beta$; note just as above that our condition on $\mathrm{supp}(\Phi\restriction\beta)$ is conserved.

The $n$-coherence of $\Phi$ is again straightforward to verify, so suppose for contradiction that
$$\Upsilon=\left\langle\upsilon_{\vec{\alpha}}:\omega_{n+1}\times\bigwedge\vec{\alpha}\to K\mid\vec{\alpha}\in (\omega_{n+1})^n\right\rangle$$
trivializes $\Phi$.
Much as in the $n=1$ case, the nontriviality of the families $\Psi$ arising in the course of our construction ensures that, for each $\beta\in\mathrm{Cof}(\omega_n)\cap\omega_{n+1}$, there exists an $f(\beta)=\vec{\alpha}\in(\beta)^{n+1}$ and $\eta\in [\beta,\beta+\omega_n)$ such that
\begin{equation}
\label{eq:sum}
\sum_{i=0}^n(-1)^i\upsilon_{\vec{\alpha}^i}(\eta,\xi)\neq 0
\end{equation}
for some $\xi<\alpha_0$ (if not, then $\Upsilon\restriction (\beta)^n$ would itself be trivial on the domain $[\beta,\beta+\omega_n)\times\beta$, implying that the associated $\Psi$ is). By Fodor's Lemma, $f$ is constantly $\vec{\alpha}$ on some cofinal $B\subseteq\omega_{n+1}$, hence the witnesses $\eta$ to equation \ref{eq:sum} for this $\vec{\alpha}$ are cofinal in $\omega_{n+1}$ as well.
Since $\mathrm{supp}(\varphi_{\vec{\alpha}})\subseteq\mathrm{min}\,B$, though, this implies that $\Upsilon$ does not trivialize $\Phi$.
This contradiction completes the argument.
\end{proof}
\begin{remark}
\label{rmk:betternontrivs}
Observe that any nontrivial abelian group $K$ would suffice for the above argument (with some nonzero $a$ taking the place of $1$), and hence for the statement of Theorem \ref{thm:nonzero_cohomology_on_opens} as well.
A deeper question is whether the dimension $\kappa$ of $W$ need really be at least $\aleph_\omega$ in Theorem \ref{thm:nonzero_cohomology_on_opens}; note that this translates in Theorem \ref{thm:injdimconstantfield} to a question of whether \emph{both} factors $S$ and $T$ need be $\aleph_\omega$-large for their product to carry infinite injective dimension. Let us briefly note that
\begin{itemize}
\item $\kappa$ must be infinite (if it is finite then the system of equation \ref{eq:Rqlim3} is flasque);
\item in the case of $n=1$, as noted, $\kappa=\aleph_0$ suffices;
\item it is consistent with the $\mathsf{ZFC}$ axioms that $\kappa=\aleph_0$ suffices for all $n>0$ (it suffices in the constructible universe $L$, by straightforward modifications of \cite[\S 3.3]{BLHCoOI}, for example);
\item but whether $\kappa=\aleph_0$ suffices for all $n>0$ in $\mathsf{ZFC}$ alone is a subtle question, one again closely related to that of the groups $\mathrm{H}^n(\omega_n;\mathbb{Z})$ (or $\mathrm{H}^n(\omega_n;\mathbb{Z}/p\mathbb{Z})$) referenced in Section \ref{sec:additivity}.
\end{itemize}

On the other hand, by the following lemma, the conditions on $X$ in Theorem \ref{thm:nonzero_cohomology_on_opens} are essentially sharp.
\end{remark}

\begin{lemma}
\label{lem:reverse_implication}
If $S=\beta X$ for a discrete space $X$ with $|P(X)|<\aleph_\omega$ then for any finite field $K$ and $K$-vector space $W$, the constant sheaf $\mathcal{W}$ on $S$ is of finite injective dimension.
\end{lemma}
\begin{proof}
Let $|P(X)|=\aleph_n$; any $\mathcal{I}_U$ as in equation \ref{eq:Rqlim3} above then determines a $\mathrm{Cof}(\leq\aleph_n)$-indexed inverse system with surjective bonding maps whose $\mathrm{Rlim}$ computes $\mathrm{H}^q(U;\mathcal{W})$.
By equations \ref{eq:Rqlim0} through \ref{eq:Rqlim3} together with Goblot's Theorem, the second, and hence the first, of Lemma \ref{lemma:eq_conds}'s conditions therefore holds.
\end{proof}

\begin{corollary}
\label{cor:413}
If $\aleph_\omega$ is a strong limit cardinal then for any finite field $K$ and $K$-vector space $W$ and $\mathsf{ED}$ space $S$ of cardinality less than $\aleph_\omega$, the constant sheaf $\mathcal{W}$ on $S$ is of finite injective dimension.
\end{corollary}
\begin{proof}
For $S=\beta X$ for a discrete space $X$, this is immediate from Lemma \ref{lem:reverse_implication}. Under our cardinal arithmetic assumptions, any other $\mathsf{ED}$ space $S$ of cardinality less than $\aleph_\omega$ is a retract of such a $\beta X$ with $|\beta X|<\aleph_\omega$ (simply let $X$ equal the underlying set of $S$), so the conclusion follows by the reasoning of footnote \ref{ftnt:retracts}.
\end{proof}

\subsection{Products of compact projective anima}
\label{subsect:products}

Recall this section's organizing concern: \emph{When is a category's class of compact projective objects closed under the formation of finite products?}
Among condensed categories, we considered three main possibilities, which we may abbreviate as follows:
\begin{enumerate}[label=(\Alph*)]
\item $\mathsf{Cond(Set)}^{\mathrm{cp}}$ is closed with respect to binary products;
\item $\mathsf{Cond(Ab)}^{\mathrm{cp}}$ is closed with respect to tensor products;
\item $\mathsf{Cond(Ani)}^{\mathrm{cp}}$ is closed with respect to binary products.
\end{enumerate}
We noted the failures of (A) and (B) and considered, in their wake, the following weakening of (B):
\begin{enumerate}[label=(\Alph*)]
\setcounter{enumi}{3}
\item Tensor products of compact projective objects in $\mathsf{Cond(Ab)}$ are of finite projective dimension.
\end{enumerate}
In this subsection, we record two main results --- namely, that (C) implies (D), but also that by the work of the previous subsection, (C), in general, fails to hold.
In fact it will be most natural to argue these points in reverse order, after a brief review of $\mathsf{Cond(Ani)}$ and the phenomenon of \emph{animation}.

To that end, let us continue with the quoted passage with which this section began:
\begin{quote}
Taken together, an object $X\in\mathcal{C}$ is compact projective if $\mathrm{Hom}(X,-)$ commutes with all filtered colimits and reflexive coequalizers:
equivalently, it commutes with all (so-called) $1$-sifted colimits. \cite[p.\ 74]{CS2}
\end{quote}
Here, filtered colimits are precisely those which commute with finite limits, and $(1\text{-})$sifted colimits are precisely those which commute with finite products.
The prefix $1$-, of course, signals that we're discussing the ordinary category theoretic version of what we'll soon upgrade to an $\infty$-category theoretic notion.
Write $\mathsf{sInd}(\mathcal{C})$ for the free cocompletion of $\mathcal{C}$ with respect to $1$-sifted colimits.
If $\mathcal{C}$ is cocomplete then we have a fully faithful embedding $\mathsf{sInd}(\mathcal{C}^{\mathrm{cp}})\to\mathcal{C}:\textnormal{``}\mathrm{colim}\textnormal{''}\,X_i\mapsto\mathrm{colim}\,X_i$, and ``if $\mathcal{C}$ is generated under small colimits by $\mathcal{C}^{\mathrm{cp}}$ then that functor is an equivalence'' (\cite[p.\ 74]{CS2}; for fuller argument, see \cite[\S 5.1.1]{CesnaviciusScholze}).
In particular, we have $\mathsf{sInd}(\mathcal{C}^{\mathrm{cp}})\cong\mathcal{C}$ for each of the enumerated examples of the section's introduction: the category $\mathsf{Set}$ is the $1$-sifted ind-completion $\mathsf{sInd}(\mathsf{Set}^{\mathrm{cp}})$ of the category $\mathsf{Set}^{\mathrm{cp}}$ of finite sets; similarly for $\mathsf{Ab}$, for $\mathsf{Cond(Ab)}$, and so on.

Animation is simply the $\infty$-category theoretic version of this operation;\footnote{For sifted diagrams in $\infty$-categories, see \cite[\S 5.5.8]{LurieHTT} or \cite[Tag 02QD]{kerodon}; these are again, intuitively, those whose colimits preserve finite products.}
continuing with \cite{CS2}, we have:
\begin{definition}
Let $\mathcal{C}$ be a category that admits all small colimits and is generated under small colimits by $\mathcal{C}^{\mathrm{cp}}$. The \emph{animation of $\mathcal{C}$} is the $\infty$-category $\mathsf{Ani}(\mathcal{C})$ freely generated under sifted colimits by $\mathcal{C}^{\mathrm{cp}}$.
\end{definition}
For example:
\begin{enumerate}
\item The animation $\mathsf{Ani(Set)}$ of the category of sets is nothing other than the $\infty$-category of ``spaces'': it is, in other words, that higher category theoretic incarnation of the homotopy category of topological spaces playing within $\infty$-category theory a role broadly analogous to that of $\mathsf{Set}$ within ordinary category theory (see \cite[\S 1.2.16]{LurieHTT} for other presentations). Its objects are termed \emph{anima}, and the category itself is frequently abbreviated $\mathsf{Ani}$.
\item The animation $\mathsf{Ani(Ab)}$ of the category of abelian groups is equivalent (via the Dold-Kan equivalence) to the $\infty$-derived category of abelian groups $\mathsf{D}_{\geq 0}(\mathsf{Ab})$ in non-negative homological degrees.
\end{enumerate}
Similarly for the categories $\mathsf{Cond(Set)}$ and $\mathsf{Cond(Ab)}$; for these and the above examples and the following definition and lemma, see again \cite[\S 11]{CS2}, where the terminology \emph{nonabelian derived category} for the animation construction is also noted.
As it happens, the $\infty$-categorical condensation and animation operations commute:
\begin{definition}
\label{def:inftycatcondensed}
Let $\mathcal{C}$ be an $\infty$-category that admits all small colimits.
For any uncountable strong limit cardinal $\kappa$, the $\infty$-category $\mathsf{Cond}_\kappa(\mathcal{C})$ of $\kappa$-condensed objects of $\mathcal{C}$ is the category of contravariant functors from $\mathsf{ED}_\kappa$ to $\mathcal{C}$ which take finite coproducts to products. One then lets $\mathsf{Cond}(\mathcal{C})=\mathrm{colim}_{\kappa\in S}\,\mathsf{Cond}_\kappa(\mathcal{C})$ for $S$ the class of strong limit cardinals, just as in Definition \ref{def:condset}.
\end{definition}

\begin{lemma}
Let $\mathcal{C}$ be a category that is generated under small colimits by $\mathcal{C}^{\mathrm{cp}}$.
Then there is a natural equivalence of $\infty$-categories $\mathsf{Cond(Ani(}\mathcal{C}))\cong\mathsf{Ani(Cond(}\mathcal{C}))$.
\end{lemma}

It follows that $\mathsf{Cond(Ani)}$, the category of condensed anima, embraces within a single setting both topological and homotopical domains, in the sense that the natural embeddings $e:\mathsf{Cond(Set)}\to\mathsf{Cond(Ani)}$ and $\mathsf{Ani(Set)}\to\mathsf{Ani(Cond(Set))}$ are each fully faithful.
For more on these matters, see \cite{Mair}; the argument therein that $\mathsf{ED}$ spaces define the compact projective objects of $\mathsf{Cond(Set)}$ also readily adapts to show that their $e$-images form a generating class of compact projective objects in $\mathsf{Cond(Ani)}$.

This, briefly, is the context for Theorem \ref{thm:productsanima}.
For its argument, we recall one further notion, that of a \emph{hypercover}.
Much as the notion of a site generalizes that of a topological space, hypercovers may be regarded as simultaneously generalizing the notions both of topological covers and
of projective resolutions: a hypercover of an object $X$ of $\mathcal{C}$ is an augmented simplicial object
$$\cdots\substack{\longrightarrow\\[-1em] \longrightarrow \\[-1em] \longrightarrow}R_1\substack{\longrightarrow\\[-1em] \longrightarrow} R_0\substack{\longrightarrow} X$$
in $\mathcal{C}$ with $X$ in the degree $-1$.
(Above, we've left degeneracy maps from lower to higher degrees undepicted simply for typographical reasons, and we will tend in what follows to denote such a hypercover simply by $R_\bullet$, i.e., to identify it with its non-negatively graded portion, so that a more scrupulous notation would be $R_\bullet\to X$.)
Like covers or resolutions, hypercovers facilitate cohomology computations by decomposing their target into more convenient objects, and the induced hypercovering system of higher-order intersections $\bigcap_{i\in I} U_i$ of elements of an open cover $\mathcal{U}$ of a topological space $X$ which underlies the \v{C}ech cochain complex of $(X,\mathcal{U})$ is among the most fundamental of examples.
The more general notion of a hypercover is often motivated as allowing for further refinements of the cover at each level $n$ of the complex, a prospect technically managed by an $n$-coskeleton condition for which we refer the reader to any standard reference (\cite[\S 8]{ArtinMazur}, \cite[\S 4]{DuggerIsak}, \cite[\S 4]{Conrad}).
The important points for our purposes are the following:
\begin{enumerate}[label=(\alph*)]
\item In the present context, the ``convenient objects'' are the $\mathsf{ED}$ spaces; in particular, any $X\in\mathsf{CHaus}$ admits a hypercovering $R_\bullet$ by $\mathsf{ED}$ spaces, and even a standard such ``free resolution'' \cite[2.2.5]{pyknotic}. For hypercoverings of topological spaces $R_\bullet\to X$, we have moreover that $X$ is weakly homotopy equivalent to the homotopy colimit of $R_\bullet$ \cite{DuggerIsakRealizations}.
\item Just as was the case for ordinary categories (Definition \ref{def:condset}), the sheaf condition defining condensed $\infty$-categories considerably simplifies when we restrict the underlying site to $\mathsf{ED}$ spaces, a simplicity on display in  Definition \ref{def:inftycatcondensed}.
In the $\infty$-categorical setting, though, that definition is equivalent not to the category of \emph{sheaves} on the profinite site, but of \emph{hypersheaves} (see \cite[\S 9]{Masterclass}, \cite[\S 2.2.2--2.2.7]{pyknotic} for this point in the pyknotic setting, and \cite[\S A.4]{Mann} for a concise review of the relevant notions).
In other words, writing $\mathsf{StrLim}$ for the class of strong limit cardinals, a condensed anima may equivalently be characterized as an element of $\mathrm{colim}_{\kappa\in\mathsf{StrLim}}\,\mathrm{Fun}(\mathsf{ProFin}_\kappa^{\mathrm{op}},\mathsf{Ani})$ satisfying the $\infty$-categorical analogues of conditions (1) through (3) of Definition \ref{def:condset}; the equalizer condition $$F(S)=\mathrm{lim}\,(F(T)\rightrightarrows F(T\times_S T))\textnormal{ for any surjection }T\twoheadrightarrow S$$ of item (3), in particular, translates precisely to 
\begin{align}
\label{eq:hypersheaf}
F(S)=\mathrm{lim}\,F(T_\bullet)
\end{align}
for any hypercover $T_\bullet$ of $S$ in the category of profinite spaces (cf.\ \cite[A.4.20]{Mann} and the references therein).
Note that this condition in turn implies that $S = \mathrm{colim}\,T_\bullet$ when each is interpreted in the category $\mathsf{Cond(Ani)}$.
This follows from the Yoneda lemma, together with the fact that
$$\mathrm{Hom}(S,F)=F(S)=\mathrm{lim}\,F(T_\bullet)=\mathrm{lim}\,\mathrm{Hom}(T_\bullet,F)=\mathrm{Hom}(\mathrm{colim}\,T_\bullet,F),$$
by (\ref{eq:hypersheaf}), for all condensed anima $F$.
More colloquially, the hypersheaf condition ensures that the embedding $\mathsf{ProFin}\to\mathsf{Cond(Ani)}$ conserves the (weak) equivalences noted in item (a) above.
\end{enumerate}

We turn now to the derivation of Theorem \ref{thm:productsanima} from Theorem \ref{thm:injdimconstantfield}.
The task essentially reduces to proving the following lemma.

\begin{lemma}
\label{lem:43to44}
For all finite fields $K$ and $\mathsf{ED}$ spaces $S$ and $T$, if $S\times T$, viewed as a product of condensed anima, is compact, then the sheaf $\mathcal{K}$ on the topological space $S\times T$ is of finite injective dimension.
\end{lemma}
The point, of course, is that for any large $S$ and $T$ as in Theorem \ref{thm:injdimconstantfield}, the sheaf $\mathcal{K}$ on the topological space $S\times T$ is \emph{not} of finite dimension, hence the product $S\times T$ of compact projective anima is not compact, and this proves Theorem \ref{thm:productsanima}.
\begin{proof}[Proof of Lemma \ref{lem:43to44}]
Fix $K,S$, and $T$ as in the statement of the lemma, so that the product $S\times T$ of anima is compact. Fix as in item (a) above a hypercover $R_\bullet$
of the topological space $S\times T$ by $\mathsf{ED}$ spaces.
For each $n\in\omega$ write $f_n$ for the induced map $R_n\to S\times T$ and $X_n$ for $\mathrm{colim}\,(R_\bullet)_{\leq n}$, that is, for the colimit in $\mathsf{Cond(Ani)}$ of the truncation of the simplicial hypercover $R_\bullet$ to its first $n+1$ terms.
By item (b) above, the sequential colimit $\mathrm{colim}_{n\in\omega}\,X_n$ over the system of natural maps $X_m\to X_n$ $(m\leq n)$ is equal in $\mathsf{Cond(Ani)}$ to $S\times T$; from this it follows by our compactness assumption that the identity map on $S\times T$ factors through some $X_n$.
Write $g$ for the associated retraction map $X_n\to S\times T$.
For any sheaves $\mathcal{A}$ and $\mathcal{B}$ on $S\times T$, then, $\mathcal{B}$ is a retract of $g_*g^*\mathcal{B}$ and hence $\mathrm{RHom}(\mathcal{A},\mathcal{B})$ is a retract of $\mathrm{RHom}(\mathcal{A},g_*g^*\mathcal{B})$ as well.
Our proof will conclude by justifying the claim that $g_*g^*\mathcal{B}$ is, in the derived $\infty$-category of sheaves on $S\times T$, the limit, over a finite diagram $I$, of the sheaves ${f_i}_*f_i^*\mathcal{B}$; let us first see, though, what the claim gives us: it supplies the first equality in the sequence
\begin{align*}\mathrm{RHom}(\mathcal{A},g_*g^*\mathcal{B}) & =\mathrm{RHom}(\mathcal{A},\mathrm{lim}_I\,{f_i}_*f_i^*\mathcal{B}) \\ & =\mathrm{lim}_I\,\mathrm{RHom}(\mathcal{A},{f_i}_*f_i^*\mathcal{B}) \\ & =\mathrm{lim}_I\,\mathrm{RHom}(f_i^*\mathcal{A},f_i^*\mathcal{B}).\end{align*}
For the second equality, recall that $\mathrm{RHom}$ preserves $\infty$-categorical limits in the second coordinate; the third is just the
pushforward-pullback adjunction.
When $\mathcal{B}=\mathcal{K}$, though, we know from Theorem \ref{thm:injdimconstantfield} that $f_i^*\mathcal{B}$ is injective.
Hence as $\mathcal{A}$ ranges over sheaves of $K$-modules on $S\times T$, the complex $\mathrm{RHom}(\mathcal{A},g_*g^*\mathcal{K})$ is a finite limit of uniformly bounded complexes and therefore itself uniformly bounded. So too, in consequence, are its retracts $\mathrm{RHom}(\mathcal{A},\mathcal{K})$, a conclusion amounting precisely to our assertion that the sheaf $\mathcal{K}$ on $S\times T$ is of finite injective dimension.

It remains only to justify our claim; in essence, it's the sum of two subclaims.
The first is that the functor $D$ which sends each profinite set $T$ to the $\infty$-category of sheaves on $T$, and morphism $f:S\to T$ to $f^*:D(T)\to D(S)$, satisfies descent, i.e., satisfies the condition (\ref{eq:hypersheaf}) for the functor $F=D$ \cite[Lem.\ 3.5.12]{Mann}.
The second is that coupling this first subclaim with Lemma D.4.7(ii) of \cite{Mann} implies the claim: within the lemma, let $I$ denote the diagram $\Delta_{\leq n}^{\mathrm{op}}$ defining $X_n$, let $\mathcal{D}_i=D(R_i)$, let $\mathcal{C}=D(S\times T)$, let $F_i=f_i^*$, and let $F=g^*$. Then by descent, $D(X_n)=\mathrm{lim}_I\,D(R_i)$, and the lemma defines the right adjoint $g^*$ of $g_*$ as a limit of the right adjoints $f_{i^*}$ of $f_i^*$ in the manner, precisely, of our claim.
\end{proof}
A similar argument shows the following:
\begin{lemma}
Let $S$ and $T$ be $\mathsf{ED}$ spaces.
If $S\times T$, viewed as a product of condensed anima, is compact, then $\mathbb{Z}[S\times T]$ is of finite projective dimension.
\end{lemma}
\begin{proof}
Just as above, one fixes a hypercover $R_\bullet$ of $S\times T$ by $\mathsf{ED}$ spaces, lets $X_n=\mathrm{colim}\,(R_\bullet)_{\leq n}$ for each $n\in\omega$, and deduces that $S\times T$ is a retract of some $X_n$.
It follows that $\mathbb{Z}[S\times T]$ is a retract of $\mathbb{Z}[X_n]$, and as the latter is a finite complex of projective condensed abelian groups, this concludes the proof.
\end{proof}
\section{Conclusion}
\label{sect:conclusion}
As Corollary \ref{cor:413} and Remark \ref{rmk:betternontrivs} respectively underscore, nothing in our work so far rules out either of the following possibilities:
\begin{question}
Is it consistent with the $\mathsf{ZFC}$ axioms that for all $\mathsf{ED}$ spaces $S$ and $T$ of cardinality less than $\aleph_\omega$, $S\times T$, viewed as a product of condensed anima, is compact?
\end{question}

\begin{question}
Is it consistent with the $\mathsf{ZFC}$ axioms that a product $S\times T$ as above is compact whenever one of the factors is of cardinality less than $\aleph_\omega$?
\end{question}

Left even more conspicuously open is item (D) of Section \ref{subsect:products}:

\begin{question}
\label{ques:condabprods}
Is it consistent with the $\mathsf{ZFC}$ axioms that tensor products of compact projective objects in $\mathsf{Cond(Ab)}$ are of finite projective dimension?
\end{question}

This might be loosely thought of as the derived version of the question of whether such products are projective (recall that by Proposition \ref{prop:projnotprods} the answer is no).
Just as for that question, instances of this one admit formulations in functional analytic terms (cf.\ \cite[Appendix to III]{CS3}). For example:

\begin{question}
Is it consistent with the $\mathsf{ZFC}$ axioms that either of the Banach spaces $C(\beta\mathbb{N}\times\beta\mathbb{N})$  or $c_0$ is of finite injective dimension?
\end{question}
If not, then the answer to Question \ref{ques:condabprods} is negative as well.
Worth noting is one further question in this line, listed as Question 3.6 of \cite{CS3}: \emph{Are all compact projective condensed abelian groups isomorphic to $\mathbb{Z}[S]$ for some $\mathsf{ED}$ space $S$?}

We turn next to a more overtly combinatorial series of questions, beginning with the following:
\begin{question}
Fix $n>1$. Is it a \textsf{ZFC} theorem that there exist cardinals $\kappa$ and $\lambda$ and abelian group $H$ that $\mathrm{lim}^n\,\mathbf{A}_{\kappa,\lambda}[H]\neq 0$?
\end{question}
Note that a nonvanishing $\mathrm{lim}^s\,\mathbf{A}_{\kappa,\lambda}$ would imply a negative answer to the $n=s+1$ instance of the following question, by way of the exact sequence in the proof of Corollary \ref{cor:nonadditive}.
\begin{question}
\label{ques:add_of_lims}
For any fixed $n>2$, is it consistent with the \textsf{ZFC} axioms that the functor $\mathrm{lim}^n:\mathsf{Pro(Ab)}\to\mathsf{Ab}$ is additive?
\end{question}
Even if, as suspected, the answer is no, $\mathrm{lim}^n$ might consistently be additive for collections of \textit{towers}, as discussed in Section \ref{sec:additivity}. A good test for this prospect is the following question; variations on it appear in the works \cite{Be17,Bergfalk_simultaneously_23,Bannister_additivity_23} as well.
\begin{question}
Is it consistent with the \textsf{ZFC} axioms that $\mathrm{lim}^n\,\mathbf{A}_{\kappa,\omega}=0$ for all $n>0$ and cardinals $\kappa$?
\end{question}
See \cite{Bergfalk_Casarosa} and \cite{Bannister_All} for more on this question; as shown in the latter work, the question is in fact equivalent to the first part of the next one:
\begin{question}
Is it consistent with the \textsf{ZFC} axioms that strong homology is additive on the class of all locally compact metric spaces, or even on all metric spaces outright?
\end{question}

\section*{Acknowledgements}
A portion of this work was carried out during its authors' research visits to the Fields Institute's spring 2023 Thematic Program on Set Theoretic Methods in Algebra, Dynamics and Geometry; we thank both the institute and the program's organizers for their hospitality and support. 
We reiterate our thanks to Dustin Clausen and Peter Scholze for their shaping role in this paper's results, as well as to Lucas Mann for his generosity in clarifying both the details and significance of several of its arguments, and thank both Catrin Mair and Michael Hru\v{s}\'{a}k for discussions benefiting the paper as well.
We wish lastly to thank the paper's referee for so thoroughly alert and prompt a reading of it.
\bibliographystyle{plain}
\bibliography{condensedbib}

\begin{thebibliography}{10}

\bibitem{ArtinMazur}
M.~Artin and B.~Mazur.
\newblock {\em Etale homotopy}, volume No. 100 of {\em Lecture Notes in
  Mathematics}.
\newblock Springer-Verlag, Berlin-New York, 1969.

\bibitem{SGA4}
Michael Artin, Alexander Grothendieck, and Jean-Louis Verdier.
\newblock {\em Theorie de Topos et Cohomologie Etale des Schemas {I}}, volume
  269 of {\em Lecture Notes in Mathematics}.
\newblock Springer, 1971.

\bibitem{Asgeirsson}
Dagur \'{A}sgeirsson.
\newblock The foundations of condensed mathematics.
\newblock \url{https://dagur.sites.ku.dk/condensed-foundations/}.
\newblock Accessed: 15 August 2024.

\bibitem{Aviles}
Antonio Avil\'es, F\'elix~Cabello S\'anchez, Jes\'us M.~F. Castillo, Manuel
  Gonz\'alez, and Yolanda Moreno.
\newblock {\em Separably injective {B}anach spaces}, volume 2132 of {\em
  Lecture Notes in Mathematics}.
\newblock Springer, 2016.

\bibitem{Bannister_additivity_23}
Nathaniel {Bannister}.
\newblock {Additivity of derived limits in the Cohen model}.
\newblock {\em arXiv e-prints}, page arXiv:2302.07222, February 2023.

\bibitem{Bannister_All}
Nathaniel {Bannister}.
\newblock {All you need is $\mathbf{A}_\kappa$}.
\newblock {\em arXiv e-prints}, page arXiv:2506.14185, June 2025.

\bibitem{Bannister_on_22}
Nathaniel Bannister, Jeffrey Bergfalk, and Justin~Tatch Moore.
\newblock On the additivity of strong homology for locally compact separable
  metric spaces.
\newblock {\em Israel J. Math.}, 255(1):349--381, 2023.

\bibitem{BBMT}
Nathaniel Bannister, Jeffrey Bergfalk, Justin~Tatch Moore, and Stevo
  Todorcevic.
\newblock A descriptive approach to higher derived limits.
\newblock {\em Journal of the {E}uropean {M}athematical {S}ociety}, 2024.

\bibitem{exodromy}
Clark {Barwick}, Saul {Glasman}, and Peter {Haine}.
\newblock {Exodromy}.
\newblock {\em arXiv e-prints}, page arXiv:1807.03281, July 2018.

\bibitem{MSRI}
Clark Barwick and Peter Haine.
\newblock Pyknotic / condensed seminar, {MSRI}.
\newblock \url{https://www.slmath.org/workshops/24809#overview_workshop}.
\newblock Accessed: 15 August 2024.

\bibitem{pyknotic}
Clark {Barwick} and Peter {Haine}.
\newblock {Pyknotic objects, I. Basic notions}.
\newblock {\em arXiv e-prints}, page arXiv:1904.09966, April 2019.

\bibitem{Be17}
Jeffrey Bergfalk.
\newblock Strong homology, derived limits, and set theory.
\newblock {\em Fund. Math.}, 236(1):71--82, 2017.

\bibitem{TFOA}
Jeffrey Bergfalk.
\newblock The first omega alephs: from simplices to trees of trees to higher
  walks.
\newblock {\em Adv. Math.}, 393:Paper No. 108083, 74, 2021.

\bibitem{IHW}
Jeffrey {Bergfalk}.
\newblock {An introduction to higher walks}.
\newblock {\em arXiv e-prints}, page arXiv:2410.00607, October 2024.

\bibitem{Bergfalk_Casarosa}
Jeffrey {Bergfalk} and Matteo {Casarosa}.
\newblock {Higher limits of wider systems}.
\newblock {\em arXiv e-prints}, page arXiv:2507.05471, July 2025.

\bibitem{Bergfalk_simultaneously_23}
Jeffrey Bergfalk, Michael Hru\v{s}\'{a}k, and Chris Lambie-Hanson.
\newblock Simultaneously vanishing higher derived limits without large
  cardinals.
\newblock {\em J. Math. Log.}, 23(1):Paper No. 2250019, 40, 2023.

\bibitem{BLHCoOI}
Jeffrey Bergfalk and Chris Lambie-Hanson.
\newblock The cohomology of the ordinals {I}: Basic theory and consistency
  results, 2019.

\bibitem{Bergfalk_simultaneously_21}
Jeffrey Bergfalk and Chris Lambie-Hanson.
\newblock Simultaneously vanishing higher derived limits.
\newblock {\em Forum Math. Pi}, 9:Paper No. e4, 31, 2021.

\bibitem{BLHS}
Jeffrey Bergfalk, Chris Lambie-Hanson, and Jan \v{S}aroch.
\newblock Whitehead's problem and condensed mathematics.
\newblock {\em arXiv e-prints}, page arXiv:2312.09122, 2024.

\bibitem{BoardmanVogt}
J.~M. Boardman and R.~M. Vogt.
\newblock {\em Homotopy invariant algebraic structures on topological spaces},
  volume Vol. 347 of {\em Lecture Notes in Mathematics}.
\newblock Springer-Verlag, Berlin-New York, 1973.

\bibitem{BorelMoore}
A.~Borel and J.~C. Moore.
\newblock Homology theory for locally compact spaces.
\newblock {\em Michigan Math. J.}, 7:137--159, 1960.

\bibitem{BousfieldKan}
A.~K. Bousfield and D.~M. Kan.
\newblock {\em Homotopy limits, completions and localizations}, volume Vol. 304
  of {\em Lecture Notes in Mathematics}.
\newblock Springer-Verlag, Berlin-New York, 1972.

\bibitem{casarosa_lh}
Matteo {Casarosa} and Chris {Lambie-Hanson}.
\newblock {Simultaneously nonvanishing higher derived limits}.
\newblock {\em arXiv e-prints}, page arXiv:2411.15856, November 2024.

\bibitem{Cisinskietal}
Denis-Charles Cisinski, Bastiaan Cnossen, Kim Nguyen, and Tashi Walde.
\newblock Formalization of higher categories.
\newblock
  \url{https://drive.google.com/file/d/1lKaq7watGGl3xvjqw9qHjm6SDPFJ2-0o/view},
  2024.
\newblock Book project in progress. Accessed: 16 December 2024.

\bibitem{CS3}
Dustin Clausen and Peter Scholze.
\newblock Condensed mathematics and complex geometry.
\newblock \url{https://people.mpim-bonn.mpg.de/scholze/Complex.pdf}.
\newblock Accessed: 23 January 2023.

\bibitem{Masterclass}
Dustin Clausen and Peter Scholze.
\newblock Masterclass in condensed mathematics.
\newblock
  \url{https://www.youtube.com/playlist?list=PLAMniZX5MiiLXPrD4mpZ-O9oiwhev-5Uq},
  2020.
\newblock Posted by the University of Copenhagen.

\bibitem{analytic_stacks}
Dustin Clausen and Peter Scholze.
\newblock Analytic stacks.
\newblock \url{https://youtu.be/YxSZ1mTIpaA?si=8PvFsTN6GSKkWaWy}, 2023.
\newblock Posted by Institut des Hautes \'{E}tudes Scientifiques (IH\'{E}S).

\bibitem{Conrad}
Brian Conrad.
\newblock Cohomological descent.
\newblock \url{https://math.stanford.edu/~conrad/papers/hypercover.pdf}, 2003.

\bibitem{CordierPorter}
J.-M. Cordier and T.~Porter.
\newblock {\em Shape theory: {C}ategorical methods of approximation}.
\newblock Ellis Horwood Series: Mathematics and its Applications. Ellis Horwood
  Ltd., Chichester; Halsted Press [John Wiley \& Sons, Inc.], New York, 1989.

\bibitem{CordierCoherent}
Jean-Marc Cordier.
\newblock Sur la notion de diagramme homotopiquement coh\'erent.
\newblock {\em Cahiers Topologie G\'eom. Diff\'erentielle}, 23(1):93--112,
  1982.
\newblock Third Colloquium on Categories, Part VI (Amiens, 1980).

\bibitem{CordierStrong}
Jean-Marc Cordier.
\newblock Homologie de {S}teenrod-{S}itnikov et limite homotopique
  alg\'ebrique.
\newblock {\em Manuscripta Math.}, 59(1):35--52, 1987.

\bibitem{Curtis}
Philip~C. Curtis, Jr.
\newblock A note concerning certain product spaces.
\newblock {\em Arch. Math.}, 11:50--52, 1960.

\bibitem{DuggerIsak}
Daniel Dugger, Sharon Hollander, and Daniel~C. Isaksen.
\newblock Hypercovers and simplicial presheaves.
\newblock {\em Math. Proc. Cambridge Philos. Soc.}, 136(1):9--51, 2004.

\bibitem{DuggerIsakRealizations}
Daniel Dugger and Daniel~C. Isaksen.
\newblock Topological hypercovers and {$\Bbb A^1$}-realizations.
\newblock {\em Math. Z.}, 246(4):667--689, 2004.

\bibitem{EdwardsHastings}
David~A. Edwards and Harold~M. Hastings.
\newblock {\em \v Cech and {S}teenrod homotopy theories with applications to
  geometric topology}, volume Vol. 542 of {\em Lecture Notes in Mathematics}.
\newblock Springer-Verlag, Berlin-New York, 1976.

\bibitem{EilenbergSteenrod}
Samuel Eilenberg and Norman Steenrod.
\newblock {\em Foundations of algebraic topology}.
\newblock Princeton University Press, Princeton, NJ, 1952.

\bibitem{Eisworth}
Todd Eisworth.
\newblock Successors of singular cardinals.
\newblock In {\em Handbook of set theory. {V}ols. 1, 2, 3}, pages 1229--1350.
  Springer, Dordrecht, 2010.

\bibitem{Fargues_Scholze_Geometrization}
Laurent Fargues and Peter Scholze.
\newblock Geometrization of the local langlands correspondence.
\newblock \url{https://people.mpim-bonn.mpg.de/scholze/Geometrization.pdf}.
\newblock Accessed: 18 July 2025.

\bibitem{Fuchs_Infinite_73}
L\'{a}szl\'{o} Fuchs.
\newblock {\em Infinite abelian groups. {V}ol. {II}}.
\newblock Pure and Applied Mathematics. Vol. 36-II. Academic Press, New
  York-London, 1973.

\bibitem{Gleason}
Andrew~M. Gleason.
\newblock Projective topological spaces.
\newblock {\em Illinois J. Math.}, 2:482--489, 1958.

\bibitem{Goblot}
R\'{e}mi Goblot.
\newblock Sur les d\'{e}riv\'{e}s de certaines limites projectives.
  {A}pplications aux modules.
\newblock {\em Bull. Sci. Math. (2)}, 94:251--255, 1970.

\bibitem{JHop}
Rok Gregoric.
\newblock Condensed mathematics, a spring 2024 {J}ohns {H}opkins seminar.
\newblock \url{https://sites.google.com/view/rokgregoric/seminars}.
\newblock Accessed: 15 August 2024.

\bibitem{GuntherSemi}
Bernd G\"unther.
\newblock The use of semisimplicial complexes in strong shape theory.
\newblock {\em Glas. Mat. Ser. III}, 27(47)(1):101--144, 1992.

\bibitem{Gunther}
Bernd G\"unther.
\newblock The {V}ietoris system in strong shape and strong homology.
\newblock {\em Fund. Math.}, 141(2):147--168, 1992.

\bibitem{Halmos}
Paul~R. Halmos.
\newblock {\em Lectures on {B}oolean algebras}, volume No. 1 of {\em Van
  Nostrand Mathematical Studies}.
\newblock D. Van Nostrand Co., Inc., Princeton, NJ, 1963.

\bibitem{Mann}
Claudius {Heyer} and Lucas {Mann}.
\newblock {6-Functor Formalisms and Smooth Representations}.
\newblock {\em arXiv e-prints}, page arXiv:2410.13038, October 2024.

\bibitem{Hoyois}
Marc Hoyois.
\newblock Higher {G}alois theory.
\newblock {\em J. Pure Appl. Algebra}, 222(7):1859--1877, 2018.

\bibitem{Iversen}
Birger Iversen.
\newblock {\em Cohomology of sheaves}.
\newblock Universitext. Springer-Verlag, Berlin, 1986.

\bibitem{Jensen_les}
C.~U. Jensen.
\newblock {\em Les foncteurs d\'{e}riv\'{e}s de {$\varprojlim$} et leurs
  applications en th\'{e}orie des modules}.
\newblock Lecture Notes in Mathematics, Vol. 254. Springer-Verlag, Berlin-New
  York, 1972.

\bibitem{jensen_constructible}
R.~Bj\"orn Jensen.
\newblock The fine structure of the constructible hierarchy.
\newblock {\em Ann. Math. Logic}, 4:229--308; erratum, ibid. {4 (1972), 443},
  1972.
\newblock With a section by Jack Silver.

\bibitem{Johnstone}
Peter~T. Johnstone.
\newblock {\em Stone spaces}, volume~3 of {\em Cambridge Studies in Advanced
  Mathematics}.
\newblock Cambridge University Press, Cambridge, 1982.

\bibitem{Kashiwara_Categories_06}
Masaki Kashiwara and Pierre Schapira.
\newblock {\em Categories and sheaves}, volume 332 of {\em Grundlehren der
  mathematischen Wissenschaften [Fundamental Principles of Mathematical
  Sciences]}.
\newblock Springer-Verlag, Berlin, 2006.

\bibitem{LurieHTT}
Jacob Lurie.
\newblock {\em Higher topos theory}, volume 170 of {\em Annals of Mathematics
  Studies}.
\newblock Princeton University Press, Princeton, NJ, 2009.

\bibitem{HA}
Jacob Lurie.
\newblock Higher {A}lgebra, 2017.

\bibitem{kerodon}
Jacob Lurie.
\newblock Kerodon.
\newblock \url{https://kerodon.net}, 2018.

\bibitem{Mair}
Catrin {Mair}.
\newblock {Animated Condensed Sets and Their Homotopy Groups}.
\newblock {\em arXiv e-prints}, page arXiv:2105.07888, May 2021.

\bibitem{Mardesic_additive_88}
S.~Marde\v{s}i\'{c} and A.~V. Prasolov.
\newblock Strong homology is not additive.
\newblock {\em Trans. Amer. Math. Soc.}, 307(2):725--744, 1988.

\bibitem{Mardesic_nonvanishing_96}
Sibe Marde\v{s}i\'{c}.
\newblock Nonvanishing derived limits in shape theory.
\newblock {\em Topology}, 35(2):521--532, 1996.

\bibitem{Mardesic_strong_00}
Sibe Marde\v{s}i\'{c}.
\newblock {\em Strong shape and homology}.
\newblock Springer Monographs in Mathematics. Springer-Verlag, Berlin, 2000.

\bibitem{ShapeTheory}
Sibe Marde\v{s}i\'{c} and Jack Segal.
\newblock {\em Shape theory: {T}he inverse system approach}, volume~26 of {\em
  North-Holland Mathematical Library}.
\newblock North-Holland Publishing Co., Amsterdam-New York, 1982.

\bibitem{Aparicio_condensed_21}
Rodrigo {Marlasca Aparicio}.
\newblock {Condensed Mathematics: The internal Hom of condensed sets and
  condensed abelian groups and a prismatic construction of the real numbers}.
\newblock {\em arXiv e-prints}, page arXiv:2109.07816, September 2021.

\bibitem{Massey78}
William~S. Massey.
\newblock {\em Homology and cohomology theory}, volume~46 of {\em Monographs
  and Textbooks in Pure and Applied Mathematics}.
\newblock Marcel Dekker, Inc., New York-Basel, 1978.

\bibitem{UChic}
Akhil Mathew and Matthew Emerton.
\newblock Condensed mathematics, a fall 2022 {U}niversity of {C}hicago course.
\newblock \url{http://math.uchicago.edu/~amathew/condensed22.html}.
\newblock Accessed: 15 August 2024.

\bibitem{MayPonto}
J.~P. May and K.~Ponto.
\newblock {\em More concise algebraic topology}.
\newblock Chicago Lectures in Mathematics. University of Chicago Press,
  Chicago, IL, 2012.
\newblock Localization, completion, and model categories.

\bibitem{Melikhov}
Sergey~A. {Melikhov}.
\newblock {Fine shape I}.
\newblock {\em arXiv e-prints}, page arXiv:1808.10228, August 2018.

\bibitem{MilneLEC}
James~S. Milne.
\newblock Lectures on etale cohomology (v2.21), 2013.
\newblock Available at www.jmilne.org/math/.

\bibitem{Milnor1960}
John Milnor.
\newblock On the {S}teenrod homology theory.
\newblock In {\em Novikov conjectures, index theorems and rigidity, {V}ol. 1
  ({O}berwolfach, 1993)}, volume 226 of {\em London Math. Soc. Lecture Note
  Ser.}, pages 79--96. Cambridge Univ. Press, Cambridge, 1995.

\bibitem{Mitchell_approachability}
William~J. Mitchell.
\newblock {$I[\omega_2]$} can be the nonstationary ideal on {${\rm
  Cof}(\omega_1)$}.
\newblock {\em Trans. Amer. Math. Soc.}, 361(2):561--601, 2009.

\bibitem{PorterTwo}
Timothy Porter.
\newblock On the two definitions of {${\rm Ho}({\rm pro}\,C)$}.
\newblock {\em Topology Appl.}, 28(3):289--293, 1988.

\bibitem{Prasolov_non_05}
Andrei~V. Prasolov.
\newblock Non-additivity of strong homology.
\newblock {\em Topology Appl.}, 153(2-3):493--527, 2005.

\bibitem{Prosmans}
Fabienne Prosmans.
\newblock Derived limits in quasi-abelian categories.
\newblock {\em Bull. Soc. Roy. Sci. Li\`ege}, 68(5-6):335--401, 1999.

\bibitem{Rainwater}
John Rainwater.
\newblock A note on projective resolutions.
\newblock {\em Proc. Amer. Math. Soc.}, 10:734--735, 1959.

\bibitem{Rinot_diamond}
Assaf Rinot.
\newblock Jensen's diamond principle and its relatives.
\newblock In {\em Set theory and its applications}, volume 533 of {\em Contemp.
  Math.}, pages 125--156. Amer. Math. Soc., Providence, RI, 2011.

\bibitem{rj_rc}
Joaqu\'in Rodrigues~Jacinto and Juan~Esteban Rodr\'iguez~Camargo.
\newblock Solid locally analytic representations of {$p$}-adic {L}ie groups.
\newblock {\em Represent. Theory}, 26:962--1024, 2022.

\bibitem{Columb}
Juan Rodriguez-Camargo and John Morgan.
\newblock Condensed mathematics, a fall 2023 {C}olumbia seminar.
\newblock
  \url{https://www.math.columbia.edu/~jmorgan/condensed_mathematics.html}.
\newblock Accessed: 15 August 2024.

\bibitem{CS2}
Peter Scholze.
\newblock Lectures on analytic geometry (all results joint with {D}ustin
  {C}lausen).
\newblock \url{https://people.mpim-bonn.mpg.de/scholze/Analytic.pdf}.
\newblock Accessed: 23 January 2023.

\bibitem{CS1}
Peter Scholze.
\newblock Lectures on condensed mathematics (all results joint with {D}ustin
  {C}lausen).
\newblock \url{https://www.math.uni-bonn.de/people/scholze/Condensed.pdf}.
\newblock Accessed: 23 January 2023.

\bibitem{ScholzePersComm}
Peter Scholze.
\newblock Personal communication, 2019--2024.

\bibitem{Sklyarenko}
E.~G. Sklyarenko.
\newblock Hyper(co)homology for left-exact covariant functors, and homology
  theory of topological spaces.
\newblock {\em Uspekhi Mat. Nauk}, 50(3(303)):109--146, 1995.

\bibitem{smith}
Marianne~Freundlich Smith.
\newblock The {P}ontrjagin duality theorem in linear spaces.
\newblock {\em Ann. of Math. (2)}, 56:248--253, 1952.

\bibitem{SteenrodCycles}
N.~E. Steenrod.
\newblock Regular cycles of compact metric spaces.
\newblock {\em Ann. of Math. (2)}, 41:833--851, 1940.

\bibitem{TangProfinite}
Jiacheng Tang.
\newblock Profinite and solid cohomology.
\newblock {\em arXiv e-prints}, page arXiv:2410.08933, 2024.

\bibitem{Todorcevic_Walks_07}
Stevo Todorcevic.
\newblock {\em Walks on ordinals and their characteristics}, volume 263 of {\em
  Progress in Mathematics}.
\newblock Birkh\"{a}user Verlag, Basel, 2007.

\bibitem{CesnaviciusScholze}
Kestutis \v{C}esnavi\v{c}ius and Peter Scholze.
\newblock Purity for flat cohomology.
\newblock {\em Ann. of Math. (2)}, 199(1):51--180, 2024.

\bibitem{Velickovic_non_21}
Boban Veli\v{c}kovi\'{c} and Alessandro Vignati.
\newblock Non-vanishing higher derived limits.
\newblock {\em Commun. Contemp. Math.}, 26(7):Paper No. 2350031, 22, 2024.

\bibitem{Weibel}
Charles~A. Weibel.
\newblock {\em An introduction to homological algebra}, volume~38 of {\em
  Cambridge Studies in Advanced Mathematics}.
\newblock Cambridge University Press, Cambridge, 1994.

\bibitem{Wiegand}
Roger Wiegand.
\newblock Sheaf cohomology of locally compact totally disconnected spaces.
\newblock {\em Proc. Amer. Math. Soc.}, 20:533--538, 1969.

\end{thebibliography}
\end{document}